\documentclass[reqno]{amsart}
\usepackage{amsmath,amsthm,amssymb,latexsym,fullpage,setspace,graphicx,float,xcolor,hyperref,verbatim,mathtools}
\usepackage{tikz-cd,tikz,pgfplots}
\usepackage[mode=buildnew]{standalone}
\usetikzlibrary{decorations.markings,math,positioning,graphs}
\usepackage[utf8]{inputenc}
\usepackage[english]{babel}

\DeclareFontFamily{U}{mathx}{\hyphenchar\font45}
\DeclareFontShape{U}{mathx}{m}{n}{
      <5> <6> <7> <8> <9> <10>
      <10.95> <12> <14.4> <17.28> <20.74> <24.88>
      mathx10
      }{}
\DeclareSymbolFont{mathx}{U}{mathx}{m}{n}
\DeclareFontSubstitution{U}{mathx}{m}{n}
\DeclareMathAccent{\widecheck}{0}{mathx}{"71}
\DeclareMathAccent{\wideparen}{0}{mathx}{"75}

\theoremstyle{plain}
\newtheorem{thm}{Theorem}[section]
\newtheorem{cor}[thm]{Corollary}
\newtheorem{lem}[thm]{Lemma}

\newtheorem{prop}[thm]{Proposition}

\theoremstyle{definition}
\newtheorem{defn}[thm]{Definition}
\newtheorem{eg}[thm]{Example}

\theoremstyle{remark}
\newtheorem{remark}[thm]{Remark}

\makeatletter
\@namedef{subjclassname@2020}{%
  \textup{2020} Mathematics Subject Classification}
\makeatother

\title{Inducing contractions of the mother of all continued fractions}
\author{Karma Dajani}
\address{Department of Mathematics, Utrecht University, P.O.~Box 80010, 3508 TA Utrecht, The Netherlands}
\email{k.dajani1@uu.nl}
\author{Cor Kraaikamp}
\address{Delft University of Technology, EWI (DIAM), Mekelweg 4, 2628 CD Delft, The
Netherlands}
\email{c.kraaikamp@tudelft.nl}
\author{Slade Sanderson}
\address{Department of Mathematics, Utrecht University, P.O.~Box 80010, 3508 TA Utrecht, The Netherlands}
\email{s.b.sanderson@uu.nl}
\date{\today}

\def\convmat#1{
\begin{pmatrix}
u_{#1} & t_{#1}\\
s_{#1} & r_{#1}
\end{pmatrix}
}

\def\F{
\mathcal{F}
}

\def\G{
\mathcal{G}
}

\begin{document}

\begin{abstract}
We introduce a new, large class of continued fraction algorithms producing what are called \emph{contracted Farey expansions}.  These algorithms are defined by coupling two acceleration techniques---\emph{induced transformations} and \emph{contraction}---in the setting of Shunji Ito's (\cite{I1989}) natural extension of the Farey tent map, which generates `slow' continued fraction expansions.  In addition to defining new algorithms, we also realise several existing continued fraction algorithms in our unifying setting.  In particular, we find regular continued fractions, the second-named author's $S$-expansions, and Nakada's parameterised family of $\alpha$-continued fractions for all $0<\alpha\le 1$ as examples of contracted Farey expansions.  Moreover, we give a new description of a planar natural extension for each of the $\alpha$-continued fraction transformations as an explicit induced transformation of Ito's natural extension.  
\end{abstract}

\maketitle
\tikzset{->-/.style={decoration={markings,mark=at position #1 with {\arrow{>}}},postaction={decorate}}}

\section{Introduction}\label{Introduction-ch1}

In 1855, Seidel (\cite{S1855}) introduced a seemingly overlooked,\footnote{Contraction \emph{is} used in the analytic theory of continued fractions, but usually only for subsequences of odd or even integers (\cite{LW2008}).  See also \cite{B2014}, where the more general contraction procedure is used on the continued fraction expansion of the golden mean, $(\sqrt{5}+1)/2$.} arithmetic procedure, called \emph{contraction}, which---under mild assumptions---allows one to produce from a given generalised continued fraction ({\sc gcf}) a new {\sc gcf} whose convergents are any prescribed subsequence of the original {\sc gcf}-convergents.  Nearly ninety years later, in 1943, Kakutani introduced in \cite{K1943} \emph{induced transformations}, which accelerate a given dynamical system by only observing the dynamics within a subregion of the domain.  In 1989, Shunji Ito (\cite{I1989}) gave an explicit natural extension of what has been called\footnote{This is true `up to isomorphism.'  The maternal moniker was originally applied to the \emph{Lehner map}, which is isomorphic to the Farey tent map (\cite{DK2000}); see also \S\ref{Farey expansions and Farey convergents} below.} `the mother of all continued fractions'---the \emph{Farey tent map}---which generates `slow' continued fraction expansions (\emph{Farey expansions}) whose convergents (\emph{Farey convergents}) consist of all regular continued fraction ({\sc rcf}) convergents and so-called \emph{mediant convergents} (see \S\ref{Regular continued fractions} below).  In this article, we obtain a broad, unifying theory for various continued fraction expansions by `inducing contractions of the mother of all continued fractions.'

More formally, we use induced transformations of Ito's natural extension to govern contractions of Farey expansions.  This coupling of inducing and contracting defines a large class of continued fraction algorithms---producing what we call \emph{contracted Farey expansions}---which are parameterised by measurable subregions of the domain of Ito's natural extension.  Within this collection of algorithms we find several well-studied examples.  In particular, contracted Farey expansions contain the theory of the second-named author's $S$-expansions, which themselves contain the theory of {\sc rcf}s, Minkowski's diagonal continued fractions, Bosma's optimal continued fractions and more (\cite{K1991}).  The collection of $S$-expansions also partially contains Nakada's parameterised family of $\alpha$-continued fractions: this latter family is defined for $0\le \alpha\le 1$, but only those for which $\alpha\ge 1/2$ are realised as $S$-expansions.  Our theory of contracted Farey expansions contains Nakada's $\alpha$-continued fractions for all $0<\alpha\le 1$---thus providing a unifying framework within which to view these two partially overlapping families---and gives a new description of the natural extension of each of the $\alpha$-continued fraction transformations as an induced transformation of Ito's natural extension (cf.~\cite{KSS2012}).  

In \cite{DKS2023}, the authors use a one-to-one correspondence between certain forward orbits determined by irrationals $x\in(0,1)$ under Ito's natural extension map and the sequence of all Farey convergents ({\sc rcf}-convergents and mediants) of $x$.  With this correspondence, certain subregions of the domain of Ito's natural extension `announce' certain types of Farey convergents.  By considering induced transformations on these subregions, the authors obtain unified and simple proofs of results from, e.g., \cite{B1990,BY1996,I1989,J91}, old and classical results of Legendre and Koksma, and various new results such as generalisations of L\'evy's constant and of the Doeblin--Lenstra conjecture to subsequences of {\sc rcf}-convergents and mediants.

The subsequences from \cite{DKS2023} of Farey convergents announced by a subregion of the domain of Ito's natural extension are also of central importance in the current article: via contraction, these subsequences form the convergents of our new contracted Farey expansions.  That is, we fix a subregion $R$ of the domain $\Omega$ of Ito's natural extension and consider the subsequence of the forward orbit of a point $(x,y)\in \Omega$ which enters $R$ under the natural extension map.  Via the aforementioned one-to-one correspondence between orbits and Farey convergents, we obtain a subsequence of Farey convergents of $x$ and use contraction to produce a new {\sc gcf}-expansion of $x$ whose convergents are precisely this subsequence.  The digits of these new {\sc gcf}-expansions may be described in terms of the dynamics of the induced transformation of Ito's natural extension on the subregion $R$, and hence we obtain a large collection of continued fraction algorithms parameterised by these subregions.  

While the present article is informed by \cite{DKS2023}, these two works may be read independently.  We remark, however, that the ideas of both articles may also be combined: in \cite{Sip}, the third-named author exploits results of \cite{DKS2023} and the present article to generate new, \emph{superoptimal continued fraction} algorithms producing {\sc gcf}-expansions which have arbitrarily good approximation properties and converge arbitrarily fast.

This article is organised as follows.  In \S\ref{Generalised, semi-regular and regular continued fractions}, we set definitions and notation for generalised, semi-regular and regular continued fractions that are used throughout, and in \S\ref{Some continued fraction algorithms} we recall several continued fraction algorithms: the Gauss map and its natural extension, Nakada's $\alpha$-continued fractions and the second-named author's $S$-expansions.  We recall the Farey tent map, Farey expansions and Farey convergents in \S\ref{The Farey tent map and Farey expansions}.  In \S\ref{Ito's natural extension of the Farey tent map} we describe Ito's natural extension of the Farey tent map, and, moreover, define and set notation for induced transformations of it (\S\ref{Inducing Ito's natural extension}).  In \S\ref{Contraction} we describe contraction in the abstract setting of generalised continued fractions and in \S\ref{Contracted Farey expansions} use induced transformations of Ito's natural extension to govern contractions of Farey expansions.  Furthermore, in \S\ref{A two-sided shift for contracted Farey expansions} we define a dynamical system which acts essentially as a two-sided shift on contracted Farey expansions.  Section \ref{Examples of contracted Farey expansions} realises each of the examples from \S\ref{Some continued fraction algorithms} within our theory of contracted Farey expansions.

\bigskip

{\flushleft \textbf{Acknowledgments.}} This work is part of project number 613.009.135 of the research programme Mathematics Clusters which is financed by the Dutch Research Council (NWO).

\bigskip

\section{Generalised, semi-regular and regular continued fractions}\label{Generalised, semi-regular and regular continued fractions}

\subsection{Generalised continued fractions}\label{Generalised continued fractions}

A \textit{generalised continued fraction} ({\sc gcf}) is a formal (infinite or finite) expression of the form
\begin{equation}\label{gcf}
[\beta_0/\alpha_0;\alpha_1/\beta_1,\alpha_2/\beta_2,\dots]=\cfrac{\alpha_{-1}}{\beta_{-1}+\cfrac{\alpha_0}{\beta_0+\cfrac{\alpha_1}{\beta_1+\cfrac{\alpha_2}{\beta_2+\ddots}}}}\ ,
\end{equation}
where $(\alpha_{-1},\beta_{-1}):=(1,0)$ and for $n\ge 0$, $\alpha_n,\beta_n\in \mathbb{C}$ with $\alpha_n\neq 0$.  

\begin{remark}
Notice that for $\alpha_0,\beta_0,x\in\mathbb{C}$ with $\alpha_0$ nonzero,
\[\cfrac{1}{0+\cfrac{\alpha_0}{\beta_0+x}}=\frac1{\alpha_0}(\beta_0+x),\]
with the convention that $c/0=\infty$ and $c/\infty=0$ for $c\in\mathbb{C}\setminus\{0\}$.  Thus---although at this point it is a strictly formal expression---a {\sc gcf} should be thought of as $1/\alpha_0$ `multiplied' with the expression 
\[\beta_0+\cfrac{\alpha_1}{\beta_1+\cfrac{\alpha_2}{\beta_2+\ddots}}\] 
(hence the choice of notation $[\beta_0/\alpha_0;\alpha_1/\beta_1,\dots]$ rather than $[\alpha_0/\beta_0;\alpha_1/\beta_1,\dots]$).  Besides allowing for this inversion of $\alpha_0$, our inclusion of $\alpha_{-1}$ and $\beta_{-1}$ in \eqref{gcf} also prevents us from needing to treat the index $0$ as a special case in the matrix notation introduced below.  
\end{remark}

The digits $\alpha_n$ and $\beta_n$ are called \emph{partial numerators} and \emph{partial denominators}, respectively.  When a {\sc gcf} has only finitely many partial numerators and partial denominators, the expression on the right-hand side of \eqref{gcf} may be evaluated to give a number in $\widehat{\mathbb{C}}:=\mathbb{C}\cup\{\infty\}$.  Define for each integer $n\ge -2$ (with the obvious restriction in the finite case) the \textit{$n^{\text{th}}$ tail} of $[\beta_0/\alpha_0;\alpha_1/\beta_1,\alpha_2/\beta_2,\dots]$ to be the {\sc gcf}
\[[0/1;\alpha_{n+1}/\beta_{n+1},\alpha_{n+2}/\beta_{n+2},\alpha_{n+3}/\beta_{n+3},\dots].\]
For each integer $n\ge -1$, set
\[B_n=B_n([\beta_0/\alpha_0;\alpha_1/\beta_1,\dots]):=\begin{pmatrix}0 & \alpha_n\\ 1 & \beta_n\end{pmatrix},\] 
and for integers $-1\le m\le n$, let
\[B_{[m,n]}=B_{[m,n]}([\beta_0/\alpha_0;\alpha_1/\beta_1,\dots]):=B_mB_{m+1}\cdots B_n.\]
Notice that $\det B_{[m,n]}=(-1)^{n-m+1}\alpha_m\alpha_{m+1}\dots\alpha_n\neq 0$.  For a matrix $A=\left(\begin{smallmatrix}a & b\\ c & d\end{smallmatrix}\right)\in\mathrm{GL}_2 \mathbb{C}$, denote by $A\cdot z:=\frac{az+b}{cz+d},\ z\in\widehat{\mathbb{C}},$ the action of $A$ as a M\"obius transformation.  (We remark that for any $r\in\mathbb{C}\backslash\{0\},$ $(rA)\cdot z=A\cdot z$; this fact will be used repeatedly throughout.) Writing the entries of $B_{[m,n]}$ as $\left(\begin{smallmatrix}R_{[m,n]} & P_{[m,n]}\\ S_{[m,n]} & Q_{[m,n]}\end{smallmatrix}\right)$, we have
\[\frac{P_{[m,n]}}{Q_{[m,n]}}=B_{[m,n]}\cdot 0=\cfrac{\alpha_m}{\beta_m+\cfrac{\alpha_{m+1}}{\beta_{m+1}+\cfrac{\alpha_{m+2}}{\ddots+\cfrac{\alpha_n}{\beta_n}}}}=[0/1;\alpha_m/\beta_m,\alpha_{m+1}/\beta_{m+1},\dots,\alpha_n/\beta_n]\in\widehat{\mathbb{C}}.\]
(Notice that if each $\alpha_j,\beta_j\in\mathbb{Z}$, then $P_{[m,n]}/Q_{[m,n]}\in \mathbb{Q}\cup\{\infty\}$, but in general, $\gcd(P_{[m,n]},Q_{[m,n]})\neq 1$.)  When $m=-1$, we use the suppressed notation 
\begin{equation*}
\begin{pmatrix}R_n & P_n\\ S_n & Q_n\end{pmatrix}:=\begin{pmatrix}R_{[-1,n]} & P_{[-1,n]}\\ S_{[-1,n]} & Q_{[-1,n]}\end{pmatrix}=B_{[-1,n]}
\end{equation*}
and call $P_n/Q_n$ the \textit{$n^\text{th}$ convergent} of\footnote{Note that $[0/1;\alpha_{-1}/\beta_{-1},\alpha_0/\beta_0,\dots,\alpha_n/\beta_n]=[\beta_{0}/\alpha_{0};\alpha_1/\beta_1,\dots,\alpha_n/\beta_n]$.} $[\beta_0/\alpha_0;\alpha_1/\beta_1,\alpha_2/\beta_2,\dots]$.  If a {\sc gcf} is finite and evaluates to $x\in\mathbb{C}$, or if it is infinite and $x=\lim_{n\to\infty}P_n/Q_n\in\mathbb{C}$ exists, we call $[\beta_0/\alpha_0;\alpha_1/\beta_1,\alpha_2/\beta_2,\dots]$ a \textit{{\sc gcf}-expansion} of $x$, write $x=[\beta_0/\alpha_0;\alpha_1/\beta_1,\alpha_2/\beta_2,\dots]$ and---when the expansion $[\beta_0/\alpha_0;\alpha_1/\beta_1,\alpha_2/\beta_2,\dots]$ is understood---refer to the convergents $P_n/Q_n$ as convergents of $x$.  

Notice for any integers $-1\le m\le n$ that
\begin{align*}
\begin{pmatrix}R_{[m,n+1]} & P_{[m,n+1]}\\ S_{[m,n+1]} & Q_{[m,n+1]}\end{pmatrix}=&B_{[m,n]}B_{n+1}=\begin{pmatrix}R_{[m,n]} & P_{[m,n]}\\ S_{[m,n]} & Q_{[m,n]}\end{pmatrix}\begin{pmatrix}0 & \alpha_{n+1}\\ 1 & \beta_{n+1}\end{pmatrix}\\
=&\begin{pmatrix}P_{[m,n]} & \beta_{n+1}P_{[m,n]}+\alpha_{n+1}R_{[m,n]}\\ Q_{[m,n]} & \beta_{n+1}Q_{[m,n]}+\alpha_{n+1}S_{[m,n]}\end{pmatrix}.
\end{align*}
In particular, $R_{[m,n+1]}=P_{[m,n]}$ and $S_{[m,n+1]}=Q_{[m,n]}$.  Setting $(P_{[m,m-1]},Q_{[m,m-1]}):=(0,1)$ for all $m\ge -1$, this gives
\[B_{[m,n]}=\begin{pmatrix}P_{[m,n-1]} & P_{[m,n]}\\ Q_{[m,n-1]} & Q_{[m,n]}\end{pmatrix}, \quad -1\le m\le n,\]
and we obtain the following recurrence relations for all $-1\le m \le n$:
\begin{alignat}{3}\label{rec_rels}
P_{[m,n+1]}&=\beta_{n+1}P_{[m,n]}+\alpha_{n+1}P_{[m,n-1]}, &\qquad P_{[m,m-1]}&=0,\ P_{[m,m]}=\alpha_m,\\
Q_{[m,n+1]}&=\beta_{n+1}Q_{[m,n]}+\alpha_{n+1}Q_{[m,n-1]}, &\qquad Q_{[m,m-1]}&=1,\ Q_{[m,m]}=\beta_m.\notag
\end{alignat}
Let $(P_{-2},Q_{-2}):=(0,1)$.  When $m=-1$, the observations above give, for $n\ge -1$,
\[B_{[-1,n]}=\begin{pmatrix}P_{n-1} & P_{n}\\ Q_{n-1} & Q_{n}\end{pmatrix}\]
and the recurrence relations
\begin{alignat}{3}\label{rec_rels_classic}
P_{n+1}&=\beta_{n+1}P_n+\alpha_{n+1}P_{n-1}, &\qquad P_{-2}&=0,\ P_{-1}=1,\\
Q_{n+1}&=\beta_{n+1}Q_n+\alpha_{n+1}Q_{n-1}, &\qquad Q_{-2}&=1,\ Q_{-1}=0\notag.
\end{alignat}

\begin{remark}\label{convs_determin_digits}
The quantities $P_n, Q_n$ are defined in terms of the partial numerators and partial denominators $\alpha_n,\beta_n$ of a {\sc gcf}.  Conversely, since $\det(B_{[-1,n]})\neq 0$, the digits $\alpha_n, \beta_n$ are also determined by the quantities $P_n, Q_n$.  In particular, the recurrence relations (\ref{rec_rels_classic}) imply
\[\begin{pmatrix}
\alpha_{n+1}\\
\beta_{n+1}
\end{pmatrix}=
B_{[-1,n]}^{-1}
\begin{pmatrix}
P_{n+1}\\
Q_{n+1}
\end{pmatrix},
\qquad n\ge -1.\]
\end{remark}

\begin{remark}
It shall sometimes be useful to allow for infinite partial denominators $\beta_{n}=\infty$ for some $n\ge 1$ in a {\sc gcf} $[\beta_0/\alpha_0;\alpha_1/\beta_1,\alpha_2/\beta_2,\dots]$.  In this case, letting $n_0\ge 0$ denote smallest index for which $\beta_{n_0+1}=\infty$, the {\sc gcf} $[\beta_0/\alpha_0;\alpha_1/\beta_1,\alpha_2/\beta_2,\dots]$ is interpreted to be the finite {\sc gcf} $[\beta_0/\alpha_0;\alpha_1/\beta_1,\alpha_2/\beta_2,\dots,\alpha_{n_0}/\beta_{n_0}]$.
\end{remark}

Letting $T_n:=[0/1;\alpha_{n+1}/\beta_{n+1},\alpha_{n+2}/\beta_{n+2},\dots],\ n\ge 0$, denote the $n^\text{th}$ tail of the {\sc gcf}-expansion $x=[\beta_0/\alpha_0;\alpha_1/\beta_1,\alpha_2/\beta_2,\dots]$, one obtains 
\begin{equation}\label{x=B*tail}
x=\cfrac{\alpha_{-1}}{\beta_{-1}+\cfrac{\alpha_0}{\beta_0+\cfrac{\alpha_1}{\ddots+\cfrac{\alpha_n}{\beta_n+T_n}}}}
=\begin{pmatrix}0 & \alpha_{-1}\\ 1 & \beta_{-1}\end{pmatrix}
\begin{pmatrix}0 & \alpha_0\\ 1 & \beta_0\end{pmatrix}
\cdots 
\begin{pmatrix}0 & \alpha_n\\ 1 & \beta_n\end{pmatrix}\cdot T_n
=B_{[-1,n]}\cdot T_n.
\end{equation}
Notice also that for any $z\in\widehat{\mathbb{C}}$, 
\begin{equation}\label{B^T*z_first}
B_{[-1,n]}^T\cdot z
=\begin{pmatrix}0 & 1\\ \alpha_n & \beta_n\end{pmatrix}
\cdots
\begin{pmatrix}0 & 1\\ \alpha_0 & \beta_0\end{pmatrix}
\begin{pmatrix}0 & 1\\ 1 & 0\end{pmatrix}\cdot z
=\cfrac{1}{\beta_n+\cfrac{\alpha_n}{\beta_{n-1}+\cfrac{\alpha_{n-1}}{\ddots+\cfrac{\alpha_1}{\beta_0+\cfrac{\alpha_0}{z}}}}},
\end{equation}
or
\begin{equation}\label{B^T*z}
B_{[-1,n]}^T\cdot z=[0/1;1/\beta_n,\alpha_n/\beta_{n-1},\dots,\alpha_1/\beta_0,\alpha_0/z],
\end{equation}
where, in the case that $z=\infty$, the right-hand side is interpreted as $[0/1;1/\beta_n,\alpha_n/\beta_{n-1},\dots,\alpha_1/\beta_0]$.

\subsection{Semi-regular continued fractions}\label{Semi-regular continued fractions}

A \emph{semi-regular continued fraction} ({\sc srcf}) is a {\sc gcf} as in \eqref{gcf} with integral partial numerators and partial denominators $\alpha_n,\beta_n\in\mathbb{Z}$ satisfying
\begin{enumerate}
\item[(i)] $\alpha_0=1$ and $\alpha_n=\pm 1$ for $n\ge 1$,
\item[(ii)] $\beta_n>0$ for $n\ge 1$, and
\item[(iii)] $\alpha_{n+1}+\beta_n\ge 1$ for $n\ge 1$.
\end{enumerate}
If there are infinitely many digits, we further require
\begin{enumerate}
\item[(iv)] $\alpha_{n+1}+\beta_n\ge 2$ infinitely often.
\end{enumerate}

By Tietze's Convergence Theorem (see, say, \cite{P1950}) the above conditions guarantee that $x=\lim_{n\to\infty}P_n/Q_n\in\mathbb{R}$ always exists, and thus we call $[\beta_0/\alpha_0;\alpha_1/\beta_1,\alpha_2/\beta_2,\dots]$ a {\sc srcf}-expansion of $x$.  Notice that the convergents $P_n/Q_n$ of any {\sc srcf}-expansion of $x\in\mathbb{R}$ are reduced since
\[|P_{n-1}Q_n-P_nQ_{n-1}|=|\det(B_{[-1,n]})|=|\alpha_{-1}\alpha_0\cdots \alpha_n|=1.\]

\subsection{Regular continued fractions}\label{Regular continued fractions}

A \textit{regular continued fraction} ({\sc rcf}) is a {\sc srcf} with partial numerators $\alpha_n=1$ for $n\ge 1$.  (Note that with this assumption on partial numerators, conditions (iii) and (iv) of {\sc srcf}s are trivially satisfied for any choice of integral partial denominators satisfying condition (ii).)  A {\sc rcf} will also be denoted by
\[[a_0;a_1,a_2,\dots]:=[a_0/1;1/a_1,1/a_2,\dots], \quad a_n\in\mathbb{Z} \quad \text{with} \quad a_n>0,\ n\neq 0.\]
For a {\sc rcf}, we use the special notation $p_n:=P_n$ and $q_n:=Q_n,\ n\ge -2$, so the recurrence relations (\ref{rec_rels_classic}) become
\begin{alignat}{3}\label{rec_rels_rcf}
p_{n+1}&=a_{n+1}p_n+p_{n-1}, &\qquad p_{-2}&=0,\ p_{-1}=1,\\
q_{n+1}&=a_{n+1}q_n+q_{n-1}, &\qquad q_{-2}&=1,\ q_{-1}=0.\notag
\end{alignat}
Since a {\sc rcf} is a {\sc srcf}, the limit $x=\lim_{n\to\infty}p_n/q_n\in\mathbb{R}$ exists for any infinite choice of $a_n,\ n\ge 0$ (this can also be proven directly; see, e.g., \cite{IK02}), and the odd- and even-indexed {\sc rcf}-convergents $(p_{2k-1}/q_{2k-1})_{k\ge 0}$ and $(p_{2k}/q_{2k})_{k\ge 0}$ form strictly decreasing and strictly increasing sequences, respectively (see, e.g., Theorem 4 of \cite{K97}).  Conversely, every real number $x$ has a {\sc rcf}-expansion.  Moreover, if $x$ is irrational, its {\sc rcf}-expansion is unique and has infinitely many partial denominators $a_n$; if $x$ is rational, it has exactly two {\sc rcf}-expansions,
\[[a_0;a_1,\dots,a_n] \quad \text{and} \quad [a_0; a_1,\dots,a_n-1,1],\]
where $a_n\ge 2$ if $n\ge 1$ (\cite{IK02}).

\subsubsection{Mediant convergents}

The fractions 
\begin{equation}\label{mediants}
\frac{\lambda p_n+p_{n-1}}{\lambda q_n+q_{n-1}} \quad \text{for}\ \lambda\in\mathbb{N},\ 1\le \lambda<a_{n+1},\ n\ge -1,
\end{equation}
are called the \textit{mediants} (or \textit{mediant convergents}) of $x=[a_0;a_1,a_2,\dots]$.  Notice that if $\lambda=0$, the expression in \eqref{mediants} gives $p_{n-1}/q_{n-1}$, while if $\lambda=a_{n+1}$, the expression gives $p_{n+1}/q_{n+1}$ by the recurrence relations \eqref{rec_rels_rcf}.  Since the mediant $(a+b)/(c+d)$ of two fractions $a/c$ and $b/d$ with positive denominators lies between them in value, monotonicity of the odd-/even-indexed {\sc rcf}-convergents give the following relations for all $n\ge 0$ (see \S1.4 of \cite{K97}):
\begin{equation}\label{meds-mon-odd}
x<\frac{p_{2n+1}}{q_{2n+1}}=\frac{a_{2n+1}p_{2n}+p_{2n-1}}{a_{2n+1}q_{2n}+q_{2n-1}}<\frac{(a_{2n+1}-1)p_{2n}+p_{2n-1}}{(a_{2n+1}-1)q_{2n}+q_{2n-1}}<\dots<\frac{p_{2n}+p_{2n-1}}{q_{2n}+q_{2n-1}}<\frac{p_{2n-1}}{q_{2n-1}}
\end{equation}
and
\begin{equation}\label{meds-mon-even}
\frac{p_{2n}}{q_{2n}}<\frac{p_{2n+1}+p_{2n}}{q_{2n+1}+q_{2n}}<\dots<\frac{(a_{2n+2}-1)p_{2n+1}+p_{2n}}{(a_{2n+2}-1)q_{2n+1}+q_{2n}}<\frac{a_{2n+2}p_{2n+1}+p_{2n}}{a_{2n+2}q_{2n+1}+q_{2n}}=\frac{p_{2n+2}}{q_{2n+2}}<x.
\end{equation}

\section{Some continued fraction algorithms}\label{Some continued fraction algorithms}

In this section we introduce some important continued fraction ({\sc cf}) algorithms which shall be revisited throughout the paper.  In particular, the reader will find in \S\ref{The Gauss map} the Gauss map, which generates {\sc rcf}-expansions; in \S\ref{Nakada's alpha-continued fractions} Nakada's parameterised family of $\alpha$-{\sc cf}s, which generate {\sc srcf}s including {\sc rcf}s, Hurwitz's singular {\sc cf}s, nearest integer {\sc cf}s, and R\'enyi's backward {\sc cf}s; and in \S\ref{S-expansions} the second-named author's $S$-expansions, which generate {\sc srcf}s including Minkowski's diagonal {\sc cf}s, Bosma's optimal {\sc cf}s and (a strict subcollection of) Nakada's $\alpha$-{\sc cf}s.

\subsection{The Gauss map}\label{The Gauss map}

\subsubsection{The Gauss map}

\begin{figure}[t]
\centering
\includegraphics[width=.27\textwidth]{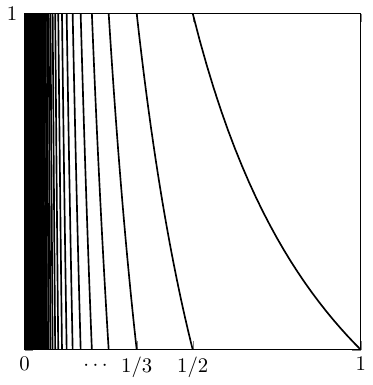}
\hspace{10pt}
\raisebox{.01\height}{\includegraphics[width=.3\textwidth]{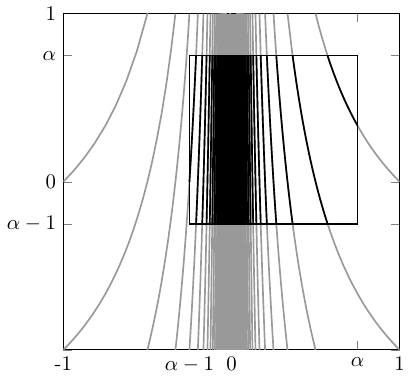}}
\hspace{17pt}
\includegraphics[width=.27\textwidth]{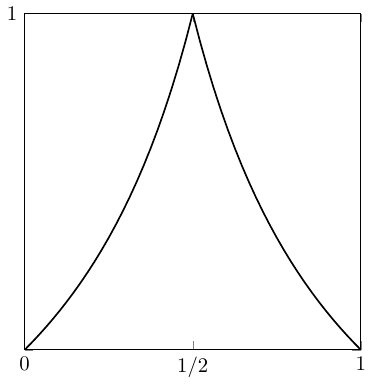}
\caption{Graphs of the Gauss map $G$ (left), Nakada's $\alpha$-{\sc cf} maps $G_\alpha$ (center), and the Farey tent map $F$ (right).}
\label{Gauss_alpha_Farey_fig}
\end{figure}

The partial denominators $a_n$ of {\sc rcf}-expansions are generated by the \textit{Gauss map} $G:[0,1]\to[0,1]$ defined by\footnote{While $G$ may be defined as a self-map of $[0,1),$ we choose to include the endpoint $1$ for later notational purposes.} $G(0)=0$ and $G(x)=1/x-\lfloor 1/x\rfloor,\ x>0$; see Figure \ref{Gauss_alpha_Farey_fig}.  Indeed, for $x\in\mathbb{R}$, set $a_0=a_0(x):=\lfloor x\rfloor$ and $x_0:=x-a_0\in [0,1)$.  Define $a(0):=\infty$, $a(x):=\lfloor 1/x\rfloor$ for $x\neq 0$, and $a_n=a_n(x):=a(G^{n-1}(x_0))$ for $n>0$.  Notice that for any integer $k\ge 1$, $a_n=k$ if and only if $G^{n-1}(x_0)\in (1/(k+1),1/k]$.  One finds that for $G^{n-1}(x_0)\neq 0$,
\[G^n(x_0)=\frac{1}{G^{n-1}(x_0)}-a_n.\]
Rearranging gives
\[G^{n-1}(x_0)=\frac1{a_n+G^n(x_0)},\]
which holds for both $G^{n-1}(x_0)\neq 0$ and $G^{n-1}(x_0)=0$, and which---with repeated applications---in turn gives
\[x=a_0+\cfrac1{a_1+\cfrac1{a_2+\ddots+\cfrac1{a_n+G^n(x_0)}}}=[a_0;a_1,\dots,a_{n-1},a_n+G^n(x_0)].\]
Symbolically, the Gauss map acts as a left-shift on {\sc rcf}-expansions.  That is, if $x=[0;a_1,a_2,\dots]\in(0,1)$, then $G(x)=[0;a_2,a_3,\dots].$

The dynamical system $([0,1],\mathcal{B},\nu_G,G)$ is exact (and hence strongly mixing, weakly mixing and ergodic; see \cite{IK02}), where $\mathcal{B}$ is the Borel $\sigma$-algebra\footnote{Throughout, $\mathcal{B}$ represents the Borel $\sigma$-algebra on the appropriate domain.} on $[0,1]$ and $\nu_G$ is the \textit{Gauss measure}, which is the absolutely continuous, $G$-invariant probability measure with density $1/(\log 2(1+x))$.

\subsubsection{The natural extension of the Gauss map}

In the late 1970s and early 1980s, Nakada, Ito and Tanaka (\cite{N1981,NIT77}) introduced an explicit, planar natural extension $(\Omega,\mathcal{B},\bar\nu_G,\mathcal{G})$ of the system $([0,1],\mathcal{B},\nu_G,G)$.  Here $\Omega:=[0,1]^2$; the map $\G:\Omega\to\Omega$ is defined by $\G(0,y)=(0,y)$ and for $z=(x,y)\in\Omega$ with $x\neq 0$,
\begin{equation}\label{GNE}
\G(z):=\left(G(x),\frac{1}{a(x)+y}\right);
\end{equation}
and the $\G$-invariant probability measure $\bar\nu_G$ has density $1/(\log 2(1+xy)^2)$.  Since $([0,1],\mathcal{B},\nu_G,G)$ is strongly mixing, so is the natural extension $(\Omega,\mathcal{B},\bar\nu_G,\mathcal{G})$.

Symbolically, the map $\G$ acts as a two-sided-shift on {\sc rcf}-expansions.  That is, if 
\[(x,y)=([0;a_1,a_2,\dots],[0;b_1,b_2,\dots])\in\Omega\] 
with $x\in(0,1)$, then
\begin{equation}\label{GaussNE}
\G(x,y)=([0;a_2,a_3,\dots],[0;a_1,b_1,b_2,\dots]).
\end{equation}
The map $\G$ may also be understood geometrically: setting
\begin{equation}\label{V&K}
V_k:=\left(\frac1{k+1},\frac1{k}\right]\times [0,1]\quad \text{and}\quad H_k:=[0,1]\times\left(\frac1{k+1},\frac1{k}\right]
\end{equation}
for each integer $k\ge 1$, one finds that $\G(V_k)=H_k$, up to null sets; see Figure \ref{GaussNE-fig}.  We call $V_k$ and $H_k$ the \textit{$k^{\textit{th}}$ vertical} and \textit{horizontal regions}, respectively.  

\begin{figure}
\centering
\includegraphics[width=.3\textwidth]{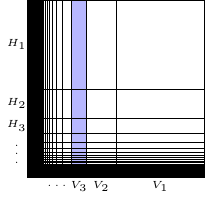}
\hspace{20pt}
\includegraphics[width=.3\textwidth]{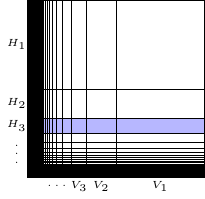}
\caption{Up to null sets, the map $\G$ sends the vertical region $V_k$ to the horizontal region $H_k$.}
\label{GaussNE-fig}
\end{figure}

\subsection{Nakada's $\alpha$-continued fractions}\label{Nakada's alpha-continued fractions}

In 1981, Nakada (\cite{N1981}) introduced a one-parameter family of continued fraction algorithms, called \emph{$\alpha$-{\sc cf} maps}, each of which generates---in a similar fashion as the Gauss map---{\sc srcf}-expansions.  For each $\alpha\in [0,1]$, Nakada's $\alpha$-{\sc cf} map $G_\alpha:[\alpha-1,\alpha]\to [\alpha-1,\alpha]$ is defined by $G_\alpha(0):=0$ and for $x\neq 0$,
\[G_\alpha(x):=\frac{1}{|x|}-\left\lfloor\frac1{|x|}+1-\alpha\right\rfloor;\]
see Figure \ref{Gauss_alpha_Farey_fig}.  Notice that $G_1=G$ is the Gauss map.  
In fact, Nakada's $\alpha$-{\sc cf}s contain several other well-studied continued fraction algorithms:  when $\alpha=(\sqrt{5}-1)/2$, $G_\alpha$ generates Hurwitz's \emph{singular} {\sc cf}s (\cite{H1889}); $\alpha=1/2$ generates the \emph{nearest integer} {\sc cf}s introduced by Minnigerode in \cite{M1873} and studied by Hurwitz in \cite{H1889}; and $\alpha=0$ generates R\'enyi's \emph{backward} {\sc cf}s (\cite{R1957}).  The latter map $G_0$ has an infinite, $\sigma$-finite, absolutely continuous invariant measure $\rho_0$ with density $1/(x+1)$, while for $\alpha\in (0,1]$, there is a unique, absolutely continuous invariant probability measure $\rho_\alpha$ (\cite{LM2008}).  Moreover, for each $\alpha\in [0,1]$, the dynamical system $([\alpha-1,\alpha],\mathcal{B},\rho_\alpha,G_\alpha)$ is exact and, hence, ergodic (\cite{LM2008}).

Since Nakada's introduction of the $\alpha$-{\sc cf}s, much work has been done to determine explicitly the invariant measures $\rho_\alpha$ and to understand the metric entropy $h(G_\alpha)=h_{\rho_\alpha}(G_\alpha)$ as a function $\alpha\mapsto h(G_\alpha)$ of $\alpha\in(0,1]$.  Nakada restricted his initial study in \cite{N1981} to $\alpha\ge 1/2$ and derived $\rho_\alpha$ by constructing an explicit, planar natural extension of $([\alpha-1,\alpha],\mathcal{B},\rho_\alpha,G_\alpha)$.  However, it was observed at the time that difficulties arose in extending these methods to $\alpha<1/2$.  The second-named author in 1991 (\cite{K1991}) reobtained Nakada's natural extensions for $\alpha\ge 1/2$ in a simple fashion as special instances within his theory of $S$-expansions; see \S\ref{S-expansions} below.  In 1999, Moussa et al. (\cite{MCM1999}) determined explicit, absolutely continuous invariant probability measures for a subset of a slightly different family of maps called \emph{folded $\alpha$-{\sc cf}s}, which are factors of the $\alpha$-{\sc cf} maps.  From their results one could obtain $\rho_\alpha$ for $\sqrt{2}-1\le \alpha< 1/2$ and---using Rohlin's formula---the entropy $h(G_\alpha)$ as a function of $\alpha\in[\sqrt{2}-1,1]$ (see \cite{LM2008}):
\begin{equation}\label{alphaCF-entropy}
h(G_\alpha)=\begin{cases}
\frac{\pi^2}{6\log(1+g)}, & \sqrt{2}-1\le\alpha\le g,\\
\frac{\pi^2}{6\log (1+\alpha)}, & g<\alpha\le 1,
\end{cases}
\end{equation}
where $g:=(\sqrt{5}-1)/2$.  Following this, the entropy function was conjectured to be monotone increasing and continuous on the remaining subinterval $(0,\sqrt{2}-1)$ (\cite{BDV2002}).  

It was thus quite surprising when, in 2008, Luzzi and Marmi (\cite{LM2008}) gave numerical evidence suggesting that $h(G_\alpha)$ possessed a seemingly complicated, non-monotone, self-similar structure on $(0,\sqrt{2}-1)$.  In the same year, Nakada and Natsui (\cite{nakada_natsui_08}) confirmed this non-monotonicity by giving countably many non-empty intervals on which the function is increasing, decreasing and constant, respectively.  These intervals are determined by a phenomenon called \emph{matching}, where the $G_\alpha$-orbits of $\alpha$ and $\alpha-1$ coincide after some finite number of steps.  Further numerical data on these so-called \emph{matching intervals} was given in \cite{CMPT2010}, and the authors also exhibited points in the complement of the union of matching intervals at which the entropy function fails to be locally monotone.  The matching intervals were completely classified in \cite{carminati_tiozzo_12}, and their union was shown to have full measure.  (These intervals have surprising connections to unimodal maps, the real slice of the boundary of the Mandelbrot set, and the parameter space of a family of maps producing signed binary expansions (\cite{bonanno_carminati_isola_tiozzo_13,carminati_tiozzo_13,DK2020}).)

In 2012, Kraaikamp, Schmidt and Steiner (\cite{KSS2012}) proved that the entropy function is indeed continuous on $(0,1]$ (this fact had also been proven in a 2009 preprint of Tiozzo for $\alpha>0.056\dots$ and was later improved to H\"older continuity on $(0,1]$ (\cite{T2009v1,T2014})).  In \cite{KSS2012}, the authors construct a planar natural extension of $([\alpha-1,\alpha],\mathcal{B},\rho_\alpha,G_\alpha)$ for each $\alpha>0$.  The domain of this natural extension is first defined theoretically as an orbit closure of a certain planar map; a further detailed analysis of the $G_\alpha$-orbits of $\alpha$ and $\alpha-1$ allow for a more explicit description of this domain (see \S7 of \cite{KSS2012}).  Moreover, the authors prove (Theorem 2 of \cite{KSS2012}) a conjecture of Luzzi and Marmi (\cite{LM2008}) that the product of the entropy $h(G_\alpha)$ and the measure of the natural extension domain (using density $1/(1+xy)^2$) is constant---in fact, equal to $\pi^2/6$---as a function of $\alpha$, and they extend the constant branch of $h(G_\alpha)$ in (\ref{alphaCF-entropy}) to the maximal interval $[g^2,g]$.  However, even equipped with such machinery, a number of open questions are left at the end of \cite{KSS2012}.  In particular, the authors ask for explicit values of the entropy $h(G_\alpha)$ for $\alpha<g^2$, and they restate a conjecture of \cite{CMPT2010} on the explicit form of the density of $\rho_\alpha$.

\subsection{$S$-expansions}\label{S-expansions}

In 1991, the second-named author introduced in \cite{K1991} a large class of new continued fraction algorithms by coupling two tools: singularisation and induced transformations of the natural extension $(\Omega,\mathcal{B},\bar\nu_G,\G)$ of the Gauss map.  Singularisation is an old, arithmetic procedure---tracing back as early as Lagrange's addendum (\cite{La1798}) to Euler's \emph{Vollst\"andige Anleitung zur Algebra}---whereby one can sometimes manipulate a {\sc srcf}-expansion to produce a new, `accelerated' {\sc srcf}-expansion of the same number.  

Indeed, suppose that $x$ has a {\sc srcf}-expansion $[\beta_0/\alpha_0;\alpha_1/\beta_1,\alpha_2/\beta_2,\dots]$ with convergents $(P_n/Q_n)_{n\ge -1}$.  Suppose, moreover, that $\beta_{n+1}=\alpha_{n+2}=1$ for some $n\ge 0$.  \emph{Singularisation at position $n$} replaces this {\sc srcf}-expansion with the {\sc srcf}-expansion
\[x=[\beta_0/\alpha_0;\alpha_1/\beta_1,\dots,\alpha_{n-1}/\beta_{n-1},\alpha_n/(\beta_n+\alpha_{n+1}),-\alpha_{n+1}/(\beta_{n+2}+1),\alpha_{n+3}/\beta_{n+3},\dots].\]
One can show that the convergents $(P_n'/Q_n')_{n\ge -1}$ of this new expansion satisfy
\[\frac{P_j'}{Q_j'}=\begin{cases}
\frac{P_j}{Q_j}, & j<n,\\
\frac{P_{j+1}}{Q_{j+1}}, & j\ge n,
\end{cases}\]
i.e., singularisation at position $n$ removes the $n^\text{th}$ convergent $P_n/Q_n$; see \cite{K1991}.  By iterating this procedure (possibly infinitely many times), one obtains a new {\sc srcf}-expansion of $x$ whose convergents are a subsequence $(P_{n_k}/Q_{n_k})_{k\ge -1}$ of the original convergents.

Beginning from a {\sc rcf}-expansion $[a_0;a_1,a_2,\dots]=[a_0/1;1/a_1,1/a_2,\dots]$ with convergents $p_n/q_n$, acceleration via singularisation admits two major restrictions:
\begin{enumerate}
\item[(i)] the convergents $p_n/q_n$ which are removed correspond to partial denominators $a_{n+1}=1$, and 

\item[(ii)] consecutive convergents $p_n/q_n,\ p_{n+1}/q_{n+1}$ cannot be removed.
\end{enumerate}
Restriction (ii) comes from the fact that in order to remove both $p_n/q_n$ and $p_{n+1}/q_{n+1}$, one would need to either first singularise the original expansion at position $n$, then singularise the new expansion again at position $n$, or first singularise at position $n+1$ and then at position $n$.  However, in either case, the partial denominator corresponding to the second singularisation is strictly greater than $1$, contrary to the singularisation requirements.  Nevertheless, one may singularise to remove non-consecutive convergents $p_n/q_n$ with $a_{n+1}=1$ independent of order and, thus, simultaneously; see \cite{K1991}.

In \cite{K1991}, the natural extension $(\Omega,\mathcal{B},\bar\nu_G,\G)$ of the Gauss map is used to govern the singularisation process, beginning from {\sc rcf}-expansions.  In particular, one fixes a measurable \emph{singularisation area}\footnote{Technically, condition (a) should be replaced by $S\subset \overline{V_1}$ and (b) by $S\cap \G(S)\subset\{(g,g)\}$ with $g=(\sqrt{5}-1)/2$; see Definition 4.4 and Remark 4.6.ii of \cite{K1991}.  Moreover, in \cite{K1991}, $\G$ is defined on $[0,1)\times [0,1]$ rather than $\Omega=[0,1]^2$.  What follows in \S\ref{S-expansions, revisited} below can be done with these minor adjustments, but for simplicity we shall omit these details.} $S\subset\Omega$ satisfying $\bar\nu_G(\partial S)=0$,
\begin{itemize}
\item[(a)] $S\subset V_1$, and
\item[(b)] $S\cap\G(S)=\varnothing$, 
\end{itemize}
and considers the $\G$-orbit of the point $(x,0)$ in $\Omega$ with $x=[0;a_1,a_2,\dots]$ irrational.  That $\bar\nu_G(\partial S)=0$ is a technical condition, called \emph{$\bar\nu_G$-continuity}, guaranteeing that for Lebesgue--a.e.~$x$, the $\G$-orbit of $(x,0)$ behaves like a `$\bar\nu_G$-generic' point; see Remark 4.6.i of \cite{K1991}.  Condition (a) guarantees that if $\G^n(x,0)\in S$, then $a_{n+1}=1$, and condition (b) guarantees that two consecutive points in the $\G$ orbit of $(x,0)$ do not belong to $S$; cf.~restrictions (i) and (ii) above.  By simultaneously singularising the {\sc rcf}-expansion of $x$ at all positions $n$ for which $\G^n(x,0)\in S$, one obtains a {\sc srcf}-expansion of $x$, called an \emph{$S$-expansion}, whose convergents are a subsequence of the {\sc rcf}-convergents of $x$.  Put a different way, $S$-expansions are {\sc srcf}s whose convergents are precisely the subsequence of {\sc rcf}-convergents $p_{n}/q_{n}$ for which $\G^{n}(x,0)\in \Omega\setminus S$ for $n\ge 0$, i.e., $S$-expansions are governed by the induced transformation of $(\Omega,\mathcal{B},\bar\mu_G,\G)$ on $\Omega\setminus S$.

In addition to defining new {\sc cf}-algorithms, the author shows in \cite{K1991} that many previously studied {\sc cf}-algorithms are realised by certain singularisation areas $S$.  Since these arise from induced transformations of a common dynamical system, ergodic properties of the underlying algorithms are easily comparable with one another.  The collection of $S$-expansions include Minkowski's diagonal {\sc cf}s (\cite{M1900}), Bosma's optimal {\sc cf}s (\cite{B1987}) and (the natural extensions of) Nakada's $\alpha$-{\sc cf}s.  However, the $\alpha$-{\sc cf}s realised as $S$-expansions are---somewhat curiously---only those for which $\alpha\ge 1/2$ (cf.~\cite{dJK2018}).  This is explained\footnote{Notation is changed from the original to match that of the current paper.} by Nakada and Natsui in \cite{nakada_natsui_08} for $\alpha\in[\sqrt{2}-1,1/2)$:

\begin{quote}
Here we note that the natural extension of $G_\alpha$ cannot be obtained by a simple induced transformation, in the sense of \cite{K1991}, of the natural extension of $G$.  This is related to the fact that a convergent of the continued fraction expansion of $x$ by $G_\alpha$ may not be a convergent of the [regular] continued fraction expansion of $x$.
\end{quote}

In \S\ref{Nakada's alpha-continued fractions, revisited} below, we exhibit the natural extension of each $([\alpha-1,\alpha],\mathcal{B},\rho_\alpha,G_\alpha)$, $\alpha\in (0,1],$ as an induced transformation of the natural extension of another, slower continued fraction map---the Farey tent map.

\section{The Farey tent map and Farey expansions}\label{The Farey tent map and Farey expansions}

In this section we introduce another {\sc cf}-map---the Farey tent map---whose natural extension (see \S\ref{Ito's natural extension of the Farey tent map} below) shall be of central importance to us.  The Farey tent map generates {\sc srcf}-expansions whose convergents consist of all {\sc rcf}-convergents and mediant convergents; see (\ref{mediants}) above and Proposition \ref{FareyConvs} below.  Much of this background can be found also in \cite{DKS2023}, but we include it here for completeness.

\subsection{The Farey tent map}

Define $\varepsilon:[0,1]\to\{0,1\}$ by
\[\varepsilon(x):=\begin{cases}
0, & x\le1/2,\\
1, & x> 1/2,
\end{cases}\]
and for $\varepsilon\in\{0,1\}$, set
\[A_{\varepsilon}:=\begin{pmatrix}1-\varepsilon & \varepsilon\\ 1 & 1\end{pmatrix}.\]
The \textit{Farey tent map} $F:[0,1]\to [0,1]$ is defined by 
\begin{equation}\label{F_interval}
F(x):=A_{\varepsilon(x)}^{-1}\cdot x=\begin{cases}
x/(1-x), & x\le1/2,\\
(1-x)/x, & x> 1/2;
\end{cases}
\end{equation}
see Figure \ref{Gauss_alpha_Farey_fig} above.  The dynamical system $([0,1],\mathcal{B},\mu,F)$ is ergodic, where $\mu$ is the infinite, $\sigma$-finite, absolutely continuous $F$-invariant measure with density $1/x$ (\cite{daniels_62,I1989,parry_62}).  One finds from the definition of $F$ that if $x\in[0,1]$ has {\sc rcf}-expansion\footnote{If the expansion of $x$ is finite, we set the remaining digits equal to $\infty$, e.g., $x=[0;a_1,\dots,a_n,\infty,\infty,\dots]$.  This also holds for $x$ equal to $0=[0;\infty,\infty,\dots]$ and $1=[0;1,\infty,\infty,\dots]$, interpreting $\infty-1=\infty$.} $x=[0;a_1,a_2,a_3,\dots]$, then
\begin{equation}\label{symbolic_farey}
F(x)=\begin{cases}
[0;a_1-1,a_2,a_3,\dots], & a_1>1,\\
[0;a_2,a_3,a_4,\dots], & a_1=1.
\end{cases}
\end{equation}
From this, it follows that the Gauss map $G$ is the \emph{jump transformation} of $F$ associated to the interval $(1/2,1]$, meaning that for $x$ as above with $x\neq 0$,
\[\min\{j\ge0\ |\ F^j(x)\in(1/2,1]\}=a_1-1, \quad \text{and} \quad G(x)=F^{a_1}(x);\]
see, e.g., \S11.4 of \cite{DK2021}.

\subsection{Farey expansions and Farey convergents}\label{Farey expansions and Farey convergents}

\subsubsection{The Farey tent map and {\sc rcf}-convergents and mediants}

In \cite{I1989}, Ito studied the ergodic properties of the dynamical system $([0,1],\mathcal{B},\mu,F)$ and showed via matrix relations that $F$ generates all convergents and mediant convergents of the {\sc rcf}-expansion of any irrational $x\in (0,1)$.  We reproduce this fact here, fixing notation\footnote{Notation is largely recycled from \cite{I1989} but with matrix entries permuted.} along the way. 

Recall from (\ref{F_interval}) that $F(x)=A_{\varepsilon(x)}^{-1}\cdot x$, or $x=A_{\varepsilon(x)}\cdot F(x)$.  Setting 
\begin{equation}\label{x_n&epsilon_n}
x_n:=F^n(x)\qquad \text{and} \qquad \varepsilon_{n+1}=\varepsilon_{n+1}(x):=\varepsilon(x_n),\qquad n\ge 0,
\end{equation}
we find for each $n\ge 0$ that $x_n=A_{\varepsilon(x_n)}\cdot F(x_n)=A_{\varepsilon_{n+1}}\cdot x_{n+1}$.  Repeatedly applying this beginning from $x=x_0$, we have
\[x=(A_{\varepsilon_1}A_{\varepsilon_2}\cdots A_{\varepsilon_{n}})\cdot x_n.\]
Let $A_{[0,0]}:=I_2$ be the identity matrix and for $n>0$,
\begin{equation}\label{A_[0,n]}
A_{[0,n]}=A_{[0,n]}(x):=A_{\varepsilon_1}A_{\varepsilon_2}\cdots A_{\varepsilon_n}.
\end{equation}
With this notation, 
\[x_n=F^n(x)=A_{[0,n]}^{-1}\cdot x, \quad n\ge 0.\]

We wish to determine the entries of $A_{[0,n]}$.  For $x=[0;a_1,a_2,\dots]$ irrational and $n\ge 0$, let $j_n=j_n(x)$ and $\lambda_n=\lambda_n(x)$ be the unique integers\footnote{These should be thought of in light of Euclid's division lemma: for integers $n,a\ge 0$, there exist unique integers $j$ and $\lambda$ such that $n=ja+\lambda$ with $0\le \lambda<a$.  Instead of summing a fixed integer $a$ with itself $j$ times, we sum the first $j$ {\sc rcf}-digits $a_1,a_2,\dots,a_j$ of $x$.} satisfying
\begin{equation}\label{j_n&lambda_n_alt}
n=a_1+a_2+\dots+a_{j_n}+\lambda_n,\qquad j_n\ge 0,\quad 0\le \lambda_n<a_{j_n+1}.
\end{equation}
From (\ref{symbolic_farey}), we have
\[\varepsilon_1\varepsilon_2\varepsilon_3\cdots=0^{a_1-1}10^{a_2-1}10^{a_3-1}1\cdots,\]
so
\begin{equation}\label{epsilon_seq}
\varepsilon_1\varepsilon_2\cdots\varepsilon_n=0^{a_1-1}10^{a_2-1}1\cdots 0^{a_{j_n}-1}10^{\lambda_n}.
\end{equation}
Denote the entries of $A_{[0,n]},\ n\ge 0,$ by
\[
\convmat{n}=
\begin{pmatrix}
u_n(x) & t_n(x)\\
s_n(x) & r_n(x)
\end{pmatrix}
:=A_{[0,n]},
\]
and observe that for any $k\in\mathbb{Z}$,
\begin{equation}\label{A_0^kA_1}
A_0^kA_1=\begin{pmatrix}1 & 0\\ 1 & 1\end{pmatrix}^k\begin{pmatrix}0 & 1\\ 1 & 1\end{pmatrix}=\begin{pmatrix}1 & 0\\ k & 1\end{pmatrix}\begin{pmatrix}0 & 1\\ 1 & 1\end{pmatrix}=\begin{pmatrix}0 & 1\\ 1 & k+1\end{pmatrix}.
\end{equation}
From (\ref{epsilon_seq}) and (\ref{A_0^kA_1}), it follows for $n>0$ that
\begin{align}\label{convmats}
\convmat{n}=A_{[0,n]}&=A_{\varepsilon_1}\cdots A_{\varepsilon_n}\notag \\
&=A_0^{a_1-1}A_1\cdots A_0^{a_{j_n}-1}A_1A_0^{\lambda_n} \notag\\
&=\begin{pmatrix}0 & 1\\ 1 & a_1\end{pmatrix}\cdots\begin{pmatrix}0 & 1\\ 1 & a_{j_n}\end{pmatrix}\begin{pmatrix}1 & 0\\ {\lambda_n} & 1\end{pmatrix} \notag\\
&=\begin{pmatrix}p_{{j_n}-1} & p_{j_n}\\ q_{{j_n}-1} & q_{j_n}\end{pmatrix}\begin{pmatrix}1 & 0\\ {\lambda_n} & 1\end{pmatrix} \notag\\
&=\begin{pmatrix}{\lambda_n} p_{j_n}+p_{{j_n}-1} & p_{j_n}\\ {\lambda_n} q_{j_n}+q_{{j_n}-1} & q_{j_n}\end{pmatrix},
\end{align}
where $p_j/q_j$ is the $j^\text{th}$ {\sc rcf}-convergent of $x$ (see also Lemma 1.1 of \cite{I1989}).  Equality of the first and final expressions also holds for $n={j_n}={\lambda_n}=0$ since, in this case, both matrices are the identity $I_2$.  
As a sequence, the quotients of the left-hand columns of the matrices in (\ref{convmats}) are 
\begin{alignat}{3}\label{FLseq}
\left(\frac{u_n}{s_n}\right)_{n\ge0}=\left(\frac{\lambda_np_{j_n}+p_{j_n-1}}{\lambda_nq_{j_n}+q_{j_n-1}}\right)_{n\ge0}
=\bigg(&\frac{p_{-1}}{q_{-1}}&&,\frac{p_0+p_{-1}}{q_0+q_{-1}}&&,\dots,\frac{(a_1-1)p_0+p_{-1}}{(a_1-1)q_0+q_{-1}},\\
&\frac{p_{0}}{q_{0}}&&,\frac{p_1+p_0}{q_1+q_0}&&,\dots,\frac{(a_2-1)p_1+p_0}{(a_2-1)q_1+q_0},\dots,\notag\\
&\frac{p_{j-1}}{q_{j-1}}&&,\frac{p_j+p_{j-1}}{q_j+q_{j-1}}&&,\dots,\frac{(a_{j+1}-1)p_j+p_{j-1}}{(a_{j+1}-1)q_j+q_{j-1}},\dots\bigg),\notag
\end{alignat}
i.e., the map $F$ generates all {\sc rcf}-convergents and mediants.  Notice that the denominators $(s_n)_{n\ge 0}$ do not form an increasing sequence.  Supposedly to `remedy' this, in \cite{I1989} Ito instead considers the sequence $((u_n+t_n)/(s_n+r_n))_{n\ge 0}$ with increasing denominators.  However, in light of Proposition \ref{FareyConvs} below, we find it more natural to study $(u_n/s_n)_{n\ge 0}$.

\subsubsection{Lehner and Farey expansions}

Originally, there was no continued fraction expansion associated to the Farey tent map $F$. Such expansions do exist and can be obtained from a map introduced by Lehner in 1994; see \cite{lehner_94}.  The \emph{Lehner map} (also referred to as `the mother of all continued fractions' in \cite{DK2000}) is the map $L:[1,2]\to[1,2]$ defined by
\[
L(x):=\begin{cases}
1/(2-x), & x\le3/2,\\
1/(x-1), & x>3/2.
\end{cases}
\]
For $x\in[1,2]$ and each $n\ge 0$, set
\[
(b_n,e_{n+1})=(b_n(x),e_{n+1}(x)):=\begin{cases}
(2,-1), & L^n(x)\le3/2,\\
(1,1), & L^n(x)> 3/2.
\end{cases}
\]
The digits $(b_n,e_{n+1})$ generate the so-called \textit{Lehner expansion} of $x\in[1,2]$,
\begin{equation}\label{lehner_expn}
x=[b_0/1;e_1/b_1,e_2/b_2,\dots],
\end{equation}
which is a {\sc srcf}-expansion (see \cite{DK2000,lehner_94}).  

Lehner studied expansions of the form (\ref{lehner_expn}) generated by $L$ but no dynamical properties of this map.  In \cite{DK2000} it is observed that the dynamical systems $([0,1],\mathcal{B},\mu,F)$ and $([1,2],\mathcal{B},\rho,L)$ are isomorphic via the translation $x\mapsto x+1$, where $\rho$ is the absolutely continuous, $L$-invariant measure with density $1/(x-1)$.  Through this isomorphism, the Farey tent map $F$ can be used to generate a \textit{Farey expansion} for each $x\in[0,1]$ (see also \cite{IS2008}).  Indeed, for $x\in[0,1]$, let $\varepsilon_{n+1}=\varepsilon_{n+1}(x),$ $n\ge 0,$ be as in \eqref{x_n&epsilon_n}, and let $[b_0/1;e_1/b_1,e_2/b_2,\dots]$ be the Lehner expansion of $x+1$.  Then $x=[b_0-1/1;e_1/b_1,e_2/b_2,\dots]$, and we have
\begin{align*}
(b_n,e_{n+1})&=\left.\begin{cases}
(2,-1), & L^n(x+1)\le3/2\\
(1,1), & L^n(x+1)> 3/2
\end{cases}\right\}
=\left.\begin{cases}
(2,-1), & F^n(x)\le1/2\\
(1,1), & F^n(x)> 1/2
\end{cases}\right\}\\
&=\left.\begin{cases}
(2,-1), & \varepsilon_{n+1}=0\\
(1,1), & \varepsilon_{n+1}=1
\end{cases}\right\}
=(2-\varepsilon_{n+1},2\varepsilon_{n+1}-1).
\end{align*}
Hence $F$ generates {\sc srcf}-expansions, called \textit{Farey expansions}:
\begin{equation}\label{farey_expn}
x=[(1-\varepsilon_1)/1;(2\varepsilon_1-1)/(2-\varepsilon_2),(2\varepsilon_2-1)/(2-\varepsilon_3),\dots].
\end{equation}

The convergents 
\[P_n/Q_n=[(1-\varepsilon_1)/1;(2\varepsilon_1-1)/(2-\varepsilon_2),(2\varepsilon_2-1)/(2-\varepsilon_3),\dots,(2\varepsilon_n-1)/(2-\varepsilon_{n+1})]\] 
of the Farey expansion of $x$ are called the \emph{Farey convergents} of $x$.  In \cite{DKS2023} (Proposition 3.1), it is observed that the sequence $(P_n/Q_n)_{n\ge -1}$ of Farey convergents is precisely the sequence $(u_n/s_n)_{n\ge 0}$ from \eqref{FLseq} of {\sc rcf}-convergents and mediants generated by $F$:

\begin{prop}\label{FareyConvs}
For each $n\ge 0$,
\[\begin{pmatrix}u_n \\ s_n\end{pmatrix}=\begin{pmatrix}P_{n-1} \\ Q_{n-1}\end{pmatrix},\]
where $P_n/Q_n$ is the $n^{\text{th}}$ Farey convergent of $x$.
\end{prop}

\section{Ito's natural extension of the Farey tent map}\label{Ito's natural extension of the Farey tent map}

We now come to one of the central tools of this article: the natural extension of the Farey tent map, originally introduced by Ito in 1989 (\cite{I1989}).  In \S\ref{The natural extension of the Farey tent map}, we recall the definition of Ito's natural extension and discuss a one-to-one correspondence between orbits under the natural extension map and Farey convergents.  Via this correspondence, we find that certain subregions of the domain of Ito's natural extension give rise to certain types of Farey convergents.  In \S\ref{Inducing Ito's natural extension}, we discuss induced transformations of Ito's natural extension and their connection with subsequences of Farey convergents.  Moreover, we revisit a theorem of Brown and Yin from \cite{BY1996} stating that the natural extension of the Gauss map is isomorphic to a certain induced transformation of Ito's natural extension, and we recall from \cite{DKS2023} that the entropy of our induced systems may be computed in terms of the measures of their domains (Theorem \ref{entropy_thm} below).  As in \S\ref{The Farey tent map and Farey expansions}, much of the material of this section can be found in \cite{DKS2023}, but we include it here for completeness and for some minor notational and definitional changes.

\subsection{The natural extension of the Farey tent map}\label{The natural extension of the Farey tent map}

In \cite{I1989}, Ito determined a planar natural extension $(\Omega,\mathcal{B},\bar\mu,\F)$ of the dynamical system $([0,1],\mathcal{B},\mu,F)$ associated to the Farey tent map.  The map $\F:\Omega\to\Omega$ is defined for each $z=(x,y)\in\Omega$ by
\begin{equation}\label{FareyNEeqn}
\F(z):=\left(A_{\varepsilon(x)}^{-1}\cdot x,A_{\varepsilon(x)}\cdot y\right)=\begin{cases}
\left(\frac{x}{1-x},\frac{y}{1+y}\right), & x\le1/2,\\
\left(\frac{1-x}{x},\frac{1}{1+y}\right), & x> 1/2,
\end{cases}
\end{equation}
where again $\Omega=[0,1]^2$, and $\bar\mu$ is the infinite, $\sigma$-finite, absolutely continuous $\F$-invariant measure with density $1/(x+y-xy)^2$.  Since $([0,1],\mathcal{B},\mu,F)$ is ergodic, so is its natural extension $(\Omega,\mathcal{B},\bar\mu,\F)$.

Notice that $\F$ is simply the Farey tent map $F$ in the first coordinate.  Setting $\varepsilon_{n+1}=\varepsilon_{n+1}(x)=\varepsilon(x_n)$ as in (\ref{x_n&epsilon_n}), we find that
\begin{equation}\label{z_n}
z_n=(x_n,y_n):=\F^n(z)=\left(A_{[0,n]}^{-1}\cdot x, A_{[n,0]}\cdot y\right), \quad n\ge 0,
\end{equation}
where $A_{[0,n]}$ is defined as in (\ref{A_[0,n]}), and 
\[A_{[n,0]}=A_{[n,0]}(x):=A_{\varepsilon_n}A_{\varepsilon_{n-1}}\cdots A_{\varepsilon_1}, \quad n\ge 1.\]
The entries of $A_{[n,0]}$ may be computed explicitly in terms of those of $A_{[0,n]}$ (recall (\ref{convmats})).  Indeed, if $x=[0;a_1,a_2,\dots],$ we have for $n>0$ that
\begin{align*}
A_{[n,0]}&=A_{\varepsilon_n}\cdots A_{\varepsilon_1}I_2\notag\\
&=A_0^{\lambda_n} A_1A_0^{a_{j_n}-1}\cdots A_1A_0^{a_1-1}A_1A_1^{-1}\notag\\
&=\begin{pmatrix}0 & 1\\ 1 & {\lambda_n}+1\end{pmatrix}\begin{pmatrix}0 & 1\\ 1 & a_{j_n}\end{pmatrix}\cdots\begin{pmatrix}0 & 1\\ 1 & a_1\end{pmatrix}\begin{pmatrix}-1 & 1\\ 1 & 0\end{pmatrix}\notag\\
&=\begin{pmatrix}0 & 1\\ 1 & {\lambda_n}+1\end{pmatrix}\left(\begin{pmatrix}0 & 1\\ 1 & a_1\end{pmatrix}\cdots\begin{pmatrix}0 & 1\\ 1 & a_{j_n}\end{pmatrix}\right)^T\begin{pmatrix}-1 & 1\\ 1 & 0\end{pmatrix}\notag\\
&=\begin{pmatrix}0 & 1\\ 1 & {\lambda_n}+1\end{pmatrix}\begin{pmatrix}p_{j_n-1} & q_{j_n-1}\\ p_{j_n} & q_{j_n}\end{pmatrix}\begin{pmatrix}-1 & 1\\ 1 & 0\end{pmatrix}\notag\\
&=\begin{pmatrix}q_{j_n}-p_{j_n} & p_{j_n}\\ ({\lambda_n}+1)q_{j_n}+q_{j_n-1}-(({\lambda_n}+1)p_{j_n}+p_{j_n-1}) &({\lambda_n}+1)p_{j_n}+p_{j_n-1}\end{pmatrix}\notag\\
&=\begin{pmatrix}r_n-t_n & t_n\\ s_n+r_n-(u_n+t_n) & u_n+t_n\end{pmatrix},
\end{align*}
and the first and final expressions are also equal to $I_2$ for $n=0$.  Notice, furthermore, that $A_{[0,n]}^T$ and $A_{[n,0]}$ are conjugate under $A_1$:
\begin{equation}\label{conjugation}
A_1A_{[0,n]}^T=\begin{pmatrix}t_n & r_n \\ u_n+t_n & s_n+r_n\end{pmatrix}=A_{[n,0]}A_1.
\end{equation}

\subsubsection{$\F$-orbits and Farey convergents}

Interpreting the map $\F$ symbolically and geometrically leads to a natural correspondence between $\F$-orbits and Farey convergents.  Fix $z=(x,y)\in\Omega$ with (finite\footnote{As in \eqref{symbolic_farey}, if the expansion of $x$ or $y$ is finite, we set the remaining digits equal to $\infty$.  If $x=1/2$, we take the shorter of its two {\sc rcf}-expansions, namely $x=[0;2]$.} or infinite) {\sc rcf}-expansions 
\begin{equation}\label{(x,y)-rcfs}
(x,y)=([0;a_1,a_2,\dots],[0;b_1,b_2,\dots]).
\end{equation}
One verifies using (\ref{symbolic_farey}) and (\ref{FareyNEeqn}) that 
\begin{equation}\label{F_NE_symbolic}
\F(z)=\begin{cases}
([0;a_1-1,a_2,\dots],[0;1+b_1,b_2,\dots]), & a_1>1,\\
([0;a_2,a_3,\dots],[0;1,b_1,b_2,\dots]), & a_1=1.
\end{cases}
\end{equation}
Recalling the vertical and horizontal regions from \eqref{V&K}, the image of the rectangle $V_{a}\cap H_{b}$ for $a>1$ is thus the rectangle $\F(V_{a}\cap H_{b})=V_{a-1}\cap H_{b+1}$ immediately below and to the right of the original rectangle, and the image of the right-half $V_1$ of $\Omega$ is the top half $\F(V_1)=H_1$, up to a Lebesgue-null set.  In particular, the iterates $\F^\lambda,\ 0\le \lambda<a$, `slide' the rectangle $V_a\cap H_1$ `down-and-right' through $a$ rectangles, and the next image $\F^{a}(V_{a}\cap H_1)$ is mapped back as a subset of $H_1$ (see Figure \ref{FareyNE}).  
\begin{figure}
\centering
\includegraphics[width=.24\textwidth]{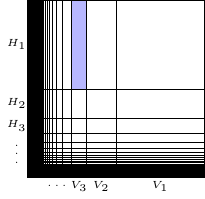}
\includegraphics[width=.24\textwidth]{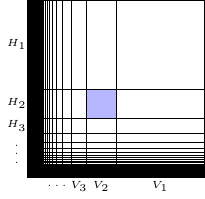}
\includegraphics[width=.24\textwidth]{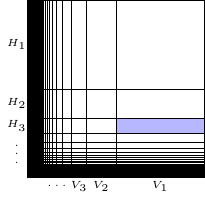}
\includegraphics[width=.24\textwidth]{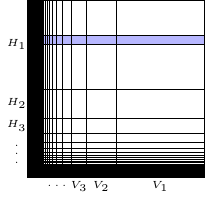}
\caption{From left to right: The sets $V_3\cap H_1,\ \F(V_3\cap H_1),\ \F^2(V_3\cap H_1)$ and $\F^3(V_3\cap H_1)$, respectively.}
\label{FareyNE}
\end{figure}

Now let $z=(x,y)$ be as in \eqref{(x,y)-rcfs} with $x$ irrational, and fix some $n\ge 0$.  Recall from (\ref{j_n&lambda_n_alt}) that $n$ may be written uniquely as
\[n=a_1+a_2+\dots+a_{j_n}+\lambda_n,\]
where $0\le \lambda_n<a_{j_n+1}$.  Repeatedly applying (\ref{F_NE_symbolic}), one finds
\begin{equation}\label{z_n-explicit}
z_n=\F^n(z)=\begin{cases}
([0;a_1-\lambda_n,a_2,\dots],[0;\lambda_n+b_1,b_2,\dots,\dots]), & n<a_1,\\
([0;a_{j_n+1}-\lambda_n,a_{j_n+2},\dots],[0;\lambda_n+1,a_{j_n},\dots,a_2,a_1-1+b_1,b_2,\dots]),\ & n\ge a_1.
\end{cases}
\end{equation}
In particular, if $z\in H_1$ so that $b_1=1$, then $z_n$ belongs to (the closure of) $V_{a_{j_n+1}-\lambda_n}\cap H_{\lambda_n+1}$ for all $n\ge 0$.

\begin{remark}\label{closure_remark}
The closure is needed in the previous statement if and only if $z=(x,1)$ with $a_1=1$ and $a_1\le n<a_1+a_2$.
Indeed, in this case
\[z_n=(x_n,y_n)=([0;a_2-\lambda_n,a_3,\dots],[0;\lambda_n+1,1]).\]
Hence $y_n=1/(\lambda_n+2)$ which implies $z_n\notin H_{\lambda_n+1}$ lies on the lower boundary of $H_{\lambda_n+1}$.  In all other cases for which $z\in H_1$, one has in fact $z_n\in V_{a_{j_n+1}-\lambda_n}\cap H_{\lambda_n+1}$ for all $n\ge 0$.  In particular, this annoyance for $z=(x,1)$ and $a_1=1$ is `corrected' for $n\ge a_1+a_2$, and the closures are no longer needed.  We shall frequently overlook this subtlety and make no mention of the special case $a_1=1$, and thus some claims should be understood up to this minor technicality.  See also Remark \ref{altered_maps} below. 
\end{remark}

Recall from Proposition \ref{FareyConvs} and (\ref{FLseq}) that the $(n-1)^\text{st}$ Farey convergent of $x$ is
\begin{equation}\label{Farey_convs}
\frac{u_n}{s_n}=\frac{\lambda_np_{j_n}+p_{j_n-1}}{\lambda_nq_{j_n}+q_{j_n-1}}.
\end{equation}
Identifying the $n^\text{th}$ point $z_n\in V_{a_{j_n+1}-\lambda_n}\cap H_{\lambda_n+1}$ of the $\F$-orbit of $z\in H_1$ with the $(n-1)^\text{st}$ Farey convergent $u_n/s_n$, we find that certain subregions $R\subset\Omega$ correspond to certain types of {\sc rcf}-convergents or mediants.  For instance, if $R=H_1$, then $z_n\in R$ if and only if $\lambda_n=0$ in (\ref{Farey_convs}), i.e., $u_n/s_n$ is a {\sc rcf}-convergent.  More generally, if $R=H_{\lambda+1}$, then $z_n\in R$ if and only if $\lambda_n=\lambda$, i.e., $u_n/s_n$ is a {\sc rcf}-convergent ($\lambda=0$) or a `$\lambda^\text{th}$ mediant' convergent ($\lambda>0$).  

Moreover, setting $R=V_a$, we have $z_n\in R$ if and only if $a=a_{j_n+1}-\lambda_n$, or $\lambda_n=a_{j_n+1}-a$.  Hence $z_n\in R$ if and only if $u_n/s_n$ is a {\sc rcf}-convergent ($a_{j_n+1}=a$) or `$(a-1)^\text{st}$-from-final' mediant convergent ($a_{j_n+1}>a$).  Lastly, setting $R=V_{a-\lambda}\cap H_{\lambda+1}$, we have $z_n\in R$ if and only if $\lambda_n=\lambda$ and $a_{j_n+1}=a$, i.e., $u_n/s_n$ is a {\sc rcf}-convergent ($\lambda=0$) or `$\lambda^\text{th}$ mediant' convergent ($\lambda>0$) with partial denominator $a_{j_n+1}=a$.

These observations naturally lead us to consider the dynamics of Ito's natural extension $(\Omega,\mathcal{B},\bar\mu,\F)$ restricted to certain subregions of the domain in order to `pick out' desired subsequences of Farey convergents.

\subsection{Inducing Ito's natural extension}\label{Inducing Ito's natural extension}

A $\bar\mu$-measurable set $R\subset \Omega$ is called \emph{inducible}\footnote{We remark that the definition of \emph{inducible} given here is broader than that in \cite{DKS2023}, where it is also required that $\bar\mu(\partial R)=0$.  This latter condition (called \emph{$\bar\mu$-continuity}) is not needed for our present purposes.} if either $0<\bar\mu(R)<\infty$ or $R=\Omega$.  In the former case (i.e., $\bar\mu(R)<\infty$), we call $R$ \emph{proper}.  For $R$ inducible, define the \textit{hitting time to $R$}, denoted $N_R:\Omega\to\mathbb{N}\cup\{\infty\}$, by
\begin{equation}\label{hit_time}
N_R(z):=\inf\{n\ge 1\ |\ \F^n(z)\in R\}.
\end{equation}
Since $(\Omega,\mathcal{B},\bar\mu,\F)$ is conservative and ergodic (\cite{BY1996}), $\bar\mu$--a.e.~$z\in\Omega$ enters $R$ infinitely often under iterates of $\F$ (see Remark 2.2.1 of \cite{DK2021}).  Unless otherwise stated, we assume throughout that the null set of points from any $S\subset \Omega$ whose $\F$-orbits enter $R$ at most finitely many times are removed from $S$, and---abusing notation---denote this new set again by $S$.  Define $\F_R:\Omega \to R$ by
\[\F_R(z):=\F^{N_R(z)}(z).\]
The \textit{induced map of $\F$ on $R$} is the map $\F_R$ restricted to $R$, and the \textit{induced measure} $\bar\mu_R$ is defined by 
\[\bar\mu_R(S):=\begin{cases}
\frac{\bar\mu(S)}{\bar\mu(R)}, \quad R\neq\Omega,\\
\bar\mu(S), \quad R=\Omega,
\end{cases} 
\quad \text{for all} \quad S\in\mathcal{B},\ S\subset R.\]
Ergodicity of the \textit{induced system} $(R,\mathcal{B},\bar\mu_R,\F_R)$ follows from that of $(\Omega,\mathcal{B},\bar\mu,\F)$.  Notice that $\bar\mu_R$ is a probability measure if and only if $R$ is proper.

Writing $z=(x,y)$ and setting $A_R(z):=A_{[0,N_R(z)]}(x),$ $A_R^{-1}(z):=(A_R(z))^{-1}$ and $A_R^{-T}(z):=(A_R(z))^{-T}$, Equations (\ref{z_n}) and (\ref{conjugation})
give
\begin{equation}\label{F_R2}
\F_R(z)=\left(A_R^{-1}(z)\cdot x,A_1A_R^T(z)A_1^{-1}\cdot y\right).
\end{equation}
We denote the entries of $A_R(z)$ by
\begin{equation}\label{A_R(z)}
\begin{pmatrix}
u_R(z) & t_R(z)\\
s_R(z) & r_R(z)
\end{pmatrix}
:=A_R(z)=A_{\varepsilon_1}\dots A_{\varepsilon_{N_R(z)}}.
\end{equation}
For $n\ge 0$, set $z_n^R=(x_n^R,y_n^R):=\F_R^n(z)$ and define $N_n^R(z)$ by $N_0^R(z):=0$ and 
\begin{equation}\label{N_n^R}
N_n^R(z):=\sum_{\ell=0}^{n-1}N_R(z_\ell^R), \quad n\ge 1.
\end{equation}
When the point $z$ is understood, we use the suppressed notation $N_n^R:=N_n^R(z),\ n\ge 0$.  
We remark that when $R=\Omega$, $N_n^R=n$ for all $n\ge 0$.  In general, the sequence $(N_n^R)_{n\ge 1}$ gives the indices $N\ge 1$ for which the forward orbit $(\F^N(z))_{N\ge 0}$ of $z$ enters the region $R$, so $\F_R^n(z)=\F^{N_n^R}(z),\ n\ge0$.  

Let $A_0^R(z):=I_2$ be the identity matrix, and for $n\ge 1$ set 
\begin{equation}\label{A_n^R}
A_n^R(z):=A_R(z_{n-1}^R)=A_{\varepsilon_{N_{n-1}^R+1}}\cdots A_{\varepsilon_{N_n^R}}.
\end{equation}
Notice that for $m\ge 1$ and $n\ge 0$,
\begin{equation}\label{A_m(z_n)=A_{m+n}(z)}
A_m^R(z_n^R)=A_R(z_{n+m-1}^R)=A_{n+m}^R(z). 
\end{equation}
Moreover, set 
\begin{equation}\label{A_{[m,n]}^R}
A_{[m,n]}^R(z):=A_m^R(z)A_{m+1}^R(z)\cdots A_{n}^R(z), \qquad 0\le m\le n,
\end{equation}
and denote the entries of $A_{[0,n]}^R(z)$ by
\begin{equation}\label{A_n^R(z)}
\begin{pmatrix}
u_n^R(z) & t_n^R(z)\\
s_n^R(z) & r_n^R(z)
\end{pmatrix}
:=A_{[0,n]}^R(z)=A_{\varepsilon_1}\cdots A_{\varepsilon_{N_n^R}}.
\end{equation}
When the point $z$ is understood, we use the suppressed notation
\begin{equation}\label{supr_not}
A_{[0,n]}^R=\begin{pmatrix}
u_n^R & t_n^R\\
s_n^R & r_n^R
\end{pmatrix}
:=A_{[0,n]}^R(z), \quad n\ge 0.
\end{equation}
From (\ref{A_n^R(z)}) and (\ref{A_[0,n]}), we have $A_{[0,n]}^R=A_{[0,N_n^R]}$ and thus, by (\ref{convmats}),
\begin{equation}\label{A_n^R_entries}
\begin{pmatrix}
u_n^R & t_n^R\\
s_n^R & r_n^R
\end{pmatrix}
=\begin{pmatrix}
u_{N_n^R} & t_{N_n^R}\\
s_{N_n^R} & r_{N_n^R}
\end{pmatrix}
=\begin{pmatrix}
\lambda_{N_n^R}p_{j_{N_n^R}}+p_{j_{N_n^R}-1} & p_{j_{N_n^R}}\\
\lambda_{N_n^R}q_{j_{N_n^R}}+q_{j_{N_n^R}-1} & q_{j_{N_n^R}}
\end{pmatrix}
, \quad n\ge 0.
\end{equation}

The following lemma will be useful in \S\ref{A two-sided shift for contracted Farey expansions} below.
\begin{lem}\label{s_R>0}
For any $z=(x,y)\in\Omega$, $u_R(z), s_R(z)\in\mathbb{Z}$ satisfy $s_R(z)>0$ and $0\le u_R(z)\le s_R(z)$. 
\end{lem}
\begin{proof}
It is clear from (\ref{A_R(z)}) that $u_R(z), s_R(z)\in\mathbb{Z}$.  Moreover, $A_R(z)=A_{[0,1]}^R$, so setting $N=N_1^R$, (\ref{A_n^R_entries}) gives
\[
\begin{pmatrix}
u_R(z)\\
s_R(z)
\end{pmatrix}=
\begin{pmatrix}
u_N\\
s_N
\end{pmatrix}=
\begin{pmatrix}
\lambda_Np_{j_N}+p_{j_N-1}\\
\lambda_Nq_{j_N}+q_{j_N-1}
\end{pmatrix}.
\]
Since $N>0$, (\ref{j_n&lambda_n_alt}) implies either $j_N>0$ or $\lambda_N>0$.  If $j_N=0$, then $u_R(z)=\lambda_Np_0+p_{-1}=p_{-1}=1$ and $s_R(z)=\lambda_Nq_0+q_{-1}=\lambda_N>0$, so the claim holds.  If $j_N>0$, then $s_R(z)\ge q_{j_N-1}\ge q_0=1$, and $u_R(z)/s_R(z)$ lies between $p_{j_N-1}/q_{j_N-1}$ and $p_{j_N+1}/q_{j_N+1}$, which are fractions between $0$ and $1$.  Thus $0\le u_R(z)\le s_R(z)$.  
\end{proof}

Notice from (\ref{A_n^R_entries}) that $(u_n^R/s_n^R)_{n\ge 0}$ is a subsequence of the Farey convergents $(u_n/s_n)_{n\ge 0}$.  In particular, the correspondence $z_n \leftrightarrow u_n/s_n$ between the $\F$-orbit of $z\in H_1$ and the Farey convergents of $x$ gives a correspondence $z_n^R \leftrightarrow u_n^R/s_n^R$ between the subsequence $(z_n^R)_{n\ge 0}=(z_{N_n^R})_{n\ge 0}$ of points in the $\F$-orbit of $z$ entering $R$ and the subsequence $(u_n^R/s_n^R)_{n\ge 0}=(u_{N_n^R}/s_{N_n^R})_{n\ge 0}$ of Farey convergents.  Hence a subregion $R$ naturally determines a subsequence of Farey convergents.  We illustrate with the examples discussed at the end of \S\ref{The natural extension of the Farey tent map}:

\begin{eg}\label{H&Vregions}
The region $R=H_{\lambda+1}$ corresponds to {\sc rcf}-convergents ($\lambda=0$) or $\lambda^{\text{th}}$ mediants ($\lambda>0$):
\[\left\{\frac{u_n^R}{s_n^R}\right\}_{n\ge 0}=\left\{\frac{\lambda_np_{j_n}+p_{j_n-1}}{\lambda_nq_{j_n}+q_{j_n-1}}\ \Big|\ \lambda_n=\lambda\right\}_{n\ge 0}.\]
The vertical regions $R=V_a$ give---in addition to the {\sc rcf}-convergents $p_{j-1}/q_{j-1}$ for which $a_{j+1}=a$---final mediants, next-to-final mediants, and so on for $a=1,2,\dots$, respectively:
\[\left\{\frac{u_n^R}{s_n^R}\right\}_{n\ge 0}=\left\{\frac{\lambda_np_{j_n}+p_{j_n-1}}{\lambda_nq_{j_n}+q_{j_n-1}}\ \Big|\ \lambda_n=a_{j_n+1}-a\right\}_{n\ge 0}.\]
The regions $R=V_{a-\lambda}\cap H_{\lambda+1}$ pick out {\sc rcf}-convergents ($\lambda=0$) or $\lambda^\text{th}$-mediants ($\lambda>1$) corresponding to partial denominators $a$ in the {\sc rcf}-expansion of $x$: 
\[\left\{\frac{u_n^R}{s_n^R}\right\}_{n\ge 0}=\left\{\frac{\lambda_np_{j_n}+p_{j_n-1}}{\lambda_nq_{j_n}+q_{j_n-1}}\ \Big|\ \lambda_n=\lambda,\ a_{j_n+1}=a\right\}_{n\ge 0}.\]
\end{eg}

\begin{remark}\label{altered_maps}
For the reasons explained in Remark \ref{closure_remark}, some of the statements in Example \ref{H&Vregions} are false for $z=(x,1)$, where $x=[0;a_1,a_2,\dots]$ with $a_1=a(x)=1$.  For some of the examples in \S\ref{Examples of contracted Farey expansions}, it will be advantageous to `fix' this.  In such cases, we may `adjust' the map $\F_R$ so that the corresponding statements on the subsets of Farey convergents are true for all $z=(x,1)$.  For instance, when $R=H_1$ and $z=(x,1)$ with $x>1/2$, the sequence $(u_n^R/s_n^R)_{n\ge 0}$ skips the {\sc rcf}-convergent $p_0/q_0$ of $x$.  To `catch' this convergent, we instead consider the map $\overline{\F}_{H_1}:\overline{H}_1\to \overline{H}_1$ where for $z=(x,y)\in \overline{H}_1$, $\overline{\F}_{H_1}(z):=\F^{a(x)}(z)$ if $x\neq 0$ and $\overline{\F}_{H_1}(z):=z$ if $x=0$.  The maps $\overline{\F}_{H_1}$ and $\F_{H_1}$ agree on $H_1\backslash (A\cup B)$, where $A=(1/2,1]\times \{1\}$ and $B=\{0\}\times(1/2,1]$.  The dynamical systems $(\overline{H_1},\mathcal{B},\bar\mu_{\overline{H}_1},\overline{\F}_{H_1})$ and $(H_1,\mathcal{B},\bar\mu_{H_1},\F_{H_1})$ are isomorphic under the identity map and thus---from an ergodic-theoretic point of view---indistinguishable.  Moreover, the map $\overline{\F}_{H_1}$ `fixes' the subtlety in Remark \ref{closure_remark} since, for $z\in A$, $\overline{\F}_{H_1}(z)=((1-x)/x,1/2)\in \overline{H}_1$.  Thus we \emph{do} include the convergent $p_0/q_0$ in $(u_n^R/s_n^R)_{n\ge 0}$.  Throughout, we shall often consider such altered systems without mention, denoting them again by $(R,\mathcal{B},\bar\mu_R,\F_R)$.
\end{remark}

Consider again $R=H_1$, which gives as a subsequence $(u_n^R/s_n^R)_{n\ge 0}$ of Farey convergents the {\sc rcf}-convergents of $x$. 
Now $R$ consists of all points $z=(x,y)$ as in (\ref{(x,y)-rcfs}) with $b_1=1$, so---after the alteration of Remark \ref{altered_maps}---we find from (\ref{z_n-explicit}) that for $x\neq 0$, $N_R(z)=a_1=a(x)$ and 
\begin{equation}\label{F_{H_1}}
\F_R(z)=\F^{a_1}([0;a_1,a_2,\dots],[0;1,b_2,b_3,\dots])=([0;a_2,a_3,\dots],[0;1,a_1,b_2,b_3,\dots]).
\end{equation}
Notice the similarity between this induced map and the map $\mathcal{G}$ from (\ref{GaussNE}); they both act essentially as a two-sided shift on {\sc rcf}-expansions.  In fact, Brown and Yin proved in 1996 that a copy of the Gauss natural extension is found sitting (inverted, scaled and `suspended' from $y=1$) within $(\Omega,\mathcal{B},\bar\mu,\F)$ as the induced system $(R,\mathcal{B},\bar\mu_R,\F_R)$ with $R=H_1$ (Theorem 1 of \cite{BY1996}):
\begin{thm}[Brown--Yin, 1996 \cite{BY1996}]\label{H_1&Gauss}
The induced system $(R,\mathcal{B},\bar\mu_R,\F_R)$ with $R=H_1$ is isomorphic to the Gauss natural extension $(\Omega,\mathcal{B},\bar\nu_G,\mathcal{G})$. 
\end{thm}

Using Theorem \ref{H_1&Gauss}, one can exploit knowledge about the Gauss natural extension $(\Omega,\mathcal{B},\bar\nu_G,\mathcal{G})$ to infer properties of other induced systems $(R,\mathcal{B},\bar\mu_R,\F_R)$.  This is used, for instance, in the proof\footnote{The aforementioned $\bar\mu$-continuity condition assumed on inducible regions $R$ in \cite{DKS2023} is \emph{not} needed in the proof.} of Theorem 4.6 of \cite{DKS2023}, which states that the measure-theoretic entropy $h(\F_R)=h_{\bar\mu_R}(\F_R)$ of the induced system $(R,\mathcal{B},\bar\mu_R,\F_R)$ is inversely proportional to the $\bar\mu$-measure of $R$:

\begin{thm}[Dajani--Kraaikamp--Sanderson, 2025 \cite{DKS2023}]\label{entropy_thm}
For any proper, inducible subregion $R\subset\Omega$, 
\[h(\F_R)=\frac{\pi^2}{6\bar\mu(R)}.\]
\end{thm}

\begin{remark}
We remark here the striking resemblance between Theorem \ref{entropy_thm}, Remark 5.10 of \cite{K1991} on the entropy of $S$-expansions and Theorem 2 of \cite{KSS2012} (conjectured in Remark 2 of \cite{LM2008}) on the entropy of $\alpha$-{\sc cf}s.  From the results of \S\ref{S-expansions, revisited} and \S\ref{Nakada's alpha-continued fractions, revisited} below, Theorem \ref{entropy_thm} may be viewed as simultaneously extending these results from \cite{K1991,KSS2012}.  
\end{remark}

\section{Inducing contractions of the mother of all continued fractions}\label{Inducing contractions of the mother of all continued fractions}

We have seen in \S\ref{Inducing Ito's natural extension} that inducible subregions $R\subset\Omega$ naturally determine subsequences $(u_n^R/s_n^R)_{n\ge 0}=(u_{N_n^R}/s_{N_n^R})_{n\ge 0}$ of the convergents of Farey expansions.  In this section, we construct new {\sc gcf}-expansions whose convergents are precisely the subsequences $(u_n^R/s_n^R)_{n\ge 0}$ (Corollary \ref{PQ_in_us_cor}).  These {\sc gcf}s arise from a general procedure described in \S\ref{Contraction} called \emph{contraction}, which---under very mild assumptions---allows one to produce from a given {\sc gcf} a new {\sc gcf} whose convergents are any desired subsequence of the original convergents.  In \S\ref{Contracted Farey expansions}, we use induced transformations of Ito's natural extension of the Farey tent map (`the mother of all continued fractions') to govern contractions of Farey expansions.  In \S\ref{A two-sided shift for contracted Farey expansions} we introduce a dynamical system---isomorphic to the induced system $(R,\mathcal{B},\bar\mu_R,\F_R)$---which acts essentially as a two-sided shift on contracted Farey expansions.

\subsection{Contraction}\label{Contraction}

Recall the singularisation procedure discussed in \S\ref{S-expansions}: beginning with a {\sc srcf}-expansion, one may (simultaneously) singularise at possibly countably many positions to produce a new {\sc srcf}-expansion whose convergents are a subsequence of the original convergents.  However, singularisation at position $n$ is subject to the condition that the partial numerator $\alpha_{n+2}$ and partial denominator $\beta_{n+1}$ are both equal to $1$.  Moreover, beginning from a {\sc rcf}-expansion, this constraint on partial denominators implies that consecutive convergents cannot be removed via singularisation.  

In this subsection, we recall an old acceleration technique of Seidel from 1855 (\cite{S1855}; see also \cite{P1950}), called \emph{contraction}, which overcomes these obstacles.  Although our main interest is in producing {\sc gcf}-expansions whose convergents are subsequences of Farey convergents, we present contraction in the general, abstract setting of {\sc gcf}s discussed in \S\ref{Generalised continued fractions}, as we feel this technique has been largely overlooked\footnote{As mentioned in the introduction (\S\ref{Introduction-ch1}), contraction is used in the analytic theory of continued fractions, but usually only for subsequences of odd or even integers (\cite{LW2008}).  See also \cite{B2014}.} and can be fruitfully applied to other continued fraction algorithms.  For the same reason, we include a proof of Seidel's theorem (Theorem \ref{convergents_thm}) below.

\begin{defn}\label{ccf}
A {\sc gcf} $[\beta_0/\alpha_0;\alpha_1/\beta_1,\alpha_2/\beta_2,\dots]$ is called \emph{contractable} if $Q_{[m+1,n]}\neq 0$ for all $0\le m\le n$.  The \textit{contracted continued fraction} ({\sc ccf}) of a contractable {\sc gcf} $[\beta_0/\alpha_0;\alpha_1/\beta_1,\alpha_2/\beta_2,\dots]$ with respect to a strictly increasing sequence of non-negative integers $(n_k)_{k\ge 0}$ is the {\sc gcf} $[\beta_0'/\alpha_0';\alpha_1'/\beta_1',\alpha_2'/\beta_2',\dots],$
where
\[\begin{pmatrix}
\alpha_{k+1}'\\
\beta_{k+1}'
\end{pmatrix}
:=
\begin{pmatrix}
-\det(B_{[n_{k-1}+2,n_k+1]})Q_{[n_{k-2}+2,n_{k-1}]}Q_{[n_k+2,n_{k+1}]}\\
Q_{[n_{k-1}+2,n_{k+1}]}
\end{pmatrix},
\qquad k\ge -1,
\]
with $n_k:=k$ for $k<0$.
\end{defn}

\begin{remark}
The requirement that a contractable {\sc gcf} satisfies $Q_{[m+1,n]}\neq 0$ for all $0\le m\le n$ guarantees that the partial numerators $\alpha_{k+1}'$ are nonzero for all $k\ge -1$, i.e., that a {\sc ccf} is in fact a {\sc gcf} as defined in \S\ref{Generalised continued fractions}.  Indeed, we have $\det (B_{[n_{k-1}+2,n_k+1]})\neq 0$ for all $k\ge -1$ as noted in \S\ref{Generalised continued fractions}, and both $Q_{[n_{k-2}+2,n_{k-1}]}\neq 0,\ k\ge 1$, and $Q_{[n_k+2,n_{k+1}]}\neq 0,\ k\ge -1,$ by assumption.  Moreover, for $k=-1$ and $k=0$, $Q_{[n_{k-2}+2,n_{k-1}]}=1\neq 0$; see \eqref{rec_rels}. 

This requirement also guarantees that the scalars $c_k$ in Theorem \ref{convergents_thm} below are non-zero.  
\end{remark}

\begin{remark}
Notice that the partial numerators $\alpha_n'$ of a {\sc ccf} in general \emph{do not} satisfy $|\alpha_n'|=1$, even if this is true of the original {\sc gcf}.  Hence contraction does not necessarily send {\sc srcf}s to {\sc srcf}s.
\end{remark}

\begin{eg}\label{ccf_eg}
We compute the {\sc ccf} with respect to $(n_k)_{k\ge 0}=(2k)_{k\ge 0}$ of $[\beta_0/\alpha_0;\alpha_1/\beta_1,\alpha_2/\beta_2,\dots]=[1/1;1/2,2/3,3/4,\dots]$, i.e., $\alpha_0=1$ and $\alpha_n=\beta_{n-1}=n$ for all $n>0$.  We first note that the recurrence relations \eqref{rec_rels} and the fact that all of the partial numerators and partial denominators are positive imply that this {\sc gcf} is in fact contractable.  From Definition \ref{ccf},
\[\begin{pmatrix}
\alpha_{0}'\\
\beta_{0}'
\end{pmatrix}
=
\begin{pmatrix}
-\det(B_{[0,0]})Q_{[-1,-2]}Q_{[1,0]}\\
Q_{[0,0]}
\end{pmatrix}
=\begin{pmatrix}
\alpha_0\cdot 1\cdot 1\\
\beta_0
\end{pmatrix}
=\begin{pmatrix}
1\\
1
\end{pmatrix},
\]
\[\begin{pmatrix}
\alpha_1'\\
\beta_1'
\end{pmatrix}
=
\begin{pmatrix}
-\det(B_{[1,1]})Q_{[0,-1]}Q_{[2,2]}\\
Q_{[1,2]}
\end{pmatrix}
=\begin{pmatrix}
\alpha_1\cdot 1\cdot \beta_2\\
\beta_2\beta_1+\alpha_2
\end{pmatrix}
=\begin{pmatrix}
1\cdot 1\cdot 3\\
3\cdot 2+2
\end{pmatrix}
=\begin{pmatrix}
3\\
8
\end{pmatrix},
\]
\[\begin{pmatrix}
\alpha_2'\\
\beta_2'
\end{pmatrix}
=
\begin{pmatrix}
-\det(B_{[2,3]})Q_{[1,0]}Q_{[4,4]}\\
Q_{[2,4]}
\end{pmatrix}
=\begin{pmatrix}
-\alpha_2\alpha_3\cdot 1\cdot \beta_4\\
\beta_4(\beta_3\beta_2+\alpha_3)+\alpha_4\beta_2
\end{pmatrix}
=\begin{pmatrix}
-2\cdot 3\cdot 1\cdot 5\\
5(4\cdot 3+3)+4\cdot 3
\end{pmatrix}
=
\begin{pmatrix}
-30\\
87
\end{pmatrix},
\]
and for $k>1$,
\begin{align*}
\begin{pmatrix}
\alpha_{k+1}'\\
\beta_{k+1}'
\end{pmatrix}
&=
\begin{pmatrix}
-\det(B_{[2k,2k+1]})Q_{[2k-2,2k-2]}Q_{[2k+2,2k+2]}\\
Q_{[2k,2k+2]}
\end{pmatrix}\\
&=\begin{pmatrix}
-\alpha_{2k}\alpha_{2k+1}\cdot \beta_{2k-2}\cdot \beta_{2k+2}\\
\beta_{2k+2}(\beta_{2k+1}\beta_{2k}+\alpha_{2k+1})+\alpha_{2k+2}\beta_{2k}
\end{pmatrix}\\
&=\begin{pmatrix}
-\left((2k)(2k+1)\right)\cdot (2k-1)\cdot (2k+3)\\
(2k+3)((2k+2)(2k+1)+(2k+1))+(2k+2)(2k+1)
\end{pmatrix}\\
&=\begin{pmatrix}
-(2k-1)(2k)(2k+1)(2k+3)\\
(2k+1)((2k+3)^2+(2k+2))
\end{pmatrix}.
\end{align*}
The first few terms of the {\sc ccf} are thus 
\[[1/1;3/8,-30/87,-420/275,-1890/623,-5544/1179,-12870/1991,-25740/3107,\dots].\]
\end{eg}

Before proving that the convergents of a {\sc ccf} are a subsequence of the original {\sc gcf}-convergents, we need the following:

\begin{lem}\label{Qs_lem}
For any {\sc gcf},
\[Q_{[m+1,n]}=\frac{Q_nP_{m-1}-P_nQ_{m-1}}{\det B_{[-1,m]}} \quad \text{for all integers $-1\le m\le n$}.\]
\end{lem}
\begin{proof}
If $m=n$, then both the left- and right-hand sides of the expression equal 1. For $m<n$, $Q_{[m+1,n]}$ is the bottom-right entry of 
\[B_{[m+1,n]}=B_{[-1,m]}^{-1}B_{[-1,n]}=\frac1{\det B_{[-1,m]}}\begin{pmatrix}Q_m & -P_m \\ -Q_{m-1} & P_{m-1}\end{pmatrix}\begin{pmatrix}P_{n-1} & P_n \\ Q_{n-1} & Q_n\end{pmatrix},\]
which is evidently $\frac{Q_nP_{m-1}-P_nQ_{m-1}}{\det B_{[-1,m]}}$.
\end{proof}

Let $[\beta_0'/\alpha_0';\alpha_1'/\beta_1',\alpha_2'/\beta_2',\dots]$ be the {\sc ccf} of a contractable {\sc gcf} $[\beta_0/\alpha_0;\alpha_1/\beta_1,\alpha_2/\beta_2,\dots]$ with respect to $(n_k)_{k\ge 0}$.  For $-1\le m\le n$, let
\[\begin{pmatrix}P'_{[m,n-1]} & P'_{[m,n]}\\ Q'_{[m,n-1]} & Q'_{[m,n]}\end{pmatrix}=B'_{[m,n]}:=B_{[m,n]}([\beta_0'/\alpha_0';\alpha_1'/\beta_1',\alpha_2'/\beta_2',\dots]),\]
and when $m=-1$, set $P_n':=P_{[-1,n]}'$ and $Q_n':=Q_{[-1,n]}'$.

\begin{thm}[Seidel, 1855 \cite{S1855}]\label{convergents_thm}
With notation as above, 
\[
\begin{pmatrix}
P_k'\\
Q_k'
\end{pmatrix}=
c_k
\begin{pmatrix}
P_{n_k}\\
Q_{n_k}
\end{pmatrix},
\quad k\ge -2, \qquad \text{where}\quad  c_k=\prod_{j=0}^{k-1} Q_{[n_{j-1}+2,n_j]},
\]
with $n_k:=k$, $k<0$, and the product defining $c_k$ set equal to $1$ for $k<1$.  In particular, the {\sc ccf} with respect to $(n_k)_{k\ge 0}$ of a contractable {\sc gcf}-expansion $x=[\beta_0/\alpha_0;\alpha_1/\beta_1,\alpha_2/\beta_2,\dots]\in\mathbb{C}$ with convergents $(P_k/Q_k)_{k\ge -1}$ is a {\sc gcf}-expansion of $x$ with convergents $(P_{n_k}/Q_{n_k})_{k\ge -1}$.
\end{thm}
\begin{proof}
The proof of the first statement is by induction on $k$.  The statement trivially holds for $k<0$.  Now let $k+1\ge 0$ and suppose the statement is true for $k$ and $k-1$.  By the recurrence relations (\ref{rec_rels_classic}) and Definition \ref{ccf}---letting $U$ represent either $P$ or $Q$---we compute
\begin{align*}
U_{k+1}'=&\beta_{k+1}'U_{k}'+\alpha_{k+1}'U_{k-1}'\\
=&Q_{[n_{k-1}+2,n_{k+1}]}c_kU_{n_k}-\det(B_{[n_{k-1}+2,n_k+1]})Q_{[n_{k-2}+2,n_{k-1}]}Q_{[n_k+2,n_{k+1}]}c_{k-1}U_{n_{k-1}}\\
=&c_k\left(Q_{[n_{k-1}+2,n_{k+1}]}U_{n_k}-\det(B_{[n_{k-1}+2,n_k+1]})Q_{[n_k+2,n_{k+1}]}U_{n_{k-1}}\right).
\end{align*}
By Lemma \ref{Qs_lem},
\[Q_{[n_{k-1}+2,n_{k+1}]}=\frac{Q_{n_{k+1}}P_{n_{k-1}}-P_{n_{k+1}}Q_{n_{k-1}}}{\det B_{[-1,n_{k-1}+1]}} \qquad \text{and} \qquad Q_{[n_k+2,n_{k+1}]}=\frac{Q_{n_{k+1}}P_{n_k}-P_{n_{k+1}}Q_{n_k}}{\det B_{[-1,n_k+1]}},\]
so the above computation gives
\begin{align*}
U_{k+1}'=&\frac{c_k\left((Q_{n_{k+1}}P_{n_{k-1}}-P_{n_{k+1}}Q_{n_{k-1}})U_{n_k}-(Q_{n_{k+1}}P_{n_k}-P_{n_{k+1}}Q_{n_k})U_{n_{k-1}}\right)}{\det B_{[-1,n_{k-1}+1]}}.
\end{align*}
For both $U=P$ and $U=Q$, the numerator of the previous expression simplifies to 
\[c_k(Q_{n_k}P_{n_{k-1}}-P_{n_k}Q_{n_{k-1}})U_{n_{k+1}}.\]
Using Lemma \ref{Qs_lem} once more, we have
\[U_{k+1}'=c_k\frac{Q_{n_k}P_{n_{k-1}}-P_{n_k}Q_{n_{k-1}}}{\det B_{[-1,n_{k-1}+1]}}U_{n_{k+1}}=c_k Q_{[n_{k-1}+2,n_k]}U_{n_{k+1}}=c_{k+1}U_{n_{k+1}}.
\]

This proves the first statement.  The second statement follows from the first, since for a contractable {\sc gcf}, $Q_{[n_{j-1}+2,n_j]}\neq 0$ for all $j\ge 0$ implies $c_k\neq 0$ for all $k\ge -1$, and $x=\lim_{k\to\infty}\frac{P_k}{Q_k}=\lim_{k\to\infty}\frac{P_{n_k}}{Q_{n_k}}.$
\end{proof}

\begin{eg}
Continuing with Example \ref{ccf_eg}, for $[\beta_0/\alpha_0;\alpha_1/\beta_1,\alpha_2/\beta_2,\dots]=[1/1;1/2,2/3,3/4,\dots]$ one computes
{\small
\[\left(\begin{pmatrix}
P_n\\
Q_n
\end{pmatrix}\right)_{n\ge 0}=\left(\begin{pmatrix}
1\\
1
\end{pmatrix},
\begin{pmatrix}
3\\
2
\end{pmatrix},
\begin{pmatrix}
11\\
8
\end{pmatrix},
\begin{pmatrix}
53\\
38
\end{pmatrix},
\begin{pmatrix}
309\\
222
\end{pmatrix},
\begin{pmatrix}
2119\\
1522
\end{pmatrix},
\begin{pmatrix}
16687\\
11986
\end{pmatrix},
\begin{pmatrix}
148329\\
106542
\end{pmatrix},
\begin{pmatrix}
1468457\\
1054766
\end{pmatrix},
\begin{pmatrix}
16019531\\
11506538
\end{pmatrix},\dots
\right).\]
}
For the {\sc ccf} 
{\small
\[[\beta_0'/\alpha_0';\alpha_1'/\beta_1',\alpha_2'/\beta_2',\dots]=[1/1;3/8,-30/87,-420/275,-1890/623,-5544/1179,-12870/1991,-25740/3107,\dots]\]
}
of $[1/1;1/2,2/3,3/4,\dots]$ with respect to $(2k)_{k\ge 0}$, we find
{\small
\[\left(\begin{pmatrix}
P_k'\\
Q_k'
\end{pmatrix}\right)_{k\ge 0}=\left(\begin{pmatrix}
1\\
1
\end{pmatrix},\begin{pmatrix}
11\\
8
\end{pmatrix},
3\begin{pmatrix}
309\\
222
\end{pmatrix},
3\cdot 5\begin{pmatrix}
16687\\
11986
\end{pmatrix},
3\cdot 5\cdot 7\begin{pmatrix}
1468457\\
1054766
\end{pmatrix},
3\cdot 5\cdot 7\cdot 9\begin{pmatrix}
190899411\\
137119578
\end{pmatrix},
\dots
\right).\]
}
As fractions, $P_k'/Q_k'=P_{2k}/Q_{2k}$ for $k\ge 0$.
\end{eg}

\subsection{Contracted Farey expansions}\label{Contracted Farey expansions}

Throughout this subsection, $R\subset\Omega$ is an inducible subregion and $z=(x,y)\in\Omega$ with $x$ irrational.  Using notation from \S\ref{Generalised continued fractions} and \S\ref{Farey expansions and Farey convergents}, let $[\beta_0/\alpha_0;\alpha_1/\beta_1,\alpha_2/\beta_2,\dots]$ denote the Farey expansion of $x$, i.e., $\alpha_0=1$, $\beta_0=1-\varepsilon_1$, and for $n>0$, $\alpha_n=2\varepsilon_n-1$ and $\beta_n=2-\varepsilon_{n+1}$; see \eqref{farey_expn}.  Below, we shall perform contraction on Farey expansions, but first we must show that Farey expansions are in fact contractable.

\begin{prop}
The Farey expansion of an irrational $x\in(0,1)$ is contractable.  
\end{prop}  
\begin{proof}
We must show that $Q_{[m+1,n]}\neq 0$ for any $0\le m\le n$.  By Lemma \ref{Qs_lem}, this is equivalent to $Q_nP_{m-1}\neq P_nQ_{m-1}$ for all $0\le m\le n$.  By Proposition \ref{FareyConvs}, the Farey convergents $(P_j/Q_j)_{j\ge -1}=(u_j/s_j)_{j\ge 0}$ are all {\sc rcf}-convergents and mediants, which are distinct by \eqref{meds-mon-odd} and \eqref{meds-mon-even}.  
\end{proof}

Recall the definition of $N_n^R=N_n^R(z)$ from \eqref{N_n^R}.

\begin{defn}\label{cfe}
The \textit{contracted Farey expansion} of $x$ with respect to $R$ and $z=(x,y)$, denoted 
\[[\beta_0^R/\alpha_0^R;\alpha_1^R/\beta_1^R,\alpha_2^R/\beta_2^R,\dots]=[\beta_0^R(z)/\alpha_0^R(z);\alpha_1^R(z)/\beta_1^R(z),\alpha_2^R(z)/\beta_2^R(z),\dots],\]
is the {\sc ccf} of the Farey expansion of $x$ with respect to $(n_k)_{k\ge 0}$, where $n_k:=N_{k+1}^R-1,\ k\ge 0$.  If $z=(x,1)$, we call this the \emph{contracted Farey expansion} of $x$ with respect to $R$.
\end{defn}
\begin{remark}
Using \eqref{z_n-explicit}, one can show that for any $z=(x,y)$ and $z'=(x',y')$ in $\Omega$ with $x=x'$, $|\F^n(z)-\F^n(z')|\to 0$.  Using this and the fact that $(\Omega,\mathcal{B},\bar\mu,\F)$ is conservative and ergodic (\cite{BY1996}), one finds that for any inducible $R\subset\Omega$ with $\bar\mu(\text{int}(R))>0$, the forward $\F$-orbit of $(x,1)$ enters $R$ infinitely often for Lebesgue--a.e.~$x\in (0,1)$; see Remark 4.3 of \cite{DKS2023}.  Hence, for such $R$, the contracted Farey expansion of $x$ with respect to $R$ exists for Lebesgue--a.e.~$x\in (0,1)$. 
\end{remark}

In order to study contracted Farey expansions using the dynamics of $\F_R$, we wish to understand these expansions and their convergents in terms of entries from the matrices $A_R(z)$ and $A_{[m,n]}^R(z)$ from \eqref{A_R(z)} and \eqref{A_{[m,n]}^R} rather than $B_{[m,n]}$.  To this end, we begin with a lemma.  Let $z_n^R=\F_R^n(z)$ and recall that $s_R(z)$ and $s_n^R(z)$ denote the bottom-left entries of the matrices $A_R(z)$ and $A_{[0,n]}^R(z)$, respectively; see \eqref{A_R(z)} and \eqref{A_n^R(z)}.
\begin{lem}\label{Qs_and_ss}
For any $z\in\Omega$ and $0\le j<k,$ one has $Q_{[N_j^R+1,N_k^R-1]}=s_{k-j}^R(z_j^R)$.  In particular, if $k=j+1$, then $Q_{[N_j^R+1,N_{j+1}^R-1]}=s_R(z_j^R)$.
\end{lem}
\begin{proof}
First, notice that for any $n>0$,
\begin{equation}\label{detA=detB}
\det A_{[0,n]}=\det\left(A_{\varepsilon_1}\cdots A_{\varepsilon_n}\right)=\prod_{j=1}^{n}(1-2\varepsilon_j)=\det (B_{-1}B_0\cdots B_n)=\det B_{[-1,n]},
\end{equation}
and equality of the left- and right-hand sides also holds for $n=0$ since both sides equal $1$.  Then by Lemma \ref{Qs_lem}, Proposition \ref{FareyConvs}, and Equations \eqref{supr_not} and \eqref{A_n^R_entries},
\[Q_{[N_j^R+1,N_k^R-1]}=\frac{Q_{N_k^R-1}P_{N_j^R-1}-P_{N_k^R-1}Q_{N_j^R-1}}{\det B_{[-1,N_j^R]}}=\frac{s_{N_k^R}u_{N_j^R}-u_{N_k^R}s_{N_j^R}}{\det A_{[0,N_j^R]}}=\frac{s_k^Ru_j^R-u_k^Rs_j^R}{\det A_{[0,j]}^R},\]
where $A_{[0,j]}^R=A_{[0,j]}^R(z)$.  For the first statement, it suffices to show that the right-hand side of the previous line equals the bottom-left entry of $A_{[0,k-j]}^R(z_j^R)$.  From (\ref{A_m(z_n)=A_{m+n}(z)}), (\ref{A_{[m,n]}^R}) and (\ref{supr_not}), we have
\begin{align*}
A_{[0,k-j]}^R(z_j^R)=&A_1^R(z_j^R)A_2^R(z_j^R)\cdots A_{k-j}^R(z_j^R)\\
=&A_{j+1}^R(z)A_{j+2}^R(z)\cdots A_k^R(z)\\
=&\left(A_{[0,j]}^{R}\right)^{-1}A_{[0,k]}^R=\frac1{\det A_{[0,j]}^R}\begin{pmatrix}
r_j^R & -t_j^R\\
-s_j^R & u_j^R
\end{pmatrix}\begin{pmatrix}
u_k^R & t_k^R\\
s_k^R & r_k^R
\end{pmatrix},
\end{align*}
and the first statement follows.  The second statement follows immediately from the first and the fact that $A_{[0,1]}^R(z_j^R)=A_1^R(z_j^R)=A_R(z_j^R)$; see \eqref{A_m(z_n)=A_{m+n}(z)} and \eqref{A_{[m,n]}^R}.
\end{proof}

Now let $[\beta_0^R/\alpha_0^R; \alpha_1^R/\beta_1^R,\alpha_2^R/\beta_2^R,\dots]$ be the contracted Farey expansion of $x$ with respect to $R$ and $z$, and for $-1\le m\le n$, let 
\begin{equation}\label{B_{[m,n]}^R}
\begin{pmatrix}
P_{[m,n-1]}^R & P_{[m,n]}^R\\
Q_{[m,n-1]}^R & Q_{[m,n]}^R
\end{pmatrix}
=B_{[m,n]}^R:=B_{[m,n]}([\beta_0^R/\alpha_0^R; \alpha_1^R/\beta_1^R,\alpha_2^R/\beta_2^R,\dots])
\end{equation}
and
\begin{equation}\label{P_n^R&Q_n^R}
P_n^R:=P_{[-1,n]}^R, \qquad Q_n^R:=Q_{[-1,n]}^R. 
\end{equation}
Then Theorem \ref{convergents_thm}, Proposition \ref{FareyConvs} and Lemma \ref{Qs_and_ss} imply:
\begin{cor}\label{PQ_in_us_cor}
With notation as above, 
\[
\begin{pmatrix}
P_k^R\\
Q_k^R
\end{pmatrix}
=c_k^R\begin{pmatrix}
u_{k+1}^R\\
s_{k+1}^R
\end{pmatrix}, \quad k\ge -1, \qquad \text{where} \quad c_k^R=\prod_{j=0}^{k-1} s_R(z_j^R),\]
with $c_k^R=1$ for $k<1$.  In particular, the contracted Farey expansion of $x$ with respect to $R$ and $z=(x,y)$ has convergents $(u_k^R/s_k^R)_{k
\ge 0}.$
\end{cor}

Corollary \ref{PQ_in_us_cor} describes the convergents of a contracted Farey expansion in terms of the entries of $A_{[0,n]}^R$ (see \eqref{supr_not}).  Proposition \ref{cfe_alt} below gives an alternative description of the partial numerators and partial denominators $\alpha_k^R,\ \beta_k^R$ in terms of entries of the matrices $A_R(z)$ (see (\ref{A_R(z)})).  For this, we introduce three integer-valued maps on $\Omega$: let $d_R, \alpha_R,\beta_R: \Omega\to\mathbb{Z}$ be defined for $z\in\Omega$ by
\[d_R(z):=\begin{cases}
s_R(\F_R^{-1}(z)), & \text{if $\F_R^{-1}(z)$ is defined},\\
1, & \text{otherwise},
\end{cases}\]
\begin{equation}\label{alpha_R}
\alpha_R(z):=-\det(A_R(z))d_R(z)s_R(\F_R(z))
\end{equation}
and
\begin{equation}\label{beta_R}
\beta_R(z):=s_R(z)u_R(\F_R(z))+r_R(z)s_R(\F_R(z)).
\end{equation}

\begin{remark}\label{d_R(x,1)=1}
Notice that Lemma \ref{s_R>0} implies $d_R(z)$ is a positive integer for any $z$.  We claim that $d_R(z)=1$ whenever $z=(x,1)$.  By definition of $\F$ (Equation (\ref{FareyNEeqn})), $\F^{-1}(z)\in (1/2,1]\times\{0\}$, and $\F^{-n}(z)\in[0,1/2]\times\{0\}$ for $n>1$.  If $\F_R^{-1}(z)$ is not defined (e.g., if $R$ does not intersect the line $[0,1]\times \{0\}$), then $d_R(z)=1$ by definition.  Otherwise, $\F_R^{-1}(z)=\F^{-n}(z)$ for some $n\ge 1$, and
\[A_R(\F_R^{-1}(z))=A_0^{n-1}A_1=\begin{pmatrix}0 & 1\\ 1 & n\end{pmatrix}\]
implies $d_R(z)=s_R(\F_R^{-1}(z))=1$.  In either case, $d_R(z)=1$ for $z=(x,1)$ as claimed.

One motivation for defining $d_R(z)$ as above lies in the second statement of the following proposition: when $d_R(z)=1$, the notation simplifies and we need not consider the index $1$ partial numerator $\alpha_1^R$ as a separate case.  
\end{remark}

\begin{prop}\label{cfe_alt}
The digits of the contracted Farey expansion of $x$ with respect to $R$ and $z=(x,y)$ are given by
\[
\begin{pmatrix}
\alpha_0^R\\
\beta_0^R
\end{pmatrix}
=
\begin{pmatrix}
s_R(z_0^R)\\
u_R(z_0^R)
\end{pmatrix},
\qquad
\begin{pmatrix}
\alpha_1^R\\
\beta_1^R
\end{pmatrix}
=
\begin{pmatrix}
\alpha_R(z_0^R)/d_R(z_0^R)\\
\beta_R(z_0^R)
\end{pmatrix}
\qquad
\text{and}
\qquad
\begin{pmatrix}
\alpha_{k+1}^R\\
\beta_{k+1}^R
\end{pmatrix}
=
\begin{pmatrix}
\alpha_R(z_k^R)\\
\beta_R(z_k^R)
\end{pmatrix}, 
\quad k> 0.
\]
When $d_R(z)=1$ (e.g., when $z=(x,1)$ by Remark \ref{d_R(x,1)=1}), this becomes
\[
\begin{pmatrix}
\alpha_0^R\\
\beta_0^R
\end{pmatrix}
=
\begin{pmatrix}
s_R(z_0^R)\\
u_R(z_0^R)
\end{pmatrix}
\qquad
\text{and}
\qquad
\begin{pmatrix}
\alpha_{k+1}^R\\
\beta_{k+1}^R
\end{pmatrix}
=
\begin{pmatrix}
\alpha_R(z_k^R)\\
\beta_R(z_k^R)
\end{pmatrix}, 
\quad k\ge 0.
\]
\end{prop}
\begin{proof}
From Definitions \ref{ccf} and \ref{cfe}, we have 
\begin{equation}\label{alphak_betak_lem}
\begin{pmatrix}
\alpha_{k+1}^R\\
\beta_{k+1}^R
\end{pmatrix}
=
\begin{pmatrix}
-\det(B_{[N_k^R+1,N_{k+1}^R]})Q_{[N_{k-1}^R+1,N_k^R-1]}Q_{[N_{k+1}^R+1,N_{k+2}^R-1]}\\
Q_{[N_k^R+1,N_{k+2}^R-1]}
\end{pmatrix}, \quad k\ge -1,
\end{equation}
where $N_k^R:=k$ for $k<0$.  For $k=-1$, 
\[
\begin{pmatrix}
\alpha_0^R\\
\beta_0^R
\end{pmatrix}
=
\begin{pmatrix}
-\det(B_{[0,0]})Q_{[-1,-2]}Q_{[N_0^R+1,N_1^R-1]}\\
Q_{[0,N_1^R-1]}
\end{pmatrix}
=
\begin{pmatrix}
Q_{[N_0^R+1,N_1^R-1]}\\
Q_{[0,N_1^R-1]}
\end{pmatrix}.
\]
Lemma \ref{Qs_and_ss} gives $Q_{[N_0^R+1,N_1^R-1]}=s_R(z_0^R)$, and Lemma \ref{Qs_lem} and Proposition \ref{FareyConvs} give
\[Q_{[0,N_1^R-1]}=\frac{Q_{N_1^R-1}P_{-2}-P_{N_1^R-1}Q_{-2}}{\det B_{[-1,-1]}}=P_{N_1^R-1}=u_1^R=u_R(z_0^R).\]
Thus the claim holds for $\alpha_0^R,\ \beta_0^R$.  Next, notice from Equation (\ref{detA=detB}) that for any $0\le m< n$, 
\[\det B_{[m+1,n]}=\frac{\det B_{[-1,n]}}{\det B_{[-1,m]}}=\frac{\det A_{[0,n]}}{\det A_{[0,m]}}=\det(A_{\varepsilon_{m+1}}\cdots A_{\varepsilon_n}),
\]
so by \eqref{A_n^R}, we have for $k\ge 0$ that
\[\det(B_{[N_k^R+1,N_{k+1}^R]})=\det(A_{\varepsilon_{N_k^R+1}}\cdots A_{\varepsilon_{N_{k+1}^R}})
=\det(A_R(z_k^R)).\]
This, Equation (\ref{alphak_betak_lem}) and Lemma \ref{Qs_and_ss} give for $k\ge 0$
\begin{align*}
\begin{pmatrix}
\alpha_{k+1}^R\\
\beta_{k+1}^R
\end{pmatrix}
=&
\begin{pmatrix}
-\det(B_{[N_k^R+1,N_{k+1}^R]})Q_{[N_{k-1}^R+1,N_k^R-1]}Q_{[N_{k+1}^R+1,N_{k+2}^R-1]}\\
Q_{[N_k^R+1,N_{k+2}^R-1]}
\end{pmatrix}\\
=&\begin{cases}
\begin{pmatrix}
-\det(A_R(z_0^R))s_R(z_1^R)\\
s_2^R(z_0^R)
\end{pmatrix}, & k=0,\\
\begin{pmatrix}
-\det(A_R(z_k^R))s_R(z_{k-1}^R)s_R(z_{k+1}^R)\\
s_2^R(z_k^R)
\end{pmatrix}, & k>0.
\end{cases}
\end{align*}
When $k=0$, this gives $\alpha_1^R=\alpha_R(z_0^R)/d_R(z_0^R)$.  When $k>0$, $s_R(z_{k-1}^R)=s_R(\F_R^{-1}(z_k^R))=d_R(z_k)$, so $\alpha_{k+1}^R=\alpha_R(z_k^R)$.  Moreover, for $k\ge 0$, $s_2^R(z_k^R)$ is the bottom-left entry of
\begin{align*}
A_{[0,2]}^R(z_k^R)=A_1^R(z_k^R)A_2^R(z_k^R)=A_R(z_k^R)A_R(z_{k+1}^R)=\begin{pmatrix}u_R(z_k^R) & t_R(z_k^R) \\ s_R(z_k^R) & r_R(z_k^R)\end{pmatrix}\begin{pmatrix}u_R(z_{k+1}^R) & t_R(z_{k+1}^R) \\ s_R(z_{k+1}^R) & r_R(z_{k+1}^R)\end{pmatrix};
\end{align*}
see (\ref{A_n^R(z)}), (\ref{A_{[m,n]}^R}), (\ref{A_m(z_n)=A_{m+n}(z)}) and (\ref{A_R(z)}).  Thus $s_2^R(z_k^R)=s_R(z_k^R)u_R(z_{k+1}^R)+r_R(z_k^R)s_R(z_{k+1}^R)=\beta_R(z_k^R)$.  This proves the first statement.  The latter statement follows immediately from the first.
\end{proof}

We refer the reader to \S\ref{Examples of contracted Farey expansions} below for examples using Proposition \ref{cfe_alt}.

\subsection{A two-sided shift for contracted Farey expansions}\label{A two-sided shift for contracted Farey expansions}

In this subsection, we associate to the induced system $(R,\mathcal{B},\bar\mu_R,\F_R)$ an isomorphic dynamical system $(\Omega_R,\mathcal{B},\bar\nu_R,\tau_R)$ acting essentially as a two-sided shift for contracted Farey expansions.  This new system will serve several purposes in \S\ref{Examples of contracted Farey expansions} below: we will see that $(\Omega_{H_1},\mathcal{B},\bar\nu_{H_1},\tau_{H_1})=(\Omega,\mathcal{B},\bar\nu_G,\G)$ is the natural extension of the Gauss map; for certain subregions $R\subset H_1$, $(\Omega_R,\mathcal{B},\bar\nu_R,\tau_R)$ will coincide with a two-sided shift system associated to $S$-expansions in \cite{K1991}; and in \S\ref{Nakada's alpha-continued fractions, revisited}, we describe the natural extension of each of Nakada's $\alpha$-{\sc cf}s, $0<\alpha\le 1$, as an induced system $(R,\mathcal{B},\bar\mu_R,\F_R)$ by using the isomorphic system $(\Omega_R,\mathcal{B},\bar\nu_R,\tau_R)$.

To ease exposition, we impose some restrictions on our inducible subregion $R\subset\Omega$ throughout this subsection.  First, we assume that $R$ is bounded away from the origin and that for any $z=(x,y)\in R$, $y>0$.  Furthermore, we assume that $s_R(z)=1$ for all $z\in R$, and hence---by Lemma \ref{s_R>0}---that $u_R(z)\in\{0,1\}$.  The regions $R$ considered in \S\ref{Examples of contracted Farey expansions} below shall satisfy these assumptions.  

\begin{remark}
A two-sided shift space may be constructed without the restriction $s_R(z)=1$, but in general the domain $\Omega_R$ consists of several planar `sheets,' and the invariant measure $\bar\nu_R$ is a sum of measures which---restricted to each of these sheets---has density of the form in Theorem \ref{two_shift_map_and_measure} below.  However, this more general system is not needed for our purposes.  
\end{remark}

Define $\varphi_R:R\to\mathbb{R}^2$, where for $z=(x,y)\in R$, 
\begin{align}\label{varphi_R}
\varphi_R(z)=(X(z),Y(z)):=&\left(\begin{pmatrix}1 & -u_R(z)\\ 0 & 1\end{pmatrix}\cdot x,\ \begin{pmatrix}-1 & 1\\ 1-u_R(z) & u_R(z)\end{pmatrix}\cdot y\right) \\
=&\begin{cases}
\left(x,\frac{1-y}{y}\right), & u_R(z)=0,\notag \\
\left(x-1,1-y\right), & u_R(z)=1.
\end{cases}
\end{align}

\begin{figure}[t]
\centering
	\includegraphics[width=.25\textwidth]{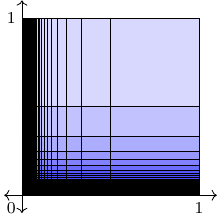}
	\hspace{30pt}
	\raisebox{.006\height}{\includegraphics[width=.45\textwidth]{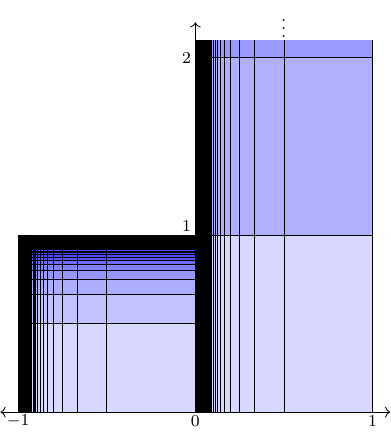}}
\caption{Left: The domain $\Omega=[0,1]^2$.  Right: The first quadrant shows part of the image of $\Omega\backslash ([0,1]\times \{0\})$ under the map $(x,y)\mapsto (x,(1-y)/y)$; the second quadrant shows the image of $\Omega$ under the map $(x,y)\mapsto(x-1,1-y)$.}
\label{varphi_fig}
\end{figure}

The map $\varphi_R$ is injective, except possibly on the null-set of points $\{(x,y)\in R\ |\ x\in\{0,1\}\}$; see Figure \ref{varphi_fig}.  Setting $\Omega_R:=\varphi_R(R)$, its inverse (off of the image of the aforementioned null-set) $\varphi_R^{-1}:\Omega_R\setminus(\{0\}\times [0,\infty))\to R$ is given by
\begin{equation*}
\varphi_R^{-1}(X,Y)=\begin{cases}
\left(X,\frac1{Y+1}\right), & X>0,\\
(X+1,1-Y), & X<0.
\end{cases}
\end{equation*}
If $z=\varphi_R^{-1}(X,Y)$, this may also be written
\begin{equation}\label{varphi_R_inv}
\varphi_R^{-1}(X,Y)=\left(\begin{pmatrix}1 & u_R(z)\\ 0 & 1\end{pmatrix}\cdot X,\ \begin{pmatrix}u_R(z) & -1\\ u_R(z)-1 & -1\end{pmatrix}\cdot Y\right)=\left(X+u_R(z),\frac{u_R(z) Y-1}{(u_R(z)-1)Y-1}\right).
\end{equation}
Define $\tau_R:\Omega_R\to \Omega_R$ by 
\begin{equation*}
\tau_R(X,Y):=\begin{cases}
(X,Y), & X=0,\\
\varphi_R\circ \F_R\circ\varphi_R^{-1}, & X\neq 0.
\end{cases}
\end{equation*}
We obtain a dynamical system $(\Omega_R,\mathcal{B},\bar\nu_R,\tau_R)$, where $\bar\nu_R:=\bar\mu_R\circ \varphi_R^{-1}$ denotes the pushforward measure of $\bar\mu_R$ under $\varphi_R$.  By construction, $(R,\mathcal{B},\bar\mu_R,\F_R)$ and $(\Omega_R,\mathcal{B},\bar\nu_R,\tau_R)$ are isomorphic.  Recall the definitions of $\alpha_R$ and $\beta_R$ from \eqref{alpha_R} and \eqref{beta_R}.

\begin{thm}\label{two_shift_map_and_measure}
The map $\tau_R:\Omega_R\to \Omega_R$ is given by $\tau_R(0,Y)=(0,Y)$ and for $X\neq 0$,
\[\tau_R(X,Y)=\left(\frac{\alpha_R(z)}{X}-\beta_R(z),\frac1{\beta_R(z)+\alpha_R(z)Y}\right),\]
where $z=\varphi_R^{-1}(X,Y)$, and the measure $\bar\nu_R$ has density
\[\frac{1}{\bar\mu(R)(1+XY)^2}.\]
\end{thm}

\begin{remark}
We remark here the resemblance between the measures and maps from $(\Omega_R,\mathcal{B},\bar\nu_R,\tau_R)$ and the natural extension $(\Omega,\mathcal{B},\bar\nu_G,\G)$ of the Gauss map from \S\ref{The Gauss map}.  We shall return to this point in \S\ref{Regular continued fractions, revisited} below.  
\end{remark}

\begin{proof}[Proof of Theorem \ref{two_shift_map_and_measure}]
We begin with the statement about the map $\tau_R$.  By definition, $\tau_R(0,Y)=(0,Y),$ so let $(X,Y)\in \Omega_R$ with $X\neq 0$.  Set $(X',Y'):=\tau_R(X,Y),\ z=(x,y):=\varphi_R^{-1}(X,Y)$ and $z'=(x',y'):=\F_R(z)$, and note that $(X',Y')=\varphi_R(z')$.  Set $u=u_R(z)$ and $u'=u_R(z')$.  Using Equations \eqref{varphi_R}, (\ref{F_R2}), (\ref{varphi_R_inv}), and symmetry of $A_1$, respectively,
\begin{align*}
(X',Y')=&\left(\begin{pmatrix}1 & -u'\\ 0 & 1\end{pmatrix}\cdot x',\ \begin{pmatrix}-1 & 1\\ 1-u' & u'\end{pmatrix}\cdot y'\right)\notag\\
=&\left(\begin{pmatrix}1 & -u'\\ 0 & 1\end{pmatrix}A_R^{-1}(z)\cdot x,\ \begin{pmatrix}-1 & 1\\ 1-u' & u'\end{pmatrix}A_1A_R^{T}(z)A_1^{-1}\cdot y\right)\notag\\
=&\left(\begin{pmatrix}1 & -u'\\ 0 & 1\end{pmatrix}A_R^{-1}(z)\begin{pmatrix}1 & u\\ 0 & 1\end{pmatrix}\cdot X,\ \begin{pmatrix}-1 & 1\\ 1-u' & u'\end{pmatrix}A_1A_R^{T}(z)A_1^{-1}\begin{pmatrix}u & -1\\ u-1 & -1\end{pmatrix}\cdot Y\right)\notag\\
=&\left(\begin{pmatrix}1 & -u'\\ 0 & 1\end{pmatrix}A_R^{-1}(z)\begin{pmatrix}1 & u\\ 0 & 1\end{pmatrix}\cdot X,\ \left(\begin{pmatrix}-1 & 1\\ 1-u' & u'\end{pmatrix}^{-T}A_1^{-1}A_R^{-1}(z)A_1\begin{pmatrix}u & -1\\ u-1 & -1\end{pmatrix}^{-T}\right)^{-T}\cdot Y\right).
\end{align*}
One easily computes
\[\begin{pmatrix}-1 & 1\\ 1-u' & u'\end{pmatrix}^{-T}A_1^{-1}=\begin{pmatrix}1 & -u' \\ 0 & 1\end{pmatrix} \qquad \text{and} \qquad A_1\begin{pmatrix}u & -1\\ u-1 & -1\end{pmatrix}^{-T}=-\begin{pmatrix}1 & u \\ 0 & 1\end{pmatrix},\]
so $(X',Y')=(M\cdot X,M^{-T}\cdot Y)$, where (recall $u=u_R(z)$, $u'=u_R(z')$ and $s_R(z)=s_R(z')=1$)
\begin{align*}
M=&\begin{pmatrix}1 & -u'\\ 0 & 1\end{pmatrix}A_R^{-1}(z)\begin{pmatrix}1 & u\\ 0 & 1\end{pmatrix}\\
=&\frac1{\det A_R(z)}\begin{pmatrix}1 & -u'\\ 0 & 1\end{pmatrix}\begin{pmatrix}r_R(z) & -t_R(z)\\ -s_R(z) & u_R(z)\end{pmatrix}\begin{pmatrix}1 & u\\ 0 & 1\end{pmatrix}\\
=&\frac1{\det A_R(z)}\begin{pmatrix}r_R(z)+s_R(z)u' & (r_R(z)u-t_R(z))-u'(u_R(z)-s_R(z)u)\\ -s_R(z) & u_R(z)-s_R(z)u\end{pmatrix}\\
=&\frac1{\det A_R(z)}\begin{pmatrix}r_R(z)+s_R(z)u' & r_R(z)u-t_R(z)\\ -s_R(z) & 0\end{pmatrix}\\
=&\frac{-1}{\det A_R(z)}\begin{pmatrix}-\left(r_R(z)s_R(z')+s_R(z)u_R(z')\right) & -\det(A_R(z))\\ 1 & 0\end{pmatrix}\\
=&\frac{-1}{\det A_R(z)}\begin{pmatrix}-\beta_R(z) & \alpha_R(z)\\ 1 & 0\end{pmatrix}.
\end{align*}
Thus
\[
\tau_R(X,Y)=(X',Y')=(M\cdot X,M^{-T}\cdot Y)=\left(\begin{pmatrix}-\beta_R(z) & \alpha_R(z)\\ 1 & 0\end{pmatrix}\cdot X,\begin{pmatrix}0 & 1\\ \alpha_R(z) & \beta_R(z)\end{pmatrix}\cdot Y\right),
\]
proving the claim about $\tau_R$.  

Next we prove the statement about the density of $\bar\nu_R$.  Let $S$ be a measurable subset of $\Omega_R$.  Using a change of variables,
\begin{align*}
\bar\nu_R(S)=&\bar\mu_R\circ\varphi_R^{-1}(S)=\int_{\varphi_R^{-1}(S)}d\bar\mu_R=\frac{1}{\bar\mu(R)}\iint_{\varphi_R^{-1}(S)}\rho(x,y)dxdy\\
=&\frac{1}{\bar\mu(R)}\iint_{S}\rho(\varphi_R^{-1}(X,Y))|\det J|dXdY,
\end{align*}
where
\[\rho(x,y):=\frac{1}{(x+y-xy)^2}\]
is the density of $\bar\mu$ and $J$ is the Jacobian of $\varphi_R^{-1}$ at $(X,Y)\in S$.  Let $u=u_R(z)\in\{0,1\}$, where $z=\varphi_R^{-1}(X,Y)$.  By Equation (\ref{varphi_R_inv}), the Jacobian of $\varphi_R^{-1}$ at $(X,Y)$ is
\[J=\begin{pmatrix}1 & 0 \\ 0 & \frac{u((u-1)Y-1)-(u-1)(u Y-1)}{((u-1)Y-1)^2}\end{pmatrix}=\begin{pmatrix}1 & 0 \\ 0 & \frac{-1}{((u-1)Y-1)^2}\end{pmatrix}.\]
Moreover,
\begin{align*}
\rho(\varphi_R^{-1}(X,Y))=&\left((X+u)+\frac{u Y-1}{(u-1)Y-1}-(X+u)\frac{u Y-1}{(u-1)Y-1}\right)^{-2}\\
=&\left(\frac{(X+u)((u-1)Y-1)+(u Y-1)-(X+u)(u Y-1)}{((u-1)Y-1)}\right)^{-2}\\
=&\left(\frac{1+XY}{(u-1)Y-1}\right)^{-2}\\
\end{align*}
so that
\[\rho(\varphi_R^{-1}(X,Y))|\det J|=\left(\frac{(u-1)Y-1}{1+XY}\right)^2\frac{1}{((u-1)Y-1)^2}=\frac1{(1+XY)^2}.\]
\end{proof}

For given $(X,Y)=(X(z),Y(z))\in \Omega_R$, set
\begin{equation}\label{XnRYnRunRsnR}
(X_n^R,Y_n^R)=(X_n^R(z),Y_n^R(z)):=\tau_R^n(X,Y), \quad n\ge 0.
\end{equation}
Then, for $X_n^R\neq 0$,
\begin{equation}\label{zn=varphi_inv_XnYN}
z_n^R=\F_R^n(z)=\F_R^n\circ \varphi_R^{-1}(X,Y)=\varphi_R^{-1}\circ \tau_R^{n}(X,Y)=\varphi_R^{-1}(X_n^R,Y_n^R), \quad n\ge 0.
\end{equation}
The next result states that the map $\tau_R$ acts essentially as a two-sided shift operator on contracted Farey expansions.

\begin{prop}\label{two_shift_prop}
Let $[\beta_0^R/\alpha_0^R;\alpha_1^R/\beta_1^R,\alpha_2^R/\beta_2^R,\dots]$ denote the contracted Farey expansion of $x\in(0,1)\backslash \mathbb{Q}$ with respect to $R$ and $z=(x,y)\in R$.  Then for $n\ge 0$,
\[(X_n^R,Y_n^R)=([0/1;\alpha_{n+1}^R/\beta_{n+1}^R,\alpha_{n+2}^R/\beta_{n+2}^R,\dots],[0/1;1/\beta_n^R,\alpha_n^R/\beta_{n-1}^R,\dots,\alpha_1^R/\beta_0^R,\alpha_0^R/(1/y-1)]).\]
\end{prop}
\begin{proof}
For each $n\ge 0$, set
\[(T_n,V_n):=([0/1;\alpha_{n+1}^R/\beta_{n+1}^R,\alpha_{n+2}^R/\beta_{n+2}^R,\dots],[0/1;1/\beta_n^R,\alpha_n^R/\beta_{n-1}^R,\dots,\alpha_1^R/\beta_0^R,\alpha_0^R/(1/y-1)]).\]
Using \eqref{x=B*tail} and \eqref{B^T*z_first}, one finds that for each $n\ge 0$,
\begin{equation}\label{TV}
(T_{n+1},V_{n+1})=\left(\begin{pmatrix}0 & \alpha_{n+1}^R\\ 1 & \beta_{n+1}^R\end{pmatrix}^{-1}\cdot T_n,\begin{pmatrix}0 & 1\\ \alpha_{n+1}^R & \beta_{n+1}^R\end{pmatrix}\cdot V_n\right)=\left(\frac{\alpha_{n+1}^R}{T_n}-\beta_{n+1}^R,\frac1{\beta_{n+1}^R+\alpha_{n+1}^RV_n}\right).
\end{equation}
We will show by induction that $(X_n^R,Y_n^R)=(T_n,V_n)$ for all $n\ge 0$.  By (\ref{varphi_R}), Proposition \ref{cfe_alt} and the fact that $s_R(z)=1$ for all $z$, 
\[X_0^R=\begin{pmatrix}1 & -u_R(z) \\ 0 & 1\end{pmatrix}\cdot x=\begin{pmatrix}\alpha_0^R & -\beta_0^R \\ 0 & 1\end{pmatrix}\cdot x=\left(\begin{pmatrix}0 & 1 \\ 1 & 0\end{pmatrix}\begin{pmatrix}0 & \alpha_0^R \\ 1 & \beta_0^R\end{pmatrix}\right)^{-1}\cdot x.\]
Setting $n=0$ in (\ref{x=B*tail}) and multiplying both sides by $B_{[-1,0]}^{-1}$ reveals that $X_0^R=T_0$.  Similarly, 
\begin{align*}
Y_0^R=&\begin{pmatrix}-1 & 1 \\ 1-u_R(z) & u_R(z)\end{pmatrix}\cdot y=\begin{pmatrix}-1 & 1 \\ \alpha_0^R-\beta_0^R & \beta_0^R\end{pmatrix}\cdot y\\
=&\begin{pmatrix}0 & 1 \\ \alpha_0^R & \beta_0^R\end{pmatrix}\begin{pmatrix}0 & 1 \\ 1 & 0\end{pmatrix}\begin{pmatrix}-1 & 1 \\ 1 & 0\end{pmatrix}\cdot y=\left(\begin{pmatrix}0 & 1 \\ 1 & 0\end{pmatrix}\begin{pmatrix}0 & \alpha_0^R\\ 1 & \beta_0^R\end{pmatrix}\right)^T\begin{pmatrix}-1 & 1 \\ 1 & 0\end{pmatrix}\cdot y.\\
\end{align*}
Since $\left(\begin{smallmatrix}-1 & 1 \\ 1 & 0\end{smallmatrix}\right)\cdot y=1/y-1$, Equation (\ref{B^T*z}) gives $Y_0^R=V_0$.  Now suppose that $(X_n^R,Y_n^R)=(T_n,V_n)$ for some $n\ge 0$.  By Theorem \ref{two_shift_map_and_measure}, Proposition \ref{cfe_alt}, our inductive hypothesis and \eqref{TV},
\begin{align*}
(X_{n+1}^R,Y_{n+1}^R)=&\left(\frac{\alpha_R(z_n^R)}{X_n^R}-\beta_R(z_n^R),\frac1{\beta_R(z_n^R)+\alpha_R(z_n^R)Y_n^R}\right)\\
=&\left(\frac{\alpha_{n+1}^R}{T_n}-\beta_{n+1}^R,\frac1{\beta_{n+1}^R+\alpha_{n+1}^RV_n}\right)\\
=&(T_{n+1},V_{n+1}).
\end{align*}
\end{proof}

For the remainder of this subsection, we restrict our attention to the full-measure subset of points $z\in R$ for which $z_n^R=(x_n^R,y_n^R):=\F_R^n(z)\in R$ is defined and $x_n^R\neq 0$ for all $n\in\mathbb{Z}$ (in particular, $x\notin\mathbb{Q}$).  We remark that as $\F_R$ is totally invariant on this subset, the induced system $(R,\mathcal{B},\bar\mu_R,\F_R)$ and its restriction to this full-measure subset are isomorphic.  The same is true of the system $(\Omega_R,\mathcal{B},\bar\nu_R,\tau_R)$ and its restriction to the image under $\varphi_R$ of our full-measure, totally $\F_R$-invariant subset of $R$.  Abusing notation, we denote these restricted, isomorphic systems again by $(R,\mathcal{B},\bar\mu_R,\F_R)$ and $(\Omega_R,\mathcal{B},\bar\nu_R,\tau_R)$.  

Now, let 
\[\Delta(0/1;\alpha_1/\beta_1,\alpha_2/\beta_2,\dots,\alpha_n/\beta_n)\times\Delta(0/1;1/\beta_0,\alpha_0/\beta_{-1},\alpha_{-1}/\beta_{-2},\dots,\alpha_{-(m-1)}/\beta_{-m})\subset\Omega_R\]
be the (possibly empty) set of points $(X(z),Y(z))\in\Omega_R$ satisfying
\[\alpha_R(z_j^R)=\alpha_{j+1}\quad \text{and}\quad \beta_R(z_k^R)=\beta_{k+1}\]
for all $-m\le j\le n-1$ and $-m-1\le k\le n-1$.  The following result is needed in \S\ref{Nakada's alpha-continued fractions, revisited} below when we realise the natural extensions of the $\alpha$-{\sc cf}s as induced systems $(R,\mathcal{B},\bar\mu_R,\F_R)$.

\begin{prop}\label{gen_set_of_B}
The Borel $\sigma$-algebra $\mathcal{B}$ restricted to $\Omega_R$ is generated by the sets
\[\Delta(0/1;\alpha_1/\beta_1,\alpha_2/\beta_2,\dots,\alpha_n/\beta_n)\times\Delta(0/1;1/\beta_0,\alpha_0/\beta_{-1},\alpha_{-1}/\beta_{-2},\dots,\alpha_{-(m-1)}/\beta_{-m}).\]
\end{prop}
\begin{proof}
We first remark that each of these sets belongs to $\mathcal{B}$ since each may be written as an intersection of preimages of integers under compositions of the measurable functions $\alpha_R$, $\beta_R$, $\F_R^{\pm 1}$ and $\varphi_R^{-1}$.  Next, notice that there are only countably many such sets, so it suffices to show that any open set $U\in \Omega_R$ can be written as \emph{some} union of these.  It thus suffices to show that for any $(X,Y)=(X(z),Y(z))\in U$, there exists some set 
\begin{equation}\label{DxD}
D_n=\Delta(0/1;\alpha_1/\beta_1,\alpha_2/\beta_2,\dots,\alpha_n/\beta_n)\times\Delta(0/1;1/\beta_0,\alpha_0/\beta_{-1},\alpha_{-1}/\beta_{-2},\dots,\alpha_{-(n-2)}/\beta_{-(n-1)})
\end{equation}
such that $(X,Y)\in D_n\subset U$.  By definition, $(X,Y)$ belongs to each $D_n,\ n\ge 1$, for which 
\begin{equation}\label{XYinDxD}
\alpha_{j+1}=\alpha_R(z_j^R)\quad \text{and}\quad \beta_{k+1}=\beta_R(z_k^R)
\end{equation}
for all $-(n-1)\le j\le n-1$ and $-n\le k\le n-1$.  Thus, to prove that there is some $n$ for which $D_n\subset U$, it suffices to show that the Euclidean diameters of the sets $D_n$ tend to $0$ uniformly in $n$.  For this, it suffices to show that 
\[|X-c_n|\to 0 \qquad \text{and}\qquad |Y-d_n|\to 0\]
uniformly in $n$, where---recycling notation---$(X,Y)$ is an arbitrary point in $D_n$ and 
\[(c_n,d_n):=([0/1;\alpha_1/\beta_1,\alpha_2/\beta_2,\dots,\alpha_n/\beta_n],[0/1;1/\beta_0,\alpha_0/\beta_{-1},\alpha_{-1}/\beta_{-2},\dots,\alpha_{-(n-2)}/\beta_{-(n-1)}]).\]

Fix $D_n$ as in (\ref{DxD}) and assume $(X,Y)=(X(z),Y(z))\in D_n$ so that (\ref{XYinDxD}) holds.  Proposition \ref{cfe_alt} and the fact that $s_R(z)=1$ (and hence $d_R(z)=1$) imply that the digits of the contracted Farey expansion of $x$ with respect to $R$ and $z=(x,y)$ are given by 
\[
\begin{pmatrix}
\alpha_0^R\\
\beta_0^R
\end{pmatrix}
=
\begin{pmatrix}
1\\
u_R(z)
\end{pmatrix}
\qquad
\text{and}
\qquad
\begin{pmatrix}
\alpha_{j+1}^R\\
\beta_{j+1}^R
\end{pmatrix}
=
\begin{pmatrix}
\alpha_R(z_{j}^R)\\
\beta_R(z_{j}^R)
\end{pmatrix},
\quad j\ge 0.
\]
The previous line and \eqref{varphi_R} give $X=x-\beta_0^R$.  Letting $P_n^R,\ Q_n^R$ be as in (\ref{P_n^R&Q_n^R}), the previous line and \eqref{XYinDxD} give
\[\frac{P_n^R}{Q_n^R}=[\beta_0^R/\alpha_0^R;\alpha_1^R/\beta_1^R,\dots,\alpha_n^R/\beta_n^R]=[\beta_0^R/1;\alpha_1/\beta_1,\dots,\alpha_n/\beta_n],\]
so also $c_n=\frac{P_n^R}{Q_n^R}-\beta_0^R$.  By Corollary \ref{PQ_in_us_cor} and \eqref{A_n^R_entries}, 
\[\frac{P_n^R}{Q_n^R}=\frac{u_{n+1}^R}{s_{n+1}^R}=\frac{\lambda_Np_{j_N}+p_{j_N-1}}{\lambda_Nq_{j_N}+q_{j_N-1}},\]
where $N=N_{n+1}^R(z)$.  By \eqref{meds-mon-odd} and \eqref{meds-mon-even},
\[|X-c_n|=\left|x-\frac{P_n^R}{Q_n^R}\right|=\left|x-\frac{\lambda_Np_{j_N}+p_{j_N-1}}{\lambda_Nq_{j_N}+q_{j_N-1}}\right|\le\left|x-\frac{p_{j_N-1}}{q_{j_N-1}}\right|\le \frac{1}{q_{j_N-1}^2},\]
where the final inequality follows classical arguments in the theory of {\sc rcf}s.  Since $R$ is bounded away from the origin, there is some integer $M>0$ such that for any integer $a\ge 1$, the number of rectangles $V_{a-\lambda}\cap H_{\lambda+1}$, $0\le \lambda<a$, intersecting $R$ is no greater than $M$.  By \eqref{j_n&lambda_n_alt} and the fact that $z_{n+1}^R\in V_{a_{j_N+1}-\lambda_N}\cap H_{\lambda_N+1}$, this implies that $j_N=j_{N_{n+1}^R(z)}$ grows uniformly in $n$.  Since the denominators $q_j$ of {\sc rcf}-convergents are strictly increasing, we have that $|X-c_n|\to 0$ uniformly in $n$.  

It remains to show that $|Y-d_n|\to 0$ uniformly in $n$.  Let $n\ge 1$, and consider $z_{-n}^R=(x_{-n}^R,y_{-n}^R)=\F_R^{-n}(z)$.  By Proposition \ref{cfe_alt}, the digits of the contracted Farey expansion of $x_{-n}^R$ with respect to $R$ and $z_{-n}^R$ are given by 
\[
\begin{pmatrix}
\alpha_0^R(z_{-n}^R)\\
\beta_0^R(z_{-n}^R)
\end{pmatrix}
=
\begin{pmatrix}
1\\
u_R(z_{-n}^R)
\end{pmatrix}
\qquad
\text{and}
\qquad
\begin{pmatrix}
\alpha_{k+1}^R(z_{-n}^R)\\
\beta_{k+1}^R(z_{-n}^R)
\end{pmatrix}
=
\begin{pmatrix}
\alpha_R(z_{-n+k}^R)\\
\beta_R(z_{-n+k}^R)
\end{pmatrix}, 
\quad k\ge 0.
\]

Since $(X_n^R(z_{-n}^R),Y_n^R(z_{-n}^R))=(X,Y)\in D_n$ (see (\ref{XnRYnRunRsnR})), we have by (\ref{XYinDxD})
\[\alpha_{j+1}=\alpha_R(z_{-n+(n+j)}^R)=\alpha_{n+j+1}^R(z_{-n}^R)\quad \text{and}\quad \beta_{k+1}=\beta_R(z_{-n+(n+k)}^R)=\beta_{n+k+1}^R(z_{-n}^R)\]
for all $-(n-1)\le j\le n-1$ and $-n\le k\le n-1$.  In particular, applying Proposition \ref{two_shift_prop} to $z_{-n}^R$, we find
\[Y=Y_n^R(z_{-n}^R)=[0/1;1/\beta_n^R(z_{-n}^R),\alpha_n^R(z_{-n}^R)/\beta_{n-1}^R(z_{-n}^R),\dots,\alpha_1^R(z_{-n}^R)/\beta_0^R(z_{-n}^R),\alpha_0^R(z_{-n}^R)/(1/y_{-n}^R-1)]),\]
while
\begin{align*}
d_n=&[0/1;1/\beta_0,\alpha_0/\beta_{-1},\alpha_{-1}/\beta_{-2},\dots,\alpha_{-(n-2)}/\beta_{-(n-1)}]\\
=&[0/1;1/\beta_n^R(z_{-n}^R),\alpha_n^R(z_{-n}^R)/\beta_{n-1}^R(z_{-n}^R),\dots,\alpha_2^R(z_{-n}^R)/\beta_1^R(z_{-n}^R)])].
\end{align*}
Set
\begin{align*}
B_{[-1,n]}^R(z_{-n}^R):=&B_{[-1,n]}([\beta_0^R(z_{-n}^R)/\alpha_0^R(z_{-n}^R);\alpha_1^R(z_{-n}^R)/\beta_1^R(z_{-n}^R),\dots,\alpha_n^R(z_{-n}^R)/\beta_n^R(z_{-n}^R)])\\
=&\begin{pmatrix}0 & 1 \\ 1 & 0\end{pmatrix}\begin{pmatrix}0 & \alpha_0^R(z_{-n}^R) \\ 1 & \beta_0^R(z_{-n}^R)\end{pmatrix}\begin{pmatrix}0 & \alpha_1^R(z_{-n}^R) \\ 1 & \beta_1^R(z_{-n}^R)\end{pmatrix}\cdots \begin{pmatrix}0 & \alpha_n^R(z_{-n}^R) \\ 1 & \beta_n^R(z_{-n}^R)\end{pmatrix},
\end{align*}
and denote the entries by
\[B_{[-1,n]}^R(z_{-n}^R)=\begin{pmatrix}
P_{n-1}^R(z_{-n}^R) & P_n^R(z_{-n}^R)\\
Q_{n-1}^R(z_{-n}^R) & Q_n^R(z_{-n}^R)
\end{pmatrix}.\]
Then Equation (\ref{B^T*z}) gives
\[Y=(B_{[-1,n]}^R(z_{-n}^R))^T\cdot \left(\frac1{y_{-n}^R}-1\right)=\frac{P_{n-1}^R(z_{-n}^R)(1-y_{-n}^R)+Q_{n-1}^R(z_{-n}^R)y_{-n}^R}{P_n^R(z_{-n}^R)(1-y_{-n}^R)+Q_n^R(z_{-n}^R)y_{-n}^R},\]
while 
\[d_n=(B_{[-1,n]}^R(z_{-n}^R))^T\cdot 0=\frac{Q_{n-1}^R(z_{-n}^R)}{Q_n^R(z_{-n}^R)}.\]
Notice by Proposition \ref{cfe_alt}, Equation (\ref{alpha_R}) and the fact that $s_R(z)=1$ for all $z$, that
\[|\det(B_{[-1,n]}^R(z_{-n}^R)|=|\alpha_0^R(z_{-n}^R)\alpha_1^R(z_{-n}^R)\cdots \alpha_n^R(z_{-n}^R)|=1.\]
Moreover, recall that $y_{-n}^R\in[0,1]$.  We thus compute 
\begin{align*}
|Y-d_n|=&\left|\frac{P_{n-1}^R(z_{-n}^R)(1-y_{-n}^R)+Q_{n-1}^R(z_{-n}^R)y_{-n}^R}{P_n^R(z_{-n}^R)(1-y_{-n}^R)+Q_n^R(z_{-n}^R)y_{-n}^R}-\frac{Q_{n-1}^R(z_{-n}^R)}{Q_n^R(z_{-n}^R)}\right|\\
=&\frac{|P_{n-1}^R(z_{-n}^R)Q_n^R(z_{-n}^R)-P_n^R(z_{-n}^R)Q_{n-1}^R(z_{-n}^R)||1-y_{-n}^R|}{|P_n^R(z_{-n}^R)(1-y_{-n}^R)+Q_n^R(z_{-n}^R)y_{-n}^R||Q_n^R(z_{-n}^R)|}\\
\le&\frac{1}{|P_n^R(z_{-n}^R)+(Q_n^R(z_{-n}^R)-P_n^R(z_{-n}^R))y_{-n}^R|Q_n^R(z_{-n}^R)}\\
\le& \frac1{\min\left\{P_n^R(z_{-n}^R),Q_n^R(z_{-n}^R)\right\} Q_n^R(z_{-n}^R)}.
\end{align*}
By Corollary \ref{PQ_in_us_cor},
\[\min\left\{P_n^R(z_{-n}^R),Q_n^R(z_{-n}^R)\right\} Q_n^R(z_{-n}^R)=\min\left\{u_{n+1}^R(z_{-n}^R),s_{n+1}^R(z_{-n}^R)\right\} s_{n+1}^R(z_{-n}^R),\]
so it suffices to show that 
\[\min\left\{u_n^R(z),s_n^R(z)\right\} s_n^R(z)\to\infty\]
uniformly in $n$.  Write $u_n^R(z)=\lambda_Np_{j_N}+p_{j_N-1}\ge p_{j_N-1}$ and $s_n^R(z)=\lambda_Nq_{j_N}+q_{j_N-1}\ge q_{j_N-1}$ where $N=N_n^R(z)$.  As before, since $R$ is bounded away from the origin, $j_N=j_{N_n^R(z)}$ grows uniformly in $n$.  Since $x\notin\mathbb{Q}$, there is some $n$ large enough (independent of $z$) for which $u_n^R(z)\ge p_{j_N-1}\ge 1$.  Since the {\sc rcf}-convergent denominators $q_j$ are strictly increasing for $j>0$, 
\[\min\left\{u_n^R(z),s_n^R(z)\right\} s_n^R(z)\ge\min\left\{p_{j_N-1},q_{j_N-1}\right\} q_{j_N-1} \to\infty\]
uniformly in $n$.  
\end{proof}

Notice from Propositions \ref{cfe_alt} and \ref{two_shift_prop} that 
\[X(z)=[0/1;\alpha_R(z_0^R)/\beta_R(z_0^R),\alpha_R(z_1^R)/\beta_R(z_1^R),\dots].\]
From the proof of Proposition \ref{gen_set_of_B}, it is evident that the convergents 
\[d_n=[0/1;1/\beta_0,\alpha_0/\beta_{-1},\alpha_{-1}/\beta_{-2},\dots,\alpha_{-(n-2)}/\beta_{-(n-1)}]\]
of the {\sc gcf}
\[[0/1;1/\beta_0,\alpha_0/\beta_{-1},\alpha_{-1}/\beta_{-2},\dots]\]
with 
\[\alpha_{j+1}=\alpha_R(z_j^R) \quad \text{and} \quad \beta_{j+1}=\beta_R(z_j^R), \quad j<0,\]
also converge to $Y(z)$.  We thus obtain {\sc gcf}-expansions of both $X(z)$ and $Y(z)$ on which $\tau_R$ acts as a two-sided shift:

\begin{cor}\label{two_shift_cor}
For $z\in R$ for which $z_n^R$ is defined for all $n\in\mathbb{Z}$, 
\[\left(X(z),Y(z)\right)=\left([0/1;\alpha_R(z_0^R)/\beta_R(z_0^R),\alpha_R(z_1^R)/\beta_R(z_1^R),\dots],[0/1;1/\beta_R(z_{-1}^R),\alpha_R(z_{-1}^R)/\beta_R(z_{-2}^R),\dots]\right),\]
and for any $n\in\mathbb{Z}$, $\tau_R^n\left(X(z),Y(z)\right)$ equals
\[\left([0/1;\alpha_R(z_n^R)/\beta_R(z_n^R),\alpha_R(z_{n+1}^R)/\beta_R(z_{n+1}^R),\dots],[0/1;1/\beta_R(z_{n-1}^R),\alpha_R(z_{n-1}^R)/\beta_R(z_{n-2}^R),\dots]\right).\]
\end{cor}

\section{Examples of contracted Farey expansions}\label{Examples of contracted Farey expansions}

In this section we consider several examples of explicit, inducible regions $R$ and the contracted Farey expansions they produce.  We shall find in \S\ref{Regular continued fractions, revisited} {\sc rcf}s, in \S\ref{S-expansions, revisited} the second-named author's $S$-expansions, and in \S\ref{Nakada's alpha-continued fractions, revisited} Nakada's $\alpha$-{\sc cf}s for $\alpha\in (0,1]$.  Throughout this section, any reference to the induced system $(H_1,\mathcal{B},\bar\mu_{H_1},\F_{H_1})$ is to the `altered' system from Remark \ref{altered_maps}.

\subsection{Regular continued fractions, revisited}\label{Regular continued fractions, revisited}

Set $R=H_1$, and recall from Theorem \ref{H_1&Gauss} above that the induced system $(R,\mathcal{B},\bar\mu_R,\F_R)$ is isomorphic to the Gauss natural extension $(\Omega,\mathcal{B},\bar{\nu}_G,\G)$.  We re-obtain this fact here through the use of contracted Farey expansions and the two-sided shift of \S\ref{A two-sided shift for contracted Farey expansions}.  

\begin{proof}[Proof of Theorem \ref{H_1&Gauss}]
Let $z=(x,y)\in R$ with $x\neq 0$ be as in (\ref{(x,y)-rcfs}).  Using \eqref{F_NE_symbolic}, we find that $N_R(z)=a_1=a(x)$, and by (\ref{A_R(z)}),
\begin{equation}\label{A_{H_1}(z)}
\begin{pmatrix}
u_R(z) & t_R(z)\\
s_R(z) & r_R(z)
\end{pmatrix}
=A_R(z)=A_{\varepsilon_1}\dots A_{\varepsilon_{a_1}}=A_0^{a_1-1}A_1=\begin{pmatrix}0 & 1\\ 1 & a_1\end{pmatrix}=\begin{pmatrix}0 & 1\\ 1 & a(x)\end{pmatrix}.
\end{equation}
In particular, $s_R(z)=1$ for all $z$, so we are in the setting of \S\ref{A two-sided shift for contracted Farey expansions}.  We know that $(R,\mathcal{B},\bar\mu_R,\F_R)$ is isomorphic to $(\Omega_R,\mathcal{B},\bar\nu_R,\tau_R)$; we shall show that this latter system is precisely $(\Omega,\mathcal{B},\bar\nu_G,\G)$.  Since $u_R(z)=0$ for all $z$, the map $\varphi_R:R\to\mathbb{R}^2$ from (\ref{varphi_R}) is
\begin{equation}\label{H_1-varphi_R}
\varphi_R(z)=\left(x,\frac{1-y}{y}\right) \qquad \text{for all $z=(x,y)\in R$,\ $x\neq 0$},
\end{equation}
and thus $\Omega_R=\varphi_R(R)=\Omega$, up to a null set.  Since $\bar\mu(R)=\log 2$, Theorem \ref{two_shift_map_and_measure} gives that $\bar\nu_R=\bar\nu_G$.  Moreover, from Equations (\ref{alpha_R}), (\ref{beta_R}), and (\ref{A_{H_1}(z)}) we find
\begin{equation}\label{alpha_{H_1},beta_{H_1}}
\begin{pmatrix}
\alpha_R(z)\\
\beta_R(z)
\end{pmatrix}
=\begin{pmatrix}
1\\
a(x)
\end{pmatrix}.
\end{equation}
But if $(X,Y)=\varphi_R(z)$, Equation (\ref{H_1-varphi_R}) gives $X=x$, so by Theorem \ref{two_shift_map_and_measure} and Equation (\ref{GNE}),
\[\tau_R(X,Y)=\left(\frac1X-a(X),\frac1{a(X)+Y}\right)=\G(X,Y).\]
Thus $(\Omega_R,\mathcal{B},\bar\nu_R,\tau_R)=(\Omega,\mathcal{B},\bar\nu_G,\G)$.
\end{proof}

Let $z=(x,y)\in R$ as in \eqref{(x,y)-rcfs} (so $b_1=1$) with $x\notin\mathbb{Q}$, and notice that repeated use of (\ref{F_{H_1}}) gives
\[z_k^R=(x_k^R,y_k^R)=\F_R^k(x,y)=([0;a_{k+1},a_{k+2},\dots],[0;1,a_k,\dots,a_1,b_2,b_3,\dots]).\]
Thus, by Proposition \ref{cfe_alt} and Equation \eqref{alpha_{H_1},beta_{H_1}}, the digits of the contracted Farey expansion of $x$ with respect to $R=H_1$ and $z=(x,y)\in R$ are
\[\begin{pmatrix}
\alpha_0^R\\
\beta_0^R
\end{pmatrix}=
\begin{pmatrix}
s_R(z)\\
u_R(z)
\end{pmatrix}=
\begin{pmatrix}
1\\
0
\end{pmatrix}
\qquad
\text{and}
\qquad
\begin{pmatrix}
\alpha_{k+1}^R\\
\beta_{k+1}^R
\end{pmatrix}=
\begin{pmatrix}
1\\
a(x_k^R)
\end{pmatrix}=
\begin{pmatrix}
1\\
a_{k+1}
\end{pmatrix}.
\]
That is, the contracted Farey expansion of $x$ with respect to $R=H_1$ and $z=(x,y)$ recovers the {\sc rcf}-expansion $[0/1;1/a_1,1/a_2,\dots]=[0;a_1,a_2,\dots]$ of $x$.

\subsection{$S$-expansions, revisited}\label{S-expansions, revisited}

We also find $S$-expansions (and thus Minkowski's diagonal {\sc cf}s, Bosma's optimal {\sc cf}s, and Nakada's $\alpha$-{\sc cf}s for $\alpha\ge 1/2$; see \cite{K1991} and \S\ref{S-expansions} above) as special instances of contracted Farey expansions.  Indeed, let $S\subset\Omega$ be a singularisation area, i.e., $S$ is $\bar\nu_G$-measurable set with $\bar\nu_G(\partial S)=0$ satisfying both 
\begin{itemize}
\item[(a)] $S\subset V_1$ and
\item[(b)] $S\cap\G(S)=\varnothing$,
\end{itemize}
and let $[\beta_0^S/\alpha_0^S;\alpha_1^S/\beta_1^S,\alpha_2^S/\beta_2^S,\dots]$ be the $S$-expansion of $x=[0;a_1,a_2,\dots]\in (0,1)\setminus\mathbb{Q}$ obtained by simultaneously singularising at all positions $n$ for which $\G^n(x,0)\in S$ (see Definitions 4.4, 4.5 of \cite{K1991} and \S\ref{S-expansions} above).  For $n\ge -1$ let
\[B_{[-1,n]}^S=\begin{pmatrix}
P_{n-1}^S & P_n^S\\
Q_{n-1}^S & Q_n^S
\end{pmatrix}:=B_{[-1,n]}([\beta_0^S/\alpha_0^S;\alpha_1^S/\beta_1^S,\alpha_2^S/\beta_2^S,\dots]).\]
From remarks preceding Theorem 4.13 and Theorem 5.3.i of \cite{K1991}, it follows that $P_{-2}^S=Q_{-1}^S=0$, $P_{-1}^S=Q_{-2}^S=1$, and for $k\ge 0$,
\[\begin{pmatrix}
P_k^S\\
Q_k^S
\end{pmatrix}=
\begin{pmatrix}
p_{j_k^S}\\
q_{j_k^S}
\end{pmatrix},\]
where $p_j/q_j$ is the $j^\text{th}$ {\sc rcf}-convergent of $x$ and $(j_k^S)_{k\ge 0}$ is the subsequence of powers $j\ge 0$ for which $\G^j(x,0)\in \Delta:=\Omega\backslash S$.  

We wish to determine a proper, inducible subregion $R\subset\Omega$ for which the contracted Farey expansion $[\beta_0^R/\alpha_0^R;\alpha_1^R/\beta_1^R,\alpha_2^R/\beta_2^R,\dots]$ of $x$ with respect to $R$ coincides with the $S$-expansion of $x$.  By Remark \ref{convs_determin_digits}, it suffices to find $R$ such that $P_k^R=P_k^S$ and $Q_k^R=Q_k^S$ for all $k\ge 0$, with
\[B_{[-1,k]}^R=\begin{pmatrix}
P_{k-1}^R & P_k^R\\
Q_{k-1}^R & Q_k^R
\end{pmatrix}
\]
as in (\ref{B_{[m,n]}^R}) and (\ref{P_n^R&Q_n^R}).

It seems natural to set $R:=\varphi_{H_1}^{-1}(\Delta)\subset  H_1$, where $\varphi_{H_1}:H_1\to \Omega$ is the isomorphism map between $(H_1,\mathcal{B},\bar\mu_{H_1},\F_{H_1})$ and $(\Omega,\mathcal{B},\bar\nu_G,\G)$ from (\ref{H_1-varphi_R}) above satisfying $\varphi_{H_1}\circ\F_{H_1}(z)=\G\circ \varphi_{H_1}(z)$ for all $z\in H_1$.  However, in the classical setting of {\sc rcf}s and, in particular, $S$-expansions, one uses the one-to-one correspondence between points in the $\G$-orbit of $(x,0)$ and {\sc rcf}-convergents $p_n/q_n$, which come from the \emph{right-hand column} of the matrix $\left(\begin{smallmatrix}p_{n-1} & p_n\\ q_{n-1} & q_n \end{smallmatrix}\right)$.  On the other hand, for contracted Farey expansions we use the one-to-one correspondence between points in the $\F_R$-orbit of $(x,1)$ and contracted Farey convergents $u_n^R/s_n^R$ coming from the \emph{left-hand column} of the matrix $A_{[0,n]}^R=A_{[0,N_n^R]}$ from (\ref{supr_not}).  When $R=\varphi_{H_1}^{-1}(\Omega)=H_1$, the matrix $A_{[0,n]}^{H_1}$ is of the form $\left(\begin{smallmatrix}p_{n-1} & p_{n}\\ q_{n-1} & q_{n} \end{smallmatrix}\right)$, so the one-to-one correspondence in this setting is between $\F_{H_1}^{n}(x,1)=\varphi_{H_1}^{-1}\circ\G^{n}(x,0)$ and $p_{n-1}/q_{n-1}$.  This indexing discrepancy is fixed by instead considering the isomorphism map $\psi:=\G^{-1}\circ \varphi_{H_1}$ between $(H_1,\mathcal{B},\bar\mu_{H_1},\F_{H_1})$ and $(\Omega,\mathcal{B},\bar\nu_G,\G)$:

\[
\begin{tikzcd}
H_1 \arrow{r}{\F_{H_1}} \arrow[swap]{d}{\varphi_{H_1}}  & H_1 \arrow{d}{\varphi_{H_1}} \arrow[swap]{dl}{\psi} \\
\Omega \arrow{r}{\G} & \Omega%
\end{tikzcd}
\]

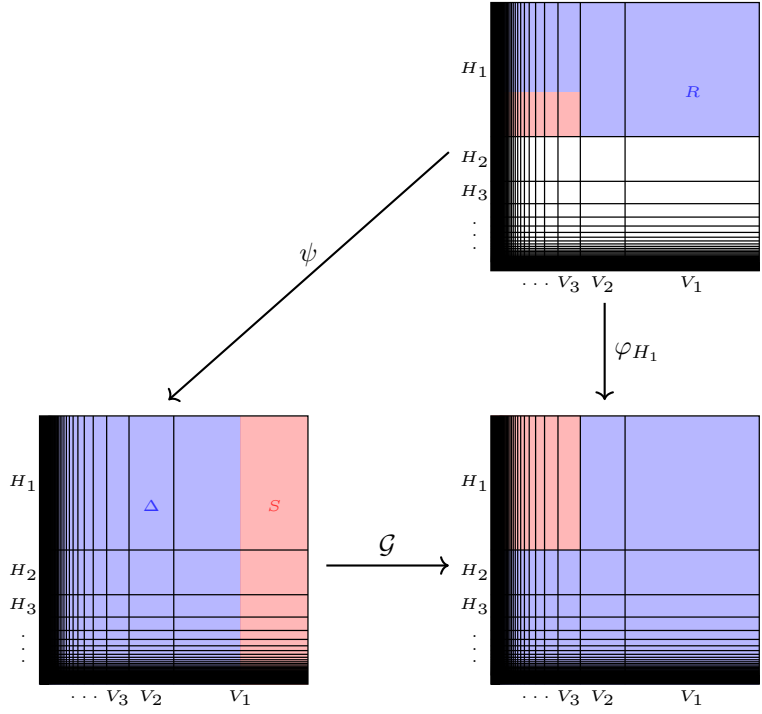
\begin{figure}[t]
\centering
\begin{tikzpicture}
\node[inner sep=0pt] (SE1) at (0,0)
	{\includegraphics[width=.25\textwidth]{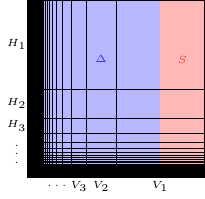}};
\node[inner sep=0pt] (SE2) at (6,0)
	{\includegraphics[width=.25\textwidth]{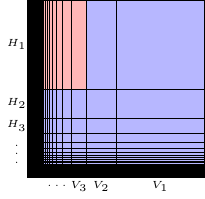}};
\node[inner sep=0pt] (SE3) at (6,5.5)
	{\includegraphics[width=.25\textwidth]{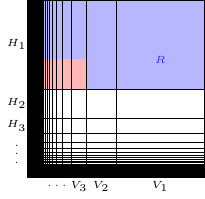}};
\draw[->,thick] (2.3,0) -- (SE2.west) node[midway, above] {$\G$};
\draw[<-,thick] (6,2.2) -- (SE3.south) node[midway, right] {$\varphi_{H_1}$};
\draw[->,thick] (SE3.west) -- (0.2,2.2) node[midway, above] {$\psi$};
\end{tikzpicture}
\caption{Bottom-left: A singularisation area $S$ and its complement $\Delta$ in $\Omega$.  Bottom-right: The images of $S$ and $\Delta$ under $\G$.  Top-right: The region $R=\psi^{-1}(\Delta)$ and its complement in $H_1$.}
\label{S-expns_fig}
\end{figure}

Set
\[R:=\psi^{-1}(\Delta)= H_1\backslash \psi^{-1}(S);\]
see Figure \ref{S-expns_fig}.  Notice that for any $z\in H_1$, either $\varphi_{H_1}(z)\in \Delta$ or $\G\circ\varphi_{H_1}(z)\in \Delta$; otherwise, both $\varphi_{H_1}(z)$ and $\G\circ \varphi_{H_1}(z)$ belong to $S$, contrary to condition (ii) of a singularisation area.  Thus, either $\psi^{-1}\circ \varphi_{H_1}(z)\in R$ or $\psi^{-1}\circ \G\circ\varphi_{H_1}(z)\in R$.  But $\psi^{-1}\circ \varphi_{H_1}=\F_{H_1}$ and $\psi^{-1}\circ \G\circ \varphi_{H_1}=\F_{H_1}^2$, so for any $z\in H_1$, either $\F_{H_1}(z)\in R$ or $\F_{H_1}^2(z)\in R$.  The entries of the matrices $A_R(z)$ depend on whether $\F_{H_1}(z)\in R$:

\begin{lem}\label{A_R-S-expns}
For any $z=(x,y)\in H_1$ with $x=[0;a_1,a_2,\dots]$,
\[\begin{pmatrix}
u_R(z) & t_R(z)\\
s_R(z) & r_R(z)
\end{pmatrix}=A_R(z)=\begin{cases}
\begin{pmatrix}
0 & 1\\
1 & a_1
\end{pmatrix} & \text{if $\F_{H_1}(z)\in R$},\\
\begin{pmatrix}
1 & a_2\\
1 & a_2+1
\end{pmatrix} & \text{if $\F_{H_1}(z)\notin R$}.
\end{cases}\]
\end{lem}
\begin{proof}
First suppose that $\F_{H_1}(z)\in R$.  Now $\F_{H_1}(z)=\F^{a_1}(z)$, and for all $1\le j<a_1$, $\F^j(z)\notin H_1$ implies $\F^j(z)\notin R\subset H_1$.  Thus $N_R(z)=a_1$, and by (\ref{A_R(z)}) we have
\[A_R(z)=A_0^{a_1-1}A_1=\begin{pmatrix}
0 & 1\\
1 & a_1
\end{pmatrix}.\]
If $\F_{H_1}(z)=\F^{a_1}(z)\notin R$, then $\F_{H_1}^2(z)=\F^{a_1+a_2}(z)\in R$.  Since $\F^j(z)\notin H_1$ for $1\le j<a_1+a_2$ with $j\neq a_1$, we have $N_R(z)=a_1+a_2$ and---by \eqref{A_R(z)}---
\[A_R(z)=A_0^{a_1-1}A_1A_0^{a_2-1}A_1=\begin{pmatrix}
0 & 1\\
1 & a_1
\end{pmatrix}\begin{pmatrix}
0 & 1\\
1 & a_2
\end{pmatrix}=\begin{pmatrix}
1 & a_2\\
a_1 & a_2a_1+1
\end{pmatrix}
.\]
But $\F_{H_1}(z)\notin R$ is equivalent to $\varphi_{H_1}(z)\notin \Delta$, or $\varphi_{H_1}(z)\in S$.  Since $\varphi_{H_1}$ acts as the identity on the first coordinate and $S\subset V_1$ by condition (a) of a singularisation area, this implies $a_1=1$.
\end{proof}

By Lemma \ref{A_R-S-expns}, $s_R(z)=1$ for all $z\in H_1$ and, in particular, for all $z=(x,1)$.  By Corollary \ref{PQ_in_us_cor}, Equation \eqref{A_n^R_entries}, and the fact that $R\subset H_1$ (so $\lambda_{N_{k+1}^R}=0$),
\[\begin{pmatrix}
P_k^R\\
Q_k^R
\end{pmatrix}=\begin{pmatrix}
u_{k+1}^R\\
s_{k+1}^R
\end{pmatrix}=
\begin{pmatrix}
p_{j_{N_{k+1}^R}-1}\\
q_{j_{N_{k+1}^R}-1}
\end{pmatrix},\qquad k\ge0.\]
Writing $N_{k+1}^R=a_1+\cdots +a_{j_{N_{k+1}^R}}$ (see \eqref{j_n&lambda_n_alt}), we see that the indices $j_{N_{k+1}^R},\ k\ge 0,$ are precisely the powers $j>0$ for which 
\[\F^{a_1+\cdots+a_j}(x,1)=\F_{H_1}^j(x,1)\in R.\]
Equivalently, these are the powers $j>0$ for which
\[\varphi_{H_1}^{-1}\circ \G^j\circ \varphi_{H_1}(x,1)\in R=\psi^{-1}(\Delta)=\varphi_{H_1}^{-1}\circ \G(\Delta),\]
or $\G^{j-1}(x,0)\in \Delta$.  Thus $j_{N_{k+1}^R}-1=j_k^S$, and 
\[\begin{pmatrix}
P_k^R\\
Q_k^R
\end{pmatrix}=
\begin{pmatrix}
p_{j_{N_{k+1}^R}-1}\\
q_{j_{N_{k+1}^R}-1}
\end{pmatrix}=
\begin{pmatrix}
p_{j_k^S}\\
q_{j_k^S}
\end{pmatrix}=
\begin{pmatrix}
P_k^S\\
Q_k^S
\end{pmatrix},
\qquad k\ge0.\]

By Remark \ref{convs_determin_digits}, this proves:

\begin{prop}\label{cfe&s-expn_agree}
The contracted Farey expansion of $x$ with respect to $R=\psi^{-1}(\Delta)$ coincides with the $S$-expansion of $x$.
\end{prop}

In \S5 of \cite{K1991}, a two-dimensional ergodic system\footnote{We replace the original notation $\Omega_S$ from \cite{K1991} by $\Gamma_S$ to avoid confusion with $\Omega_R$ defined \S\ref{A two-sided shift for contracted Farey expansions}.  However, we shall see in Proposition \ref{S-expn_two_shift} below that, in fact, $\Gamma_S=\Omega_R$.} $(\Gamma_S,\mathcal{B},\rho,\tau)$ is constructed corresponding to the two-sided shift operator for $S$-expansions.  We briefly recall this system here and show that it coincides with $(\Omega_R,\mathcal{B},\bar\nu_R,\tau_R)$ as defined in \S\ref{A two-sided shift for contracted Farey expansions}.  (Note by Lemma \ref{A_R-S-expns} that $s_R(z)=1$ for all $z\in R$, so we are in the setting of \S\ref{A two-sided shift for contracted Farey expansions}.)  Set 
\[\Delta^-:=\G(S) \qquad \text{and} \qquad \Delta^+:=\Delta\backslash\Delta^-.\]
Define $M:\Delta\to\mathbb{R}^2$ for $z=(x,y)$ by
\[M(z):=\begin{cases}
(x,y), & z\in \Delta^+,\\
\left(\frac{-x}{1+x},1-y\right), & z\in \Delta^-,
\end{cases}\]
and let $\Gamma_S:=M(\Delta)$.  The map $\tau:\Gamma_S\to\Gamma_S$ is defined by $\tau:=M\circ\G_\Delta\circ M^{-1},$ where $\G_\Delta:\Delta\to\Delta$ is the map $\G$ induced on $\Delta$, i.e., $\G_\Delta(z)=\G(z)$ if $\G(z)\in \Delta$ and $\G_\Delta(z)=\G^2(z)$ otherwise.  The measure $\rho$ is the probability measure on $(\Gamma_S,\mathcal{B})$ with density $1/(\log2\bar\nu_G(\Delta)(1+XY)^2)$ (see Theorem 5.9 of \cite{K1991}).  Setting 
\[X_k^S:=[0/1;\alpha_{k+1}^S/\beta_{k+1}^S,\alpha_{k+2}^S/\beta_{k+2}^S,\dots], \qquad k\ge 0,\]
$Y_0^S:=0$ and 
\[Y_k^S:=[0/1;1/\beta_k^S,\alpha_k^S/\beta_{k-1}^S,\dots,\alpha_2^S/\beta_1^S], \qquad k\ge 1,\]
where $x=[\beta_0^S/\alpha_0^S;\alpha_1^S/\beta_1^S,\alpha_2^S/\beta_2^S,\dots]$ is the $S$-expansion of $x$, it is observed following Definition 5.8 of \cite{K1991} that 
\[(X_k^S,Y_k^S)=\tau^k(X_0^S,Y_0^S), \qquad k\ge 0.\]

Note that by Propositions \ref{two_shift_prop} and \ref{cfe&s-expn_agree}, $\tau^n$ and $\tau_R^n$ agree for all $n\ge 0$ when evaluated at $(X_0^S,Y_0^S)=(X_0^S,0)$.  We claim that in fact $(\Omega_R,\mathcal{B},\bar\nu_R,\tau_R)=(\Gamma_S,\mathcal{B},\rho,\tau)$.  By \eqref{varphi_R} and Lemma \ref{A_R-S-expns}, 
\begin{equation}\label{S-expn_phi_R}
\varphi_R(z)=\begin{cases}
\left(x,\frac{1-y}{y}\right), & \F_{H_1}(z)\in R,\\
\left(x-1,1-y\right), & \F_{H_1}(z)\notin R.
\end{cases}
\end{equation}

\begin{lem}\label{M_lem}
For any $z\in \Delta,$ $M(z)=\varphi_R\circ\psi^{-1}\circ\G_\Delta^{-1}(z)$.
\end{lem}
\begin{proof}
Suppose first that $z=(x,y)\in \Delta^+$.
Now, since $z\notin \Delta^-=\G(S)$, we have $\G^{-1}(z)\in \Omega\setminus S=\Delta$.  Hence $\G_\Delta^{-1}(z)=\G^{-1}(z)$.  Then
\[\varphi_R\circ\psi^{-1}\circ\G_\Delta^{-1}(z)=\varphi_R\circ\varphi_{H_1}^{-1}(z)=\varphi_R\left(x,\frac1{1+y}\right).\]
Notice that 
\[\F_{H_1}\left(x,\frac1{1+y}\right)=\psi^{-1}\circ \varphi_{H_1}\left(x,\frac1{1+y}\right)=\psi^{-1}(z)\in \psi^{-1}(\Delta)=R,\]
so by \eqref{S-expn_phi_R}, $\varphi_R\left(x,1/(1+y)\right)=z$.  Thus, for $z\in\Delta^+$, $\varphi_R\circ\psi^{-1}\circ\G_\Delta^{-1}(z)=z=M(z).$

Next, suppose that $z\in \Delta^-$.  Then $z\in \G(S)$, so $\G^{-1}(z)\in S=\Omega\setminus \Delta$ and $\G_\Delta^{-1}(z)=\G^{-2}(z)$.  Moreover, since $\G^{-1}(z)\in S\subset V_1$, we have $\G^{-1}(z)=(1/(x+1),1/y-1)$.  With these observations, we find
\[\varphi_R\circ\psi^{-1}\circ\G_\Delta^{-1}(z)=\varphi_R\circ\varphi_{H_1}^{-1}\circ\G^{-1}(z)=\varphi_R\circ\varphi_{H_1}^{-1}\left(\frac1{x+1},\frac1y-1\right)=\varphi_R\left(\frac1{x+1},y\right).\] 
We claim that $\F_{H_1}(1/(x+1),y)\notin R$.  This is equivalent to $\psi^{-1}\circ \varphi_{H_1}(1/(x+1),y)\notin\psi^{-1}(\Delta)$, or $\varphi_{H_1}(1/(x+1),y)\in S$.  But $\varphi_{H_1}(1/(x+1),y)=(1/(x+1),1/y-1)=\G^{-1}(z)\in S$ by assumption, so the claim holds.  Thus, from \eqref{S-expn_phi_R}, we have $\varphi_R(1/(x+1),y)=(-x/(x+1),1-y)$ and $\varphi_R\circ\psi^{-1}\circ\G_\Delta^{-1}(z)=(-x/(x+1),1-y)=M(z)$ for $z\in \Delta^-$.
\end{proof}

\begin{prop}\label{S-expn_two_shift}
With $R=\psi^{-1}(\Delta),$
\[(\Omega_R,\mathcal{B},\bar\nu_R,\tau_R)=(\Gamma_S,\mathcal{B},\rho,\tau).\]
\end{prop}
\begin{proof}
By Lemma \ref{M_lem},
\[\Omega_R=\varphi_R(R)=\varphi_R\circ \psi^{-1}(\Delta)=\varphi_R\circ \psi^{-1}\circ\G_\Delta^{-1}(\Delta)=M(\Delta)=\Gamma_S.\]
Moreover, 
\begin{align*}\psi\circ \F_R(z)=&\begin{cases}
\psi\circ \F_{H_1}(z), & \F_{H_1}(z)\in R,\\
\psi\circ \F_{H_1}^2(z), & \F_{H_1}(z)\notin R,\\
\end{cases}\\
=&\begin{cases}
\varphi_{H_1}(z), & \varphi_{H_1}(z)\in \Delta,\\
\G\circ\varphi_{H_1}(z), & \varphi_{H_1}(z)\notin \Delta,\\
\end{cases}\\
=&\begin{cases}
\G\circ\psi(z), & \G\circ\psi(z)\in \Delta,\\
\G^2\circ\psi(z), & \G\circ\psi(z)\notin \Delta,\\
\end{cases}\\
=&\G_{\Delta}\circ \psi(z),
\end{align*}
so
\[\tau_R=\varphi_R\circ\F_R\circ\varphi_R^{-1}=\varphi_R\circ\psi^{-1}\circ \G_\Delta\circ \psi\circ\varphi_R^{-1}=M\circ \G_\Delta\circ M^{-1}=\tau.\]
Lastly, $\bar\nu_R=\rho$ since these are both probability measures on $\Omega_R=\Gamma_S$ with densities of the form $C(1+XY)^{-2}$, where $C$ is a normalising constant.
\end{proof}

\subsection{Nakada's $\alpha$-continued fractions, revisited}\label{Nakada's alpha-continued fractions, revisited}

Recall Nakada's parameterised family of $\alpha$-{\sc cf} maps from \S\ref{Nakada's alpha-continued fractions}, which are defined for all $0\le\alpha\le 1$.  Moreover, recall from the end of \S\ref{S-expansions} that the natural extensions of the $\alpha$-{\sc cf}s are realised as $S$-expansion systems, but only for $\alpha\ge 1/2$.  Since, by \S\ref{S-expansions, revisited}, $S$-expansions are realised as contracted Farey expansions, so are Nakada's $\alpha$-{\sc cf}s for $\alpha\ge 1/2$.  In this subsection we extend this fact to $\alpha>0$, giving a new description of a planar natural extension of $([\alpha-1,\alpha],\mathcal{B},\rho_\alpha,G_\alpha)$ as an explicit induced transformation $(R,\mathcal{B},\bar\mu_R,\F_R)$ of Ito's natural extension of the Farey tent map (Theorem \ref{alpha_cf_ne} below; cf.~\cite{KSS2012}).

\begin{remark}
One finds that $G_\alpha([\alpha-1,\alpha])=[\alpha-1,\alpha)$, so $([\alpha-1,\alpha],\mathcal{B},\rho_\alpha,G_\alpha)$ is isomorphic to the restriction of this system to $[\alpha-1,\alpha)$, which we denote by $([\alpha-1,\alpha),\mathcal{B},\rho_\alpha,G_\alpha)$.  The endpoint $\alpha$ was included in the domain in \S\ref{Nakada's alpha-continued fractions} so that we could speak of matching, which depends on the $G_\alpha$-orbits of $\alpha$ and $\alpha-1$.  However, it shall be more convenient in this subsection to consider the isomorphic system $([\alpha-1,\alpha),\mathcal{B},\rho_\alpha,G_\alpha)$.  
\end{remark}

The domain $R$ will be constructed in two steps: first, we define a subset $A\subset H_1$ via an integer-valued map $k$ on $H_1$; second, $R$ is defined by `pushing' part of $A$ down into $\Omega\setminus H_1$ with the map $\F$.  Fix $\alpha\in(0,1]$ and define $k:H_1\to\mathbb{N}\cup\{\infty\}$ by
\[k(z):=\inf\{j>0\ |\ \F_{H_1}^{-j}(z)\in [0,\alpha)\times[1/2,1]\},\qquad z\in H_1,\]
and let
\[A:=\{z\in H_1\ |\ k(z)\ \text{is odd}\};\]
see Figure \ref{alphaNE_fig}.  Recall the definition of hitting times $N_R$ from (\ref{hit_time}).  The restriction of $N_R$ to $R$---also denoted $N_R$---is called the \emph{return time} to $R$.  We wish to determine the return times $N_A$ under $\F$.  For this, we use the following:

\begin{lem}\label{k(F(z))}
For any $z=(x,y)\in H_1$,
\[k(\F_{H_1}(z))=\begin{cases}
1, & x<\alpha,\\
k(z)+1, & x\ge \alpha.
\end{cases}\]
\end{lem}
\begin{proof}
First, notice that $\F_{H_1}^{-1}(\F_{H_1}(z))=z$ belongs to $[0,\alpha)\times[1/2,1]$ if and only if $x<\alpha$.  Thus, if $x<\alpha$, then $k(\F_{H_1}(z))=1$.  If $x\ge \alpha$, then $k(z)$ is the infimum of powers $j>0$ for which 
\[\F_{H_1}^{-(j+1)}(\F_{H_1}(z))=\F_{H_1}^{-j}(z)\in[0,\alpha)\times[1/2,1].\]
Hence $k(\F_{H_1}(z))=k(z)+1$.
\end{proof}

\begin{lem}\label{N_A}
The return times $N_A:A\to\mathbb{N}$ under $\F$ are
\[N_A(z)=\begin{cases}
a_1, & x<\alpha,\\
a_1+a_2, & x\ge \alpha,
\end{cases}\]
where $z=(x,y)\in A$ with $x=[0;a_1,a_2,\dots]$.
\end{lem}
\begin{proof}
First, suppose that $x<\alpha$, and notice that for all $0<j<a_1$, $\F^j(z)\in H_{j+1}\neq H_1$ implies $\F^j(z)\notin A\subset H_1$.  On the other hand, by Lemma \ref{k(F(z))}, $k(\F_{H_1}(z))=1$ is odd, so $\F_{H_1}(z)\in A$.  Since $\F_{H_1}(z)=\F^{a_1}(z)$, we have $N_A(z)=a_1$.  

Next, suppose that $x\ge \alpha$.  As above, $\F^j(z)\notin H_1$ for all $0<j<a_1+a_2$ with $j\neq a_1$, so $\F^j(z)\notin A$ for such $j$.  Moreover, $z\in A$ implies that $k(z)$ is odd, and thus by Lemma \ref{k(F(z))}, $k(\F_{H_1}(z))=k(z)+1$ is even.  Hence $\F^{a_1}(z)=\F_{H_1}(z)\notin A$.  Write $z'=(x',y'):=\F_{H_1}(z)$.  Again by Lemma \ref{k(F(z))}, 
\[k(\F_{H_1}(z'))=\begin{cases}
1, & x'<\alpha,\\
k(z')+1, & x'\ge \alpha.\\
\end{cases}\]
But $k(z')=k(\F_{H_1}(z))$ is even, so in either case $k(\F_{H_1}(z'))$ is odd.  Hence $\F^{a_1+a_2}(z)=\F_{H_1}^2(z)=\F_{H_1}(z')\in A$, and $N_A(z)=a_1+a_2$.  
\end{proof}

We now define the subregion $R\subset \Omega$ in terms of the set $A\subset H_1$.  For each integer $a>1$, let
\[A_a:=A\cap V_a\cap ([\alpha,1]\times[1/2,1])\]
be the set of points $z=(x,y)\in A$ for which $x=[0;a_1,a_2,\dots]$ with $a_1=a$ and $x\ge \alpha$.  Next, define
\begin{equation}\label{R_alpha_cfs}
R:=A\cup\bigcup_{a>1}\bigcup_{\lambda=1}^{a-1}\F^\lambda(A_a)
\end{equation}
as the region $A\subset H_1$ together with each $A_a$ `pushed down' into $\Omega\backslash H_1$ under $\F$ a maximal number of times; see Figure \ref{alphaNE_fig}.  Notice that if $\alpha>1/2$, then $A_a=\varnothing$ for $a>1$ and hence $R=A$.  

\begin{remark}
In Figure \ref{alphaNE_fig}, the region $A$ consists of rectangles extending from $x=0$ to $x=1$, and the region $R$ consists of $A$ together with rectangles extending from $x=F(1/4)=1/3$ to $x=1$ and $x=F^2(1/4)=1/2$ to $x=1$.  These `full' rectangles are due to the fact that $\alpha=1/4$ is of the form $\alpha=1/n$ for some integer $n\ge 1$; see also \cite{LM2008}, where the natural extension of the $\alpha$-{\sc cf} maps are constructed for such $\alpha$.  For general $\alpha>0$, one can show that $A$ consists of rectangles extending from various $x=x_0\in[0,1)$ to $x=1$.  
\end{remark}

\begin{figure}[t]
\centering
\includegraphics[width=.37\textwidth]{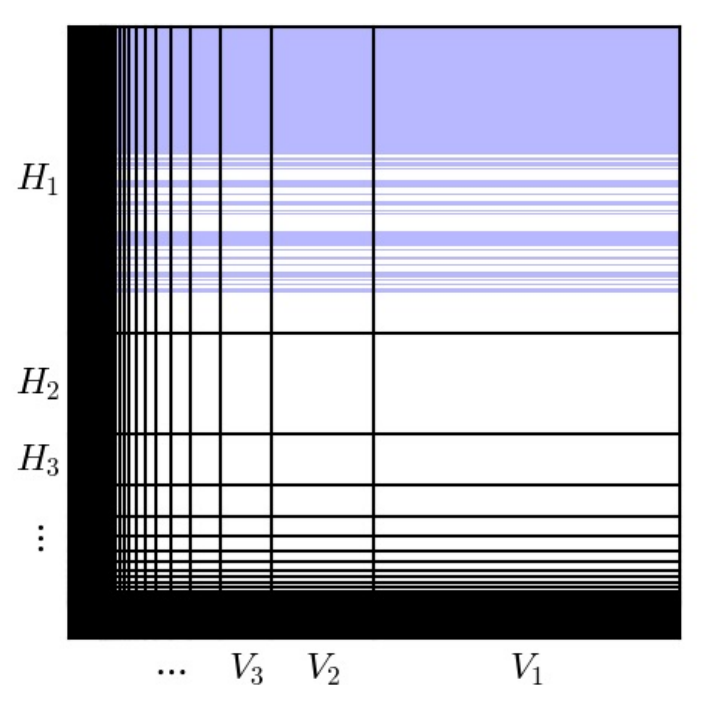}
\hspace{2pt}
\includegraphics[width=.37\textwidth]{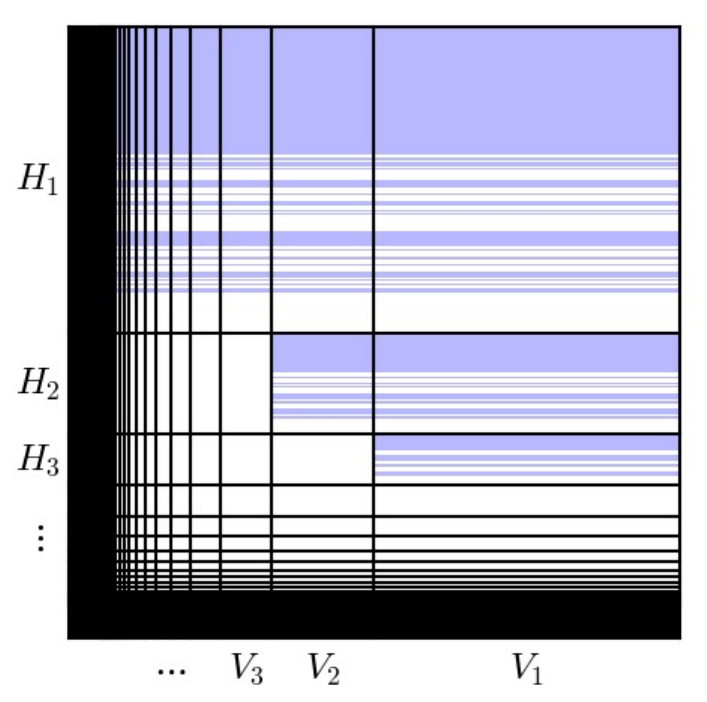}
\includegraphics[width=.7\textwidth]{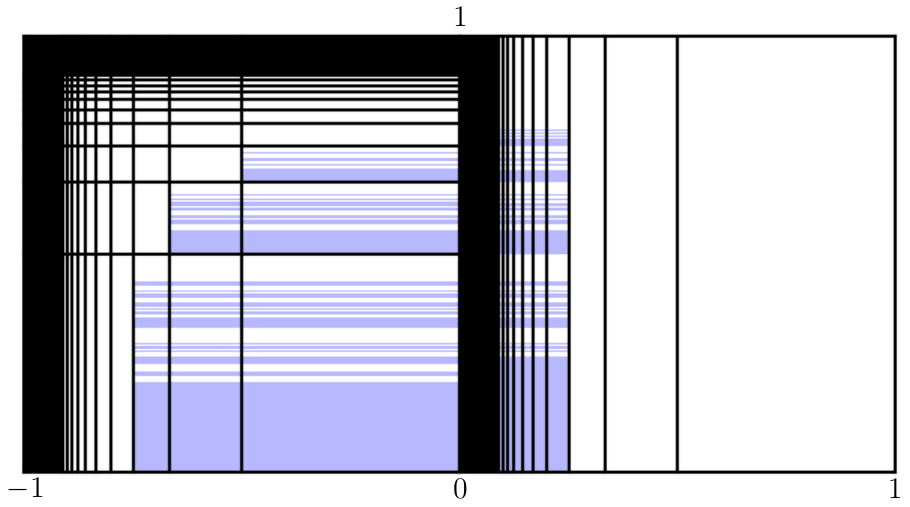}
\caption{Approximations of the regions $A$ (top-left), $R$ (top-right), and $\Omega_R$ (bottom) for $\alpha=1/4$.}
\label{alphaNE_fig}
\end{figure}

Lemma \ref{N_A} and the definition of $R$ give the following:

\begin{cor}\label{ret_times}
The return times $N_R:R\to\mathbb{N}$ under $\F$ are given by $N_R=N_A$ if $\alpha>1/2$ and 
\[N_R(z)=\begin{cases}
a_1, & x<\alpha,\\
1, & \alpha\le x\le 1/2,\\
a_2+1, & 1/2<x,
\end{cases}\]
if $\alpha\le 1/2$, where $z=(x,y)\in A$ with $x=[0;a_1,a_2,\dots]$.
\end{cor}

From this and Equation (\ref{A_R(z)}), we find that if $\alpha>1/2$, then
\begin{equation}\label{A_R,alpha>1/2}
A_R(z)=\begin{pmatrix}
u_R(z) & t_R(z)\\
s_R(z) & r_R(z)
\end{pmatrix}
=\left\{\begin{array}{lr}
A_0^{a_1-1}A_1, & x<\alpha\\
A_1A_0^{a_2-1}A_1, & x\ge\alpha
\end{array}\right\}
=\begin{cases}
\begin{pmatrix}
0 & 1\\
1 & a_1
\end{pmatrix}, & x<\alpha,\\
\begin{pmatrix}
1 & a_2\\
1 & a_2+1
\end{pmatrix}, & x\ge\alpha,
\end{cases}
\end{equation}
while if $\alpha\le 1/2$,
\begin{equation}\label{A_R,alpha<=1/2}
A_R(z)=\begin{pmatrix}
u_R(z) & t_R(z)\\
s_R(z) & r_R(z)
\end{pmatrix}
=\left\{\begin{array}{lr}
A_0^{a_1-1}A_1, & x<\alpha\\
A_0, & \alpha\le x\le 1/2\\
A_1A_0^{a_2-1}A_1, & 1/2<x
\end{array}\right\}
=\begin{cases}
\begin{pmatrix}
0 & 1\\
1 & a_1
\end{pmatrix}, & x<\alpha,\\
\begin{pmatrix}
1 & 0\\
1 & 1
\end{pmatrix}, & \alpha\le x\le 1/2,\\
\begin{pmatrix}
1 & a_2\\
1 & a_2+1
\end{pmatrix}, & 1/2<x.
\end{cases}
\end{equation}
Notice, in particular, that $s_R(z)=1$ for all $z$, and $R$ satisfies the assumptions of \S\ref{A two-sided shift for contracted Farey expansions}.  Moreover, 
\begin{equation}\label{alpha_cfs-u_R}
u_R(z)=\begin{cases}
0, & x<\alpha,\\
1 & x\ge \alpha,
\end{cases}
\end{equation}
so the map $\varphi_R:R\to\mathbb{R}^2$ from \eqref{varphi_R} is given by 
\begin{equation}\label{varphi_R-alpha}
\varphi_R(z)=\begin{cases}
\left(x,\frac{1-y}{y}\right), & x<\alpha,\\
\left(x-1,1-y\right), & x\ge \alpha.
\end{cases}
\end{equation}
The region $\Omega_R=\varphi_R(R)$ is shown in Figure \ref{alphaNE_fig}.

Before proving that $(R,\mathcal{B},\bar\mu_R,\F_R)$ is the natural extension of $([\alpha-1,\alpha),\mathcal{B},\rho_\alpha,G_\alpha)$, we determine the values of $\alpha_R(z)$ and $\beta_R(z)$ defined in (\ref{alpha_R}) and (\ref{beta_R}).

\begin{lem}\label{alpha_R,beta_R,alpha-cf}
Let $z=(x,y)\in R$ with $x\neq 0,1$.  Then 
\[\alpha_R(z)=\begin{cases}
1, & x<\alpha,\\
-1, & x\ge \alpha,
\end{cases}
\qquad \text{and} \qquad
\beta_R(z)=\begin{cases}
\left\lfloor\frac1x+1-\alpha \right\rfloor, & x<\alpha,\\
\left\lfloor\frac1{1-x}+1-\alpha \right\rfloor, & x\ge\alpha.\\
\end{cases}\]
\end{lem}
\begin{proof}
Let $z=(x,y)\in R$ with $x=[0;a_1,a_2,\dots]$, and set $z'=(x',y')=\F_R(z)$.  From (\ref{alpha_R}), the fact that $s_R(z)=1$ for all $z$, and (\ref{A_R,alpha>1/2}) and (\ref{A_R,alpha<=1/2}), we have
\[\alpha_R(z)=-\det(A_R(z))=\begin{cases}
1, & x<\alpha,\\
-1, & x\ge \alpha,
\end{cases}\]
as claimed.  

Next, from (\ref{beta_R}) and the fact that $s_R(z)=1$ for all $z$, we have $\beta_R(z)=r_R(z)+u_R(z')$.  If $\alpha>1/2$, then from (\ref{A_R,alpha>1/2}) and \eqref{alpha_cfs-u_R}, we find that
\begin{equation}\label{beta_R,alpha>1/2}
\beta_R(z)=\left\{\begin{array}{lr}
a_1+u_R(z'), & x<\alpha\\
a_2+1+u_R(z'), & x\ge \alpha
\end{array}\right\}
=\begin{cases}
a_1, & x<\alpha\ \text{and}\ x'<\alpha,\\
a_1+1, & x<\alpha\ \text{and}\ x'\ge\alpha,\\
a_2+1, & x\ge \alpha\ \text{and}\ x'<\alpha,\\
a_2+2, & x\ge \alpha\ \text{and}\ x'\ge\alpha.
\end{cases}
\end{equation}
Now suppose $\alpha\le 1/2$.  Notice that if $\alpha\le x\le 1/2$, then by (\ref{F_R2}) and (\ref{A_R,alpha<=1/2}), 
\begin{equation}\label{x',alpha<=x<=1/2}
x'=A_R^{-1}(z)\cdot x=\begin{pmatrix}1 &0\\ -1 & 1\end{pmatrix}\cdot x=\frac{x}{1-x}> x\ge \alpha
\end{equation}
implies $u_R(z')=1$.  Hence, again from (\ref{A_R,alpha<=1/2}) and \eqref{alpha_cfs-u_R},
\begin{equation}\label{beta_R,alpha<=1/2}
\beta_R(z)=\left\{\begin{array}{lr}
a_1+u_R(z'), & x<\alpha\\
1+u_R(z'), & \alpha\le x\le 1/2\\
a_2+1+u_R(z'), & 1/2<x
\end{array}\right\}
=\begin{cases}
a_1, & x<\alpha\ \text{and}\ x'<\alpha,\\
a_1+1, & x<\alpha\ \text{and}\ x'\ge\alpha,\\
2, & \alpha\le x\le 1/2,\\
a_2+1, & 1/2<x\ \text{and}\ x'<\alpha,\\
a_2+2, & 1/2<x\ \text{and}\ x'\ge\alpha.
\end{cases}
\end{equation}
The remainder of the proof consists of cases.  Throughout, we repeatedly use the two inequalities $\alpha\le 1+x'$ and $x'<1+\alpha$, which follow from $\alpha\in (0,1]$ and $x'\in[0,1]$.
\begin{enumerate}
\item[(i)] Suppose that $x<\alpha$.  We must show $\beta_R(z)=\left\lfloor\frac1x+1-\alpha \right\rfloor.$  By (\ref{F_R2}), (\ref{A_R,alpha>1/2}) and (\ref{A_R,alpha<=1/2}),
\[x'=A_R^{-1}(z)\cdot x=\begin{pmatrix}a_1 &-1\\ -1 & 0\end{pmatrix}\cdot x=\frac1x-a_1,\]
so
\[\frac1x+1-\alpha=x'+a_1+1-\alpha.\]
\begin{enumerate}
\item[(a)]  If $x'<\alpha$, then 
\[a_1\le x'+a_1+1-\alpha< a_1+1,\]
and by (\ref{beta_R,alpha>1/2}) and (\ref{beta_R,alpha<=1/2}), $\beta_R(z)=a_1=\left\lfloor\frac1x+1-\alpha \right\rfloor.$
\item[(b)]  If $x'\ge \alpha$, then 
\[a_1+1\le x'+a_1+1-\alpha< a_1+2,\]
so by (\ref{beta_R,alpha>1/2}) and (\ref{beta_R,alpha<=1/2}), $\beta_R(z)=a_1+1=\left\lfloor\frac1x+1-\alpha \right\rfloor.$
\end{enumerate}

\item[(ii)] Now suppose that $x\ge \alpha$.  We must show $\beta_R(z)=\left\lfloor \frac1{1-x}+1-\alpha\right\rfloor$.  
\begin{enumerate}
\item[(a)]  If $x\le 1/2$, then from the computation in (\ref{x',alpha<=x<=1/2}), $1+x'=1/(1-x)$.  Hence
\[\frac1{1-x}+1-\alpha=2+x'-\alpha,\]
and (using $x'>x$)
\[2\le 2+x-\alpha< 2+x'-\alpha<3.\]
By (\ref{beta_R,alpha<=1/2}), $\beta_R(z)=2=\left\lfloor \frac1{1-x}+1-\alpha\right\rfloor.$

\item[(b)] 
Now suppose that $x>1/2$.  By (\ref{F_R2}), (\ref{A_R,alpha>1/2}) and (\ref{A_R,alpha<=1/2}),
\[x'=A_R^{-1}(z)\cdot x=\begin{pmatrix}a_2+1 &-a_2\\ -1 & 1\end{pmatrix}\cdot x=\frac{x}{1-x}-a_2,\]
and
\[\frac{1}{1-x}+1-\alpha=\frac{x}{1-x}+2-\alpha=a_2+2+x'-\alpha.\]
\begin{enumerate}
\item[(1)]  If $x'<\alpha$, then 
\[a_2+1\le a_2+2+x'-\alpha<a_2+2,\]
and by (\ref{beta_R,alpha>1/2}) and (\ref{beta_R,alpha<=1/2}), $\beta_R(z)=a_2+1=\left\lfloor \frac1{1-x}+1-\alpha\right\rfloor.$

\item[(2)]  Lastly, if $x'\ge \alpha$, then 
\[a_2+2\le a_2+2+x'-\alpha< a_2+3,\]
so by (\ref{beta_R,alpha>1/2}) and (\ref{beta_R,alpha<=1/2}), $\beta_R(z)=a_2+2=\left\lfloor \frac1{1-x}+1-\alpha\right\rfloor.$
\end{enumerate}
\end{enumerate}
\end{enumerate}
\end{proof}

We are now in a position to prove:

\begin{thm}\label{alpha_cf_ne}
The induced system $(R,\mathcal{B},\bar\mu_R,\F_R)$ is the natural extension of $([\alpha-1,\alpha),\mathcal{B},\rho_\alpha,G_\alpha)$.
\end{thm}
\begin{proof}
When $\alpha=1$, then $R=[0,1)\times [1/2,1)$, and $(R,\mathcal{B},\bar\mu_R,\F_R)$ is isomorphic to $(H_1,\mathcal{B},\bar\mu_{H_1},\F_{H_1})$.  The result follows from Theorem \ref{H_1&Gauss} and the fact that for $\alpha=1$, $([\alpha-1,\alpha),\mathcal{B},\rho_\alpha,G_\alpha)$ is (isomorphic to) $([0,1],\mathcal{B},\nu_G,G)$. 

Now suppose $\alpha\in (0,1)$.  Since $(R,\mathcal{B},\bar\mu_R,\F_R)$ and the system $(\Omega_R,\mathcal{B},\bar\nu_R,\tau_R)$ from \S\ref{A two-sided shift for contracted Farey expansions} are isomorphic, it suffices to show that the latter system is the natural extension of $([\alpha-1,\alpha),\mathcal{B},\rho_\alpha,G_\alpha)$.  Throughout, we shall consider the restrictions of $(R,\mathcal{B},\bar\mu_R,\F_R)$ and $(\Omega_R,\mathcal{B},\bar\nu_R,\tau_R)$ to the full-measure subsets on which $\F_R^n$ and $\tau_R^n$ are defined for all $n\in\mathbb{Z}$, and such that for any $(x,y)\in R$ and any $(X,Y)\in \Omega_R$, both $x$ and $X$ are irrational; see the discussion preceding Proposition \ref{gen_set_of_B}.  Since $G_\alpha([\alpha-1,\alpha)\backslash\mathbb{Q})\subset [\alpha-1,\alpha)\backslash\mathbb{Q}$, we shall in fact show that $(\Omega_R,\mathcal{B},\bar\nu_R,\tau_R)$ is the natural extension of $([\alpha-1,\alpha),\mathcal{B},\rho_\alpha,G_\alpha)$ restricted to $[\alpha-1,\alpha)\backslash\mathbb{Q}$, which we denote $([\alpha-1,\alpha)\backslash\mathbb{Q},\mathcal{B},\rho_\alpha,G_\alpha)$

To distinguish the Borel $\sigma$-algebras restricted to $\Omega_R$ and $[\alpha-1,\alpha)\backslash\mathbb{Q}$, we shall denote these by $\mathcal{C}$ and $\mathcal{D}$, respectively.  Notice that $([\alpha-1,\alpha)\backslash\mathbb{Q},\mathcal{D},\rho_\alpha,G_\alpha)$ is non-invertible and $(\Omega_R,\mathcal{C},\bar\nu_R,\tau_R)$ is invertible.  We will show (i) that $([\alpha-1,\alpha)\backslash\mathbb{Q},\mathcal{D},\rho_\alpha,G_\alpha)$ is a factor of $(\Omega_R,\mathcal{C},\bar\nu_R,\tau_R)$ with factor map $\pi_X:\Omega_R\to[\alpha-1,\alpha)\backslash\mathbb{Q}$ being the projection onto the first coordinate, and (ii) that the factor map $\pi_X$ satisfies
\[\bigvee_{n=0}^\infty \tau_R^n\circ \pi_X^{-1}(\mathcal{D})=\mathcal{C},\]
where $\bigvee_{n=0}^\infty \tau_R^n\circ \pi_X^{-1}(\mathcal{D})$ is the smallest $\sigma$-algebra containing each $\sigma$-algebra $\tau_R^n\circ \pi_X^{-1}(\mathcal{D}),\ n\ge 0$.

\begin{enumerate}
\item[(i)]  We must show that $\pi_X:\Omega_R\to [\alpha-1,\alpha)\backslash \mathbb{Q}$ is measurable, surjective, and satisfies $\pi_X\circ \tau_R=G_\alpha\circ \pi_X$ and $\bar\nu_R\circ\pi_X^{-1}=\rho_\alpha$.  Certainly $\pi_X$ is measurable, since for any Borel set $A\in\mathcal{D}$, $\pi_X^{-1}(A)=(A\times [0,1])\cap \Omega_R\in \mathcal{C}$ is a Borel set in $\Omega_R$.  For surjectivity, suppose $\alpha$ has {\sc rcf}-expansion $\alpha=[0;\alpha_1,\alpha_2,\dots]$, and let $z=(x,y)\in H_1$ with $x=[0;a_1,a_2,\dots]\notin\mathbb{Q}$ and $y=[0;1,b,b,b,\dots]$ for some $b>\alpha_1$.  Then 
\[\F_{H_1}^{-1}(z)=([0;b,a_1,a_2,\dots],[0;1,b,b,\dots])\in [0,\alpha)\times[1/2,1],\] 
so $k(z)=1$ is odd and $z\in A$.  Similarly, $k(\F_{H_1}^{-n}(z))=1$ for all $n\ge 0$, so $\F_{H_1}^{-n}(z)\in A$ for all $n\ge 0$.  This---together with Corollary \ref{ret_times}---implies that $\F_{R}^{n}(z)\in R$ is defined for all $n\in\mathbb{Z}$.  Since $x\in[0,1]\backslash\mathbb{Q}$ was arbitrary, \eqref{varphi_R-alpha} gives $\pi_X(\Omega_R)=\pi_X(\varphi_R(R))=[\alpha-1,\alpha)\backslash\mathbb{Q}$, i.e., $\pi_X$ is surjective.  

Next, we show $\pi_X\circ \tau_R=G_\alpha\circ \pi_X$.  Let $(X,Y)=(X(z),Y(z))\in \Omega_R$, where $z=(x,y)\in R$, and notice from \eqref{varphi_R-alpha} that 
\[X=\begin{cases}
x, & x<\alpha,\\
x-1, & x\ge\alpha.
\end{cases}\]
Moreover, $x<\alpha$ if and only if $X>0$, and $x\ge \alpha$ if and only if $X<0$.  These observations, together with Theorem \ref{two_shift_map_and_measure} and Lemma \ref{alpha_R,beta_R,alpha-cf}, give
\begin{align*}
\pi_X\circ \tau_R(X,Y)=&\frac{\alpha_R(z)}{X}-\beta_R(z)\\
=&\begin{cases}
\frac1X-\left\lfloor\frac1x+1-\alpha\right\rfloor, & x<\alpha,\\
-\frac1X-\left\lfloor\frac1{1-x}+1-\alpha\right\rfloor, & x\ge \alpha,
\end{cases}\\
=&\begin{cases}
\frac1X-\left\lfloor\frac1X+1-\alpha\right\rfloor, & X>0,\\
-\frac1X-\left\lfloor-\frac1X+1-\alpha\right\rfloor, & X<0,
\end{cases}\\
=&\frac{1}{|X|}-\left\lfloor\frac1{|X|}+1-\alpha\right\rfloor\\
=&G_\alpha\circ\pi_X(X,Y)
\end{align*}
as desired.  Lastly, notice that for any Borel set $A\in\mathcal{D}$, $\tau_R$-invariance of $\bar\nu_R$ gives
\[\bar\nu_R\circ \pi_X^{-1}(G_\alpha^{-1}(A))=\bar\nu_R\circ \tau_R^{-1}(\pi_X^{-1}(A))=\bar\nu_R\circ \pi_X^{-1}(A),\]
so $\bar\nu_R\circ \pi_X^{-1}$ is an absolutely continuous, $G_\alpha$-invariant probability measure.  Uniqueness of $\rho_\alpha$ implies $\bar\nu_R\circ \pi_X^{-1}=\rho_\alpha$.  Thus $([\alpha-1,\alpha),\mathcal{D},\rho_\alpha,G_\alpha)$ is a factor of $(\Omega_R,\mathcal{C},\bar\nu_R,\tau_R)$.

\item[(ii)]  We now show that
\[\bigvee_{n=0}^\infty \tau_R^n\circ \pi_X^{-1}(\mathcal{D})=\mathcal{C}.\]
The forward inclusion follows from measurability of $\pi_X$ and $\tau_R^{-1}$, so it suffices to show the backward inclusion.  For this, it suffices to show that every element of a generating set of the Borel $\sigma$-algebra $\mathcal{C}$ on $\Omega_R$ can be written as $\tau_R^k\circ \pi_X^{-1}(D)$ for some $D\in \mathcal{D}$ and $k\ge 0$.  By Proposition \ref{gen_set_of_B}, $\mathcal{C}$ is generated by the sets
\[C=\Delta(0/1;\alpha_1/\beta_1,\alpha_2/\beta_2,\dots,\alpha_n/\beta_n)\times\Delta(0/1;1/\beta_0,\alpha_0/\beta_{-1},\alpha_{-1}/\beta_{-2},\dots,\alpha_{-(m-1)}/\beta_{-m})\]
containing all points $(X(z),Y(z))\in \Omega_R$ for which
\[\alpha_R(z_j^R)=\alpha_{j+1}\quad \text{and}\quad \beta_R(z_k^R)=\beta_{k+1}\]
for all $-m\le j\le n-1$ and $-m-1\le k\le n-1$. 

Let $D\in\mathcal{D}$ be the set of irrationals $X\in [\alpha-1,\alpha)$ for which 
\[\text{sgn}(G_\alpha^j(X))=\alpha_{j-m} \quad \text{and} \quad \left\lfloor\frac1{|G_\alpha^k(X)|}+1-\alpha\right\rfloor=\beta_{k-m}\]
for all $1\le j\le n+m$ and $0\le k\le n+m$.  Let $X\in [\alpha-1,\alpha)\setminus\mathbb{Q}$, $(X,Y)=(X(z),Y(z))\in\pi_X^{-1}(\{X\})$, and $z_k^R=\F_R^k(z)$ for all $k\in\mathbb{Z}$.  Using the fact that $G_\alpha\circ\pi_X=\pi_X\circ\tau_R$, Equations \eqref{XnRYnRunRsnR}, \eqref{zn=varphi_inv_XnYN} and Lemma \ref{alpha_R,beta_R,alpha-cf} give 
\[\text{sgn}(G_\alpha^k(X))=\alpha_R(z_k^R)\quad \text{and} \quad \left\lfloor\frac1{|G_\alpha^k(X)|}+1-\alpha\right\rfloor=\beta_R(z_k^R), \quad k\ge 0,\]
so $\pi_X^{-1}(D)$ is the set of points $(X(z),Y(z))\in\Omega_R$ such that 
\[\alpha_R(z_j^R)=a_{j-m}\quad \text{and} \quad \beta_R(z_k^R)=\beta_{k-m}\]
for all $1\le j\le n+m$ and $0\le k\le n+m$.  By Corollary \ref{two_shift_cor}, this is the set of points of the form
\begin{align*}
X=X(z)=&[0/1;\alpha_R(z_0^R)/\beta_R(z_0^R),\alpha_R(z_1^R)/\beta_R(z_1^R),\dots]\\
=&[0/1;\alpha_R(z_0^R)/\beta_{-m},\alpha_{-(m-1)}/\beta_{-(m-1)},\dots,\alpha_{n}/\beta_n,\alpha_R(z_{n+m+1}^R)/\beta_R(z_{n+m+1}^R),\dots]
\end{align*}
and
\[Y=Y(z)=[0/1;1/\beta_R(z_{-1}^R),\alpha_R(z_{-1}^R)/\beta_R(z_{-2}^R),\dots].\]
Since $(X_{m+1}^R,Y_{m+1}^R)=\tau_R^{m+1}(X,Y)$ is of the form
\[X_{m+1}^R=[0/1;\alpha_1/\beta_1,\dots,\alpha_n/\beta_n,\alpha_R(z_{n+m+1}^R)/\beta_R(z_{n+m+1}^R),\dots]\]
and
\[Y_{m+1}^R=[0/1;1/\beta_0,\alpha_0/\beta_{-1},\dots,\alpha_{-(m-1)}/\beta_{-m},\alpha_R(z_0^R)/\beta_R(z_{-1}^R),\dots],\]
we have
$\tau_R^{m+1}\circ\pi_X^{-1}(D)=C$.
\end{enumerate}
\end{proof}

\begin{remark}
Recall from the end of \S\ref{Nakada's alpha-continued fractions} that there are several open questions about Nakada's $\alpha$-{\sc cf}s, including explicit descriptions of the values of the entropy $h(G_\alpha)$ for $\alpha<g^2$, $g=(\sqrt{5}-1)/2$, and of the densities of the invariant measures $\rho_\alpha$ (\cite{KSS2012}).  It is also open to explicitly compute the so-called \emph{Legendre constant} for $\alpha<g^2$ (\cite{dJK2018,Nat11}).

Each of these questions may be answered with an understanding of the domain of the natural extension of $([\alpha-1,\alpha),\mathcal{B},\rho_\alpha,G_\alpha)$; see, e.g., Theorem \ref{entropy_thm} for the entropy.  To date, however, the description of this domain has proven to be unmanageable for these tasks.  Our new description of the natural extension $(R,\mathcal{B},\bar\mu_R,\F_R)$ could bring many of these questions within reach.  Indeed, by \eqref{R_alpha_cfs}, in order to understand $R$ is suffices to understand the set $A\subset H_1$.  We hope to return to these questions in subsequent work and suspect that matching (see \S\ref{Nakada's alpha-continued fractions} above) will play a crucial role in their resolution.
\end{remark}


\bibliography{bib/bib}
\bibliographystyle{acm}

\end{document}